\documentclass[a4paper,11pt]{article} 

\usepackage[utf8]{inputenc} 
\usepackage{amsfonts}
\usepackage{amssymb}
\usepackage{amsthm}
\usepackage{amsmath}
\usepackage{a4wide}
\usepackage{mathrsfs}
\usepackage{mathtools}
\usepackage{epsfig}
\usepackage{palatino}
\usepackage{pifont}

\usepackage{xfrac}

\usepackage{tikz}
\usetikzlibrary{trees,positioning,calc,fit,decorations.markings}
\tikzstyle{every node}=[anchor=west]
\tikzstyle{sec}=[shape=rectangle, rounded corners,draw=gray, dotted,
                 minimum width = 6cm,minimum height = 1.5cm,yshift=-1cm]
\tikzstyle{lat}=[xshift=7cm]
\tikzstyle{optional}=[dashed,fill=gray!15]
\tikzset{bigarr/.style={
decoration={markings,mark=at position 1 with {\arrow[scale=1.7]{>}}},
postaction={decorate}}}

\usepackage{bm}
\usepackage{fp,ifthen}
\usepackage{nicefrac}
\usetikzlibrary{decorations.pathreplacing}
\usepackage{float}
\usepackage{tcolorbox}
\usepackage{enumerate} 
\usepackage{enumitem} 
\usepackage{bbm}
\usepackage[italian,english]{babel}

\usepackage{nicematrix}

\usepackage{nccmath}

\usepackage{tocloft}
\let\OLDthebibliography\thebibliography
\renewcommand\thebibliography[1]{
  \OLDthebibliography{#1}
  \setlength{\parskip}{0pt}
  \setlength{\itemsep}{4pt plus 0.3ex}
}

\newcommand{\mres}{\mathbin{\vrule height 1.6ex depth 0pt width
0.13ex\vrule height 0.13ex depth 0pt width 1.3ex}}

\usepackage[a4paper,twoside,top=2.5cm, bottom=2cm, left=2.5cm, right=2.5cm]{geometry} 

\newcommand{\pdfgraphics}{\ifpdf\DeclareGraphicsExtensions{.pdf,.jpg}\else\fi}
\usepackage{graphicx}
\usepackage{makecell}

\usepackage{lipsum}
\usepackage{color}
\definecolor{hanblue}{rgb}{0.27, 0.42, 0.81}
\definecolor{red}{rgb}{1.0, 0.0, 0.0}
\usepackage[colorlinks, citecolor=blue,linkcolor=blue, urlcolor = blue]{hyperref}

\theoremstyle{plain}

\newtheorem{teo}{Theorem}[section]
\newtheorem{lemma}[teo]{Lemma}
\newtheorem{prop}[teo]{Proposition}
\newtheorem{cor}[teo]{Corollary}

\theoremstyle{definition}
\newtheorem{defn}[teo]{Definition}

\newtheorem{rem}[teo]{Remark}

\theoremstyle{remark}

\numberwithin{equation}{section}

\newcommand \eps{\ensuremath{\varepsilon}} 
\renewcommand{\epsilon}{\varepsilon}
\newcommand{\N}{\ensuremath{\mathbb{N}}}
\newcommand{\Z}{\ensuremath{\mathbb Z}}
\newcommand{\R}{\ensuremath{\mathbb R}}

\newcommand{\supp}{\textnormal{spt}\hspace{1pt}}
\newcommand{\Cube}{\ensuremath{\mathcal{Q}}}
\newcommand{\link}{\operatorname{link}}

\newcommand{\diam}{\operatorname{diam}}
\newcommand{\osc}{\operatorname{osc}}

\newcommand{\dist}{\operatorname{dist}}

\usepackage{esint}

\newcommand{\U}{\ensuremath{\mathcal{U}}}

\newcommand{\sing}{\textnormal{sing}}
\newcommand{\dive}{\text{div}}

\newcommand{\s}{\hspace{7pt}}
\newcommand{\sms}{\hspace{4.5pt}}
\newcommand{\vsp}{\vspace{3pt}}
\newcommand{\Sp}{\mathbb{S}}
\newcommand{\ep}{\varepsilon}
\newcommand{\E}{\mathcal{E}}

\newcommand{\M}{\mathbb{M}}

\newcommand{\K}{\mathcal{K}}
\newcommand{\D}{\mathcal{D}}
\newcommand{\B}{\mathcal{B}}
\newcommand{\F}{\mathbf{F}}
\newcommand{\Leb}{\mathscr{L}}

\newcommand{\lcurr}{[\hspace{-1.3pt}[}
\newcommand{\rcurr}{]\hspace{-1.3pt}]}

\DeclarePairedDelimiter{\abs}{\lvert}{\rvert}

\makeatletter
\renewcommand*\env@matrix[1][*\c@MaxMatrixCols c]{%
  \hskip -\arraycolsep
  \let\@ifnextchar\new@ifnextchar
  \array{#1}}
\makeatother

\usepackage{titlesec}

\titleformat{\subsubsection}[runin]
  {\normalfont\normalsize\bfseries}
  {\thesubsubsection}{0.5em}{}[.]

\titlespacing*{\subsubsection}{0pt}{.8em}{0.8em}

\titlespacing*{\paragraph}{0pt}{.8em}{0.8em}

\usepackage{url}
\usepackage{authblk}

\makeatletter
\newcommand{\addressa}[1]{\gdef\@addressa{#1}}
\newcommand{\emaila}[1]{\gdef\@emaila{\url{#1}}}

\newcommand{\addressb}[1]{\gdef\@addressb{#1}}
\newcommand{\emailb}[1]{\gdef\@emailb{\url{#1}}}

\newcommand{\addressc}[1]{\gdef\@addressc{#1}}
\newcommand{\emailc}[1]{\gdef\@emailc{\url{#1}}}

\newcommand{\@endstuff}{\par\vspace{\baselineskip}\noindent
\begin{tabular}{@{}l}\scshape\@addressa\\\textit{E-mail address:} \@emaila\end{tabular} 

\vspace{12pt} \noindent
\begin{tabular}{@{}l}\scshape\@addressb\\ \textit{E-mail address:} \@emailb\end{tabular}

\vspace{12pt} \noindent
\begin{tabular}{@{}l}\scshape\@addressc\\ \textit{E-mail address:} \@emailc\end{tabular}
}

\AtEndDocument{\@endstuff}
\makeatother

\begin{document}

\pdfgraphics 

\title{Another look at a notion of fractional mass in codimension~two}


\author{Michele Caselli, Mattia Freguglia and Nicola Picenni}


\addressa{Michele Caselli \\ 
University of Sydney -- Quadrangle A14, Camperdown, NSW 2006, Australia \\ Princeton University -- 304 Washington Rd, Princeton NJ 08540, US}
\emaila{mc3147@princeton.edu}

\addressb{Mattia Freguglia \\ Bocconi University --  Via Roentgen 1, 20136 Milan, Italy}
\emailb{mattia.freguglia@unibocconi.it}

\addressc{Nicola Picenni \\ University of Pisa  -- Largo Bruno Pontecorvo 5, 56127 Pisa, Italy }
\emailc{nicola.picenni@unipi.it}


\date{}

\maketitle

\vspace{-0.5cm}

\begin{abstract}

We study a notion of fractional $s$-mass for codimension-two currents on closed Riemannian manifolds, defined via energy minimization with a prescribed Jacobian constraint. We prove equi-coercivity and $\Gamma$-convergence, with respect to the flat topology, of the $s$-mass on general codimension-two currents.

 We also prove several additional results for fixed $s$. We establish improved regularity for $s$-harmonic maps that are minimizing among competitors with vanishing Jacobian and show that their singular set has Minkowski dimension at most $n-3$. Moreover, we show that the $s$-mass defined via weak linking, as recently introduced by the authors, agrees with the prescribed Jacobian formulation used here, clarifying the extent to which the $s$-mass depends, or ultimately does not depend, on the way singularities are prescribed.

\end{abstract}

\tableofcontents

\vspace{5ex}



\vsp


\section{Introduction}

Nonlocal variational models provide a flexible way to approximate local
geometric quantities by energies that retain information at all length scales. After the seminal work of Bourgain, Brezis, and Mironescu \cite{BBM01}, many different nonlocal approximations of Sobolev-type seminorms have been investigated in recent years (see \cite{Ponce-calcvar,nguyen-gamma-conv,agmp,nonlocal-book,ARBDN,BSVSY,lahti,gp-gamma-liminf,nguyen-cras,gennaioli-stefani} and many others). Some variants have also been developed in more geometric settings, to provide nonlocal approximations for classical geometric functionals \cite{Cesaroni_et_al, Mattia-Willmore-limsup, Sav-Val-Gamma, 2023-Solci, ABSS}.

When applied to characteristic functions, this point of view naturally leads to a
nonlocal characterization of sets of finite perimeter and to the convergence of
the fractional perimeter, introduced in \cite{CRS}, to the classical perimeter. Beyond pointwise convergence \cite{Davila-BBM}, the convergence of the fractional perimeter to the classical one holds at a deep geometric level. For example, the convergence holds in the sense of $\Gamma$-convergence \cite{Gammaconv}. This implies, in particular, that minimizers of the fractional perimeter converge suitably to minimizers of the perimeter. 

More recently, significant progress \cite{CDSV, Enric} has been made in the convergence of general critical points of the fractional perimeter to critical points of the area (i.e. to minimal surfaces). For stable or finite Morse index critical points of the $s$-perimeter, the support converges suitably (in low dimension) to classical minimal surfaces as $s \to 1$. These results show that
fractional perimeters are not merely pointwise approximations of perimeter, but
provide a powerful variational framework for studying classical questions about
minimal hypersurfaces.

The effectiveness of this codimension-one theory suggests a broader program: developing nonlocal energies that approximate higher-codimension area with the same properties that made the codimension-one theory effective. Several authors have proposed extensions of the notion of fractional perimeter in $\R^n$ to higher codimension \cite{PonsiglioneMass, CorneliaSeguin, FMfvar}, or even to surfaces with boundary \cite{Seguin, SeguinBdry}, and have proved their pointwise convergence to the classical area as $s \to 1$. 

Nevertheless, all these notions are based in some way on the additive structure of $\R^n$ (translation invariance of the kernel or integral-geometric formulas). Given the program outlined above, it is crucial to work with an intrinsic notion that applies also to closed Riemannian manifolds, both for geometric applications and for a sound variational framework.

A first necessary step in such a program is to establish the same variational foundations available in codimension one: compactness, $\Gamma$-convergence, and a robust intrinsic definition on nonsmooth geometric objects. 

\paragraph{The fractional $s$-mass.} Given a closed Riemannian manifold $M$ and $s \in (0,1)$, one defines (see \cite{SerraSurv, codim2-frac, BadranSing}) the fractional $s$-mass of a surface $\Sigma \subset M$ of codimension $k \in \{1,\dots,n\}$ as
    \begin{equation*}
      \M_{n-k,s}(\Sigma) := \min \Big\{  [u]^2_{H^{\frac{k-1+s}{2}}(M)} : u \in  H^{\frac{k-1+s}{2}}(M; \Sp^{k-1}) \ \text{with} \ {\star Ju} = \omega_{k-1} \Sigma \Big\}. 
    \end{equation*}
Here $\Sigma$ is a smooth $(n-k)$-boundary in $M$, ${\star J u}$ denotes the Jacobian $(n-k)$-current associated to the map $u$, $\omega_{k-1}:=\mathcal{H}^{k-1}(\Sp^{k-1})$, and the minimized quantity is the $H^{\frac{k-1+s}{2}}(M)$-seminorm. For $k=1$ and $\Sigma =\partial E$, this definition recovers the notion of fractional $s$-perimeter since the only competitor in the minimum problem is the characteristic function of $E$.

The definition of $\mathbb{M}_{n-k,s}$ was first proposed, for smooth boundaries in a slightly different form, by Serra~\cite{SerraSurv} as an extension of the fractional perimeter to higher codimension, preserving the property that its critical points can be interpreted as singular sets of stationary fractional harmonic maps into spheres.  

In \cite{codim2-frac}, the authors recently proved that for closed, oriented, \emph{smooth} $(n-2)$-dimensional boundaries with integer multiplicity, $\M_{n-2,s}$ suitably rescaled $\Gamma$-converges as $s \to 1$ to a multiple of the classical $(n-2)$-dimensional Hausdorff measure, thus identifying $\M_{n-2,s}$ as an effective nonlocal approximation of codimension-two area. 

However, in \cite{codim2-frac} the variational framework of this approximation remained incomplete in two essential ways: the corresponding compactness/equi-coercivity was left open, and the analysis was restricted to \emph{smooth} codimension-two boundaries. In this work, we address these two open points and extend the analysis further. Moreover, we develop the theory in the more general setting of closed Riemannian manifolds.

\subsection{Main results} 

Since the rest of the work deals only with the codimension-two case, in what follows we will denote $\M_{n-2,s}$ simply by $\M_{s}$. Let $n\ge 2$ and $M$ be a smooth, closed, connected, oriented $n$-dimensional Riemannian manifold. In this work, we say that $\Sigma$ is an \emph{admissible boundary} if
\begin{equation}\label{eq:sigma class def}
\parbox{0.8\linewidth}{\centering
$\Sigma=\partial T $, where $T$ is an integer rectifiable $(n-1)$-current of finite mass. 
}
\end{equation}
In particular, $\Sigma$ is an $(n-2)$-dimensional current of finite flat norm, though not necessarily of finite mass. If $\Sigma$ also has finite mass, then $\Sigma$ is an integral $(n-2)$-boundary. Among other things, throughout the paper we identify if and when finite mass of $\Sigma$ is recovered from fractional information; see Corollary \ref{cor: bdd rescled mass -> integral} and Section \ref{sbs: domain of the s-mass}.

\subsubsection{$\Gamma$-convergence on currents}

A main goal of this work is to extend the $\Gamma$-convergence result of \cite{codim2-frac} for $(1-s)^2 \M_{s}$ to the most natural geometric class of admissible boundaries \eqref{eq:sigma class def}, and to prove the corresponding equi-coercivity property.

\begin{teo}[Equi-coercivity and $\Gamma$-convergence]\label{thm: new main asymptotics}
    As $s \to 1$, the rescaled functionals $(1-s)^2 \M_{s}$ are equi-coercive on the class of admissible boundaries with respect to the flat topology of $M$, and 
    \begin{equation*}
      \Gamma - \lim_{s \to 1} (1-s)^2 \M_{s}(\Sigma) = \frac{2\pi\omega_{n-1}}{n} \M(\Sigma), 
   \end{equation*}
where $\M(\Sigma)$ denotes the mass of $\Sigma$, and the $\Gamma$-limit is intended with respect to the flat topology. That is
\begin{itemize}
    \item[$(i)$] If $s \to 1$ is a (discrete) sequence and $(1-s)^2 \M_{s}(\Sigma_s)$ is uniformly bounded, then $\{\Sigma_s\}$ is relatively compact in the flat topology of $M$ and every limit point $\Sigma$ has finite mass.
    \item[$(ii)$] Whenever $s \to 1$ and $\Sigma_s \to \Sigma$ in the flat topology of $M$, there holds
    \[
        \liminf_{s \to 1} \hspace{0.03cm} (1-s)^2 \M_{s}(\Sigma_s)
        \geq \frac{2\pi\omega_{n-1}}{n}\M(\Sigma).
    \]
    Moreover, for every admissible boundary $\Sigma$ there exists a sequence of admissible boundaries $\Sigma_s$
    with $\Sigma_s \to \Sigma$ in the flat topology of $M$ such that
    \[
        \limsup_{s \to 1} \hspace{0.03cm} (1-s)^2 \M_{s}(\Sigma_s)
        \leq \frac{2\pi\omega_{n-1}}{n}\M(\Sigma).
    \]
\end{itemize}
\end{teo}


Therefore, the fractional $s$-mass provides a nonlocal approximation of the codimension-two area in full geometric generality. This gives a codimension-two analog of the classical $\Gamma$-convergence for the fractional perimeter \cite{Gammaconv, Ponce-calcvar}.

In codimension one, it is classical (see e.g. \cite[Corollary~5 and Remark~3]{BBM01}) that a set $E$ has finite perimeter if and only if $(1-s){\rm Per}_s(E)$ is uniformly bounded as $s \to 1$. As a direct application of $(i)$ in Theorem \ref{thm: new main asymptotics} to the constant sequence $\Sigma_s\equiv \Sigma$, we obtain one of these implications in codimension two for the $s$-mass.

\begin{cor}[BBM-type characterization]\label{cor: bdd rescled mass -> integral} Let $\Sigma$ be an admissible boundary such that
\begin{equation*}
    \limsup_{s \to 1} \hspace{0.03cm} (1-s)^2 \M_{s}(\Sigma) <+\infty. 
\end{equation*}
Then $\Sigma$ has finite mass and is an integral boundary.
\end{cor}

Observe that the converse of Corollary~\ref{cor: bdd rescled mass -> integral} does not hold: finite classical mass alone does not imply finite $s$-mass. See $(iii)$ in Section~\ref{sbs: domain of the s-mass} for an example of such a phenomenon.

\subsubsection{Relationship between $\Sigma$ and minimizers in the definition of $\M_s(\Sigma)$}

A second theme of this work is to investigate the relation between the prescribed \(\Sigma\) and the class of minimizers
\[
\mathcal{A}(\Sigma):=\Bigl\{u\in\mathfrak{F}_s^J(\Sigma):\
[u]^2_{H^{\frac{1+s}{2}}(M)}=\M_s(\Sigma)\Bigr\}
\]
in the definition of \(\M_s(\Sigma)\). Since every \(u\in\mathcal{A}(\Sigma)\) has Jacobian prescribed by \(\Sigma\), it is natural to ask whether the top-dimensional singularities (i.e. $(n-2)$-dimensional) of \(u\) are exactly \( \supp \Sigma\), or whether additional singularities of the same dimension (but zero degree)
can appear away from \(\Sigma\).

This question is closely related to the regularity theory of $\Sp^1$-valued fractional harmonic maps. For \(\sigma:=\tfrac{1+s}{2}\in(1/2,1)\) and \(n\geq 2\), it is known \cite{GlobRegFrac, Partialreg} that stationary
\(\sigma\)-harmonic maps are smooth outside a singular set of Minkowski dimension at most
\(n-2\): if \(n=2\), the singular set is locally finite, while if \(n\geq 3\), then
\(\dim_{\rm Min}\sing(u)\leq n-2\). Of course, the same holds for minimizing
\(\sigma\)-harmonic maps. Some results in this direction have also recently been established for intrinsic fractional $\sigma$-harmonic maps \cite{IntrinsicFHM}. 

However, in our case dimension \(n-2\) is precisely the
dimension of the prescribed \(\Sigma\). Thus the general theory alone does not
distinguish singularities generated by the Jacobian constraint from possible degree-zero
singularities away from \(\Sigma\). Our next result shows that, for \(s\) sufficiently close to \(1\), minimizers in \(\mathcal{A}(\Sigma)\) do not have additional top-dimensional singularities away
from \(\Sigma\). 

\begin{teo}\label{thm: improved W1p p<3} There exists $s_* \in (0,1)$ with the following property. Let $s \in (s_*,1)$ and $u \in \mathcal{A}(\Sigma) $ be a minimizer in the definition of $\M_s(\Sigma)$. Then
\begin{itemize}
    \item[$(i)$] The Minkowski dimension of ${\rm sing}(u) \setminus \supp \Sigma $  is at most $n-3$.
     \item[$(ii)$] $u \in W^{1,p}_{\rm loc} (M\setminus \supp \Sigma)$ for all $ p \in [1,3)$ with quantitative estimates.
\end{itemize}
\end{teo}

This theorem follows from the analogous local result, Theorem \ref{thm: local improved W1p p<3}, for a suitable class of $\Sp^1$-valued fractional harmonic maps with vanishing Jacobian; see Section \ref{sec: sec 3} for details. 

The central mechanism behind Theorem \ref{thm: improved W1p p<3} (and Theorem \ref{thm: local improved W1p p<3}) is a rigidity result for tangent maps with vanishing Jacobian. After a classical dimension-reduction, this amounts to classifying \(0\)-homogeneous $\sigma$-harmonic minimizing maps in $\R^2$ whose Jacobian is zero.
The next theorem gives this classification when \(\sigma\) is sufficiently close
to \(1\).

We point out that the threshold $\sigma_*$ (and thus the corresponding $s_*$ in Theorem~\ref{thm: improved W1p p<3}) is not obtained by compactness, but in principle it might be made explicit (see Remark~\ref{rem:s*}).

\begin{teo}\label{thm: classification R2} There exists $\sigma_* \in (1/2,1)$ with the following property. Let $\sigma\in (\sigma_* ,1)$ and let $u \in H^\sigma_{\rm loc}(\R^2; \Sp^1)$ be a $0$-homogeneous, null-Jacobian minimizing map (see Definition \ref{def:null-jac-min}). Then $u$ is constant.  
\end{teo}

In the local case, say for $p$-harmonic maps in two dimensions with $p<2$, the analogous rigidity statement follows immediately from lifting and the maximum principle applied to the trace of the phase on the unit circle. In the fractional setting, this strategy does not work directly. Although there is a lifting and the trace solves a suitable integro-differential equation on $\Sp^1$, this equation is not elliptic and does not satisfy a maximum principle without an a priori bound on the oscillation of the phase. Instead, the proof requires a genuinely nonlocal argument in $\R^2$, similar in spirit to the classification of two-dimensional $s$-minimal cones for the fractional perimeter by Savin and Valdinoci \cite{SavVal-2Dcones}, which is nontrivial even in codimension one.

 \begin{rem}
    It is reasonable to expect that, perhaps under mild additional assumptions, minimizers in $\mathcal{A}(\Sigma)$ are in fact smooth away from $\Sigma$. In view of the dimension-reduction, one possible way to prove this would be to classify $0$-homogeneous tangent maps with an isolated singularity at the origin in every dimension $n\ge 3$, for $s$ sufficiently close to $1$. We do not address this question in this work. 
\end{rem}

Theorem \ref{thm: improved W1p p<3} shows that, for $s$ sufficiently close to $1$, minimizers $u \in \mathcal{A}(\Sigma)$ have no $(n-2)$-dimensional singularities away from the prescribed $\Sigma$. Thus, when $\Sigma$ is stationary for $\M_s$ (with respect to inner-variations), every associated stationary $\sigma$-harmonic map in $\mathcal{A}(\Sigma)$ realizes $\Sigma$ exactly as its top-dimensional singular stratum.

\begin{cor}
   Let $s \in (s_*,1)$ and $\Sigma$ be stationary for $\M_s$ under inner-variations. Then, every map $u\in \mathcal{A}(\Sigma)$ is a stationary $\sigma$-harmonic map such that the $(n-2)$-dimensional stratum of the singular set of $u$ is exactly $\Sigma$.  
\end{cor}

This corollary follows immediately from Theorem~\ref{thm: improved W1p p<3} and the following envelope formula for the first variation of $\M_s$. Since $\M_s$ is defined by minimizing the energy under the Jacobian constraint, differentiating $\M_s$ under a deformation of $\Sigma$ requires comparing the energies of transported minimizers in $\mathcal{A}(\Sigma)$, which may be non-unique even up to rotations in $\Sp^1$. The next result gives the corresponding envelope formula.

\begin{prop}[Envelope formula for the first variation of $\M_s$]\label{lem:envelope-Ms}
Let $s\in (0,1)$ and let $\Sigma\subset M $ be an admissible boundary such that $\M_s(\Sigma) <+\infty$. Let also $X$ be a smooth vector field on $M$. Then
\begin{align}
\frac{d}{dt}\bigg|_{t=0^+} \M_s(\phi_t^X(\Sigma))
&=
\min_{u\in\mathcal{A}(\Sigma)}
\frac{d}{dt}\bigg|_{t=0^+} [u\circ\phi_{-t}^X]^2_{H^{\frac{1+s}{2}}(M)}, \label{eq: right der} \\[1ex]
\frac{d}{dt}\bigg|_{t=0^-} \M_s(\phi_t^X(\Sigma))
&=
\max_{u\in\mathcal{A}(\Sigma)}
\frac{d}{dt}\bigg|_{t=0^-} [u\circ\phi_{-t}^X]^2_{H^{\frac{1+s}{2}}(M)}. \label{eq: left der} 
\end{align}
\end{prop}

In particular, whenever the first variation $\frac{d}{dt}\big|_{t=0} \M_s(\phi_t^X(\Sigma)) $ exists (e.g. when $\Sigma$ is minimal/stationary in a suitable class, and this quantity is zero), or whenever $\mathcal{A}(\Sigma)$ is a singleton up to rotations (i.e. the minimum is unique up to rotations), we have 
\begin{equation*}
\frac{d}{dt}\bigg|_{t=0} \M_s(\phi_t^X(\Sigma))
=
\frac{d}{dt}\bigg|_{t=0} [u \circ\phi_{-t}^X]^2_{H^{\frac{1+s}{2}}(M)} \quad \mbox{for all } u\in \mathcal{A}(\Sigma). 
\end{equation*}

\subsubsection{Additional results at fixed $s$}

Lastly, we also establish several structural properties of the fractional $s$-mass for fixed $s \in (0,1)$. 

A first point concerns the definition of the $s$-mass itself. In \cite[Remark~1.6]{codim2-frac} the authors observed that one could replace the weak-linking condition used there by the requirement of having a prescribed Jacobian, as in the present work, without affecting the asymptotic results for $s \to 1$. In Lemma \ref{lem: equivalence linking-Jacobian} we show that the two constructions yield the same notion of fractional mass: the weak-linking and prescribed-Jacobian definitions of $\M_{s}$ are equivalent. In addition, in Corollary \ref{cor: equivalence inf on smooth} we show the equivalence with the definition of $\M_s$ via minimization on \emph{smooth} linking maps.

We then analyze the natural domain of $\M_s$ on currents and its relation to the classical mass; see Section \ref{sbs: domain of the s-mass}. In particular, this discussion clarifies which finite-mass currents are detected by our fractional theory at a fixed $s$, and how this differs from the regime $s \to 1$. Finally, we prove basic variational properties of $\M_s$, including coercivity, lower semicontinuity, and a fractional isoperimetric inequality; see Proposition \ref{prop: fixed s coerc and lsc} and Proposition \ref{prop: frac iso}.

Taken together, these results show that, in codimension two, the fractional $s$-mass is not merely an asymptotic approximation of the area, but also defines a robust intrinsic nonlocal theory on closed manifolds.

\begin{rem}\label{rem: localized stuff} For clarity, we have focused on presenting our results in the case of a closed manifold $M$. Nevertheless, every statement and proof remain valid when $M$ is replaced by $\R^n$ or by a smooth bounded domain $\Omega \subset \R^n$, with essentially no modifications. The only result that would require a minor modification is the compactness $(i)$ of Theorem \ref{thm: new main asymptotics}. In the case of a bounded domain, one replaces the spectral characterization of \(H^\sigma\) used here on a closed manifold with the corresponding Fourier transform characterization in $\R^n$, after applying a bounded extension operator from \(H^\sigma(\Omega)\) to \(H^\sigma(\mathbb R^n)\). 
\end{rem}

\smallskip\noindent
\textbf{Organization of the paper.}

Section \ref{sec:Preliminaries}: we collect the needed preliminaries on fractional Sobolev spaces on manifolds, currents and distributional Jacobians, the flat topology, and we recall the Ginzburg--Landau equi-coercivity theorem of \cite{2005-Indiana-ABO}. 

Section \ref{sec: sec 3}: we prove Theorem~\ref{thm: improved W1p p<3} and Theorem~\ref{thm: classification R2}, establishing the (local) improved regularity and higher integrability for minimizing $\sigma$-harmonic maps with vanishing Jacobian. 

Section \ref{sec: frac s mass}: we introduce the fractional $s$-mass, prove its equivalence with the weak-linking formulation, prove the envelope formula for the first variation in Proposition~\ref{lem:envelope-Ms} together with further properties at fixed $s$.

Section \ref{sec: gamma comp} is devoted to the proof of the equi-coercivity and $\Gamma$-convergence results on currents, including the liminf and limsup inequalities. 

Lastly, in the Appendix we collect auxiliary results used throughout the paper, which we have not found anywhere in the literature.

\section{Preliminaries}\label{sec:Preliminaries}

We write $M$ for an $n$-dimensional smooth, oriented, closed (i.e. compact and without boundary) Riemannian manifold, $dV$ for the Riemannian volume form of $M$ and $\Delta$ for its Laplace--Beltrami operator, with the convention that $\Delta$ is a nonpositive operator. 

\smallskip

For $x\in M$ and $r>0$, we denote by $\mathcal{B}_r(x)$ the geodesic ball centered at $x$ with radius $r$ in~$M$. In the Euclidean setting, $B_r(x)$ denotes the usual Euclidean ball.

\subsection{Fractional Sobolev spaces on Riemannian manifolds}

Let $\sigma \in (0,1)$. We recall here the definition of the fractional Sobolev spaces $H^{\sigma}(M)$ on (closed) manifolds and their associated seminorm. We refer to \cite{CFSfrac} for more details and proofs.

Let $H_M(x,y,t)$ denote the heat kernel of $M$, namely, the minimal positive fundamental solution to the heat equation $\partial_t u - \Delta u = 0$ on $M$ such that $u(\cdot, t) \to \delta_y$ in the distributional sense as $t \to 0^+$.

For $x,y \in M$ we define the singular kernel
\begin{equation*}
  \K_\sigma(x,y)
  := \alpha_{n,\sigma}^{-1} \frac{\sigma}{\Gamma (1-\sigma )}
     \int_0^{\infty} H_M(x,y,t)\,\frac{dt}{t^{1+\sigma}} , 
  \end{equation*}
 where 
\begin{equation}\label{eq: alpha-def}
\alpha_{n,\sigma} := \frac{4^\sigma \Gamma(n/2+\sigma)}{\pi^{n/2} |\Gamma(-\sigma)|},
\end{equation}
and there holds
\begin{equation}\label{eq: alpha-est}
c \, \sigma(1-\sigma) \le \alpha_{n,\sigma} \le C \sigma(1-\sigma)
\end{equation}
for some dimensional constants \(c,C>0\).

Then, the fractional Sobolev seminorm of $u\in L^2(M; \R^2)$ is given by 
\begin{equation}\label{eq:Hs-kernel}
  [u]_{H^{\sigma}(M)}^2
  := \iint_{M\times M} |u(x)-u(y)|^2 \K_\sigma(x,y)\,dV(x) dV(y) , 
\end{equation}
and we set
\begin{gather*}
     H^{\sigma}(M; \R^2)
  := \bigl\{u\in L^2(M;\R^2) \, : \,   [u]_{H^{\sigma}(M)}<\infty \bigr\}, \\[0.5ex]  H^{\sigma}(M; \Sp^1)
  := \bigl\{u \in H^{\sigma}(M; \R^2) \, : \,  |u|=1 \mbox{ a.e. on } M \bigr\}.
\end{gather*}

For a domain $\Omega\subseteq M$ we will also use the fractional energy
$$\E_\sigma(u,\Omega):=\iint_{M\times M \setminus \Omega^c\times \Omega^c} |u(x)-u(y)|^2 \K_\sigma(x,y)\,dV(x) dV(y).$$

\begin{rem}\label{rem: normalization K}
    Note that the normalization constant in the definition of $\K_\sigma$ has been chosen such that, if the closed manifold $M$ is replaced by the Euclidean space $\R^n$, then
    \begin{align*}
    \K_{\sigma}(x,y) & = \alpha_{n,\sigma}^{-1} \frac{\sigma}{\Gamma(1-\sigma)} \int_{0}^{\infty} H_{\R^n}(x,y,t) \frac{dt}{t^{1+\sigma}} \\[0.5ex] & = \alpha_{n,\sigma}^{-1} \frac{\sigma}{\Gamma(1-\sigma)} \int_{0}^{\infty} \bigg(  \frac{1}{(4\pi t)^{\frac{n}{2}}} e^{-\frac{|x-y|^2}{4t}}\bigg) \frac{dt}{t^{1+\sigma}} = \frac{1}{|x-y|^{n+ 2 \sigma}},
\end{align*}
Hence, we recover the usual form of the Gagliardo seminorm on $\mathbb{R}^n$. Also, observe that in our definition of $\K_\sigma$, we have normalized by an additional factor of $\alpha_{n,\sigma}^{-1}$ compared to \cite{CFSfrac}. 
\end{rem}

The kernel $\K_\sigma$ is symmetric, smooth away from the diagonal, and
exhibits the expected singular behavior $\K_\sigma(x,y)\sim {\rm dist}(x,y)^{-(n+2 \sigma)}$
as $x \to y$; see Proposition \ref{prop: comparison kernel sharp} in the Appendix or \cite{CFSfrac} for a detailed discussion. 

\subsubsection{Spectral and Fourier representations} In the proof of {\it (i)} in Theorem \ref{thm: new main asymptotics}, we will need the equivalence of our characterization of $H^\sigma$ above with the one given by spectral theory for the Laplacian. 

\begin{prop}[{\cite[Proposition 3.2]{CFSfrac}}] \label{prop: spectral characterization} Let $\{\varphi_k\}_{k\ge 0}$ be an orthonormal basis of $L^2(M)$ of eigenfunctions of $-\Delta$, and let $\{\lambda_k\}_{k \ge 0}$ be the corresponding eigenvalues. Let $\sigma \in (0,1)$ and $u\in H^\sigma(M)$. Then 
\begin{equation*}
    \frac{\alpha_{n,\sigma}}{2} [u]^2_{H^{\sigma}(M)} =  \sum_{k \ge 0} \lambda_k^\sigma \langle u, \varphi_k\rangle_{L^2(M)}^2 ,  
\end{equation*}
where $\alpha_{n,\sigma}$ is given by \eqref{eq: alpha-def}. 
    
\end{prop}

\subsection{Currents, Jacobians and flat topology}

In this subsection, we fix the basic notation for Jacobians and currents and refer the reader to \cite[Section~2]{2005-Indiana-ABO}, \cite[Section~3]{ABO2-singularities}, and \cite[Section~2]{p-energy} for a detailed discussion. 

Let \(\Omega\subset M \) be an open set (possibly coinciding with the whole $M$) and denote by \(\mathcal{D}_k(\Omega)\) the space of $k$-dimensional currents in $\Omega$, endowed with the mass $\mathbb{M}_\Omega(T)$ and the boundary operator \(\partial T\).

For a \(k\)-rectifiable set $E$, a measurable function $\theta : E \to \Z\setminus\{0\}$, and a choice of unit simple \(k\)-vectors \(\tau :  E\to \bigwedge_k\R^n\) spanning \(T_x E\) for \(\mathcal{H}^k\)-a.e. \(x\in E\), we denote by $ \lcurr E,\theta,\tau \rcurr$ the \(k\)-current defined by
\[
\langle \lcurr E,\theta,\tau \rcurr ,\omega\rangle
:= 
\int_E \theta(x)\,\tau(x)\cdot \omega(x) \,d\mathcal{H}^k(x), \qquad \forall \, \omega \in \mathcal{D}^k (M) . 
\]
A \(k\)-current \(T\) is said to be \emph{rectifiable} if \(T= \lcurr E,\theta,\tau \rcurr\) for some \(E\), \(\theta\), and \(\tau\) as above.

\subsubsection{Flat norm} For an integral $k$-boundary $T$ we define the \emph{flat norm} of $T$ as
\begin{equation}\label{eq: flat int}
\F_\Omega (T):=\inf \big\{\M_\Omega (S): \text{$S$ is an integral $(k+1)$-current with $\partial S =T$ in $\Omega$}\big\}.
\end{equation}
In our case of $k=n-2$, the above notion of ``integral flat norm" is known (see Proposition~4.1 and Proposition~4.2 in \cite{Bre-Mi}) to coincide with 
\begin{equation*}
    \F_\Omega (T):=\inf \big\{\M_\Omega (S): \text{$S$ is any $(n-1)$-current with $\partial S = T $ in $\Omega$}\big\}.
\end{equation*}
Furthermore, by a classical theorem of Wolfe \cite[IX, Theorem 7C]{WhytBook}, see also \cite{WhitFormsExt}, this last expression can be equivalently expressed by duality as  
\begin{equation}\label{eq: flat dual}
    \F_\Omega(T):=\sup \big\{ \langle T,\omega \rangle: \omega \in \D^{n-2} (\Omega), \, \|d \omega \|_{L^\infty(\Omega)}\leq 1 \big\} . 
\end{equation}
We will implicitly use the equivalence between \eqref{eq: flat int}-\eqref{eq: flat dual} many times. 

\subsubsection{The distributional Jacobian}

For a map $u\in W^{1,p}(\Omega;\R^k)$ with \(p\ge k-1\), we denote by \( \star Ju\in\mathcal{D}_{n-k}(\Omega)\) its distributional Jacobian, characterized by
\begin{equation} \label{eq:def-jac-1}
    \langle \star Ju,\omega\rangle \;=\; \int_\Omega d\omega\wedge j(u),\qquad 
    j(u):=\frac1k \sum_{i=1}^k (-1)^{i-1} u^i\, du^1\wedge\cdots\wedge \widehat{du^i}\wedge\cdots\wedge du^k,
\end{equation}
for every test form \(\omega\in C_c^\infty(\Omega;\Lambda^{n-k} )\). For maps \(u\) taking values in \(\Sp^{k-1}\) the Jacobian \(\star Ju\) is always an admissible boundary, so whenever it has locally finite mass it is an integral \((n-k)\)-current (see \cite[Section~3]{ABO2-singularities}). We recall here only the basic properties needed in the sequel, referring to the works mentioned above for further background and proofs.

For our purposes, we need the following quantitative estimates of the flat distance between Jacobians
of nearby maps, due to Brezis and Nguyen.

\begin{teo}[see Theorem~1 in \cite{BN-Invent} and Corollary~1.7 in \cite{Bre-Mi}]\label{teo:BN-Invent}
Let $\Omega \subseteq M$ be a bounded Lipschitz domain on a Riemannian manifold $M$, let $p\in [k-1,+\infty]$ and let $q\in [1,+\infty]$ be such that $(k-1)/p + 1/q=1$. Then, for every $u,v\in W^{1,p}(\Omega;\R^k)\cap L^{q}(\Omega;\R^k)$, there holds
\begin{equation*}
\F_\Omega(\star Ju-\star Jv)\leq C_{k,\Omega} \|u-v\|_{L^q}(\|\nabla u\|_{L^p} + \|\nabla v\|_{L^p})^{k-1} .
\end{equation*}
\end{teo}

\subsubsection{Equi-coercivity for the Ginzburg--Landau functional}

The equi-coercivity part of our main results, {\it (i)} in Theorem \ref{thm: new main asymptotics}, relies on the analogous equi-coercivity result in \cite{2005-Indiana-ABO} for the Ginzburg--Landau functionals 
\begin{equation} \label{eq:def-GL-1}
    {GL}_{\eps}^n(u,\Omega) := \int_{\Omega} \frac{1}{2} |\nabla u|^2 + \frac{(1-|u|^2)^2}{4 \eps^2}, \quad u \in W^{1,2}(\Omega; \R^2) .  
\end{equation}
Here $n\ge 2$ and $\Omega$ is either a bounded domain in $\R^n$ or a closed Riemannian manifold of dimension $n$.

\begin{teo}[Theorem 1.1 in \cite{2005-Indiana-ABO}]\label{thm: ABO comp} Let $n\ge 2$ and $\ep>0$. Assume that a (discrete) sequence of maps $u_\ep \in W^{1,2}(\Omega; \R^2)$ satisfies
\begin{equation*}
\limsup_{\ep \to 0} \, \frac{{GL}_\eps^n(u_\eps,\Omega)}{\abs{\log \eps}}<+\infty .
\end{equation*}
Then, there exists an integral $(n-2)$-boundary $\Sigma$ such that (up to a subsequence) the Jacobians $\star Ju_{\ep}$ converge to $\pi \Sigma$ in the flat topology of $\Omega$. 
\end{teo}

\subsection{Factorization theorem and Jacobian of maps in $H^{\sigma}(M;\Sp^1)$}\label{sbs: factor Jac}

Let $\sigma\in (1/2,1)$. There are several equivalent ways to define a Jacobian $\star Ju$ for a map $u \in H^{\sigma}(M; \Sp^1)$, which coincides with the usual Jacobian $(n-2)$-current if $u$ is also in $W^{1,1}(M; \Sp^1)$. For our purposes, it is convenient to use the factorization result \cite[Theorem 7.1]{Bre-Mi}, stating that every $u \in H^{\sigma}(M; \Sp^1)$ can be written as 
\begin{equation}\label{eq: Mir factorization}
    u= e^{i \varphi} v , \qquad \varphi \in H^{\sigma}(M) , \quad  v \in W^{1, 2\sigma}(M ; \Sp^1)  . 
\end{equation}
Then, one defines 
\begin{equation*}
    \star Ju := \star Jv , 
\end{equation*}
where $\star Jv$ is the Jacobian $(n-2)$-current associated to $v$ defined by~\eqref{eq:def-jac-1}. This definition is independent of the choice of the factorization and is independent of $\sigma$; we refer to \cite[Section 8.1]{Bre-Mi} for details.

We shall also use a fractional counterpart of Theorem \ref{teo:BN-Invent}, which follows from the local version applied to the harmonic extension (see \cite[Lemma~9]{BBM-ihes}, \cite[Theorem~3]{BN-Invent}, \cite[Theorem~1.1]{Bousquet-Mironescu}).

\begin{teo}\label{teo:BN-Invent fractional}
Let $\Omega \subseteq M$ be a bounded Lipschitz domain on a Riemannian manifold $M$. Then, for every $u,v\in H^{1/2}(\Omega;\R^2)$, there holds 
\begin{equation*}
\F_\Omega(\star Ju-\star Jv)\leq C_{\Omega} [u-v]_{H^{1/2}} ( [u]_{H^{1/2}} + [ v ]_{H^{1/2}}) .
\end{equation*}
\end{teo}

Finally, we recall a useful density result in $H^\sigma(M,\Sp^1)$. Define the class 
\begin{align*}
      \mathcal{R} :=  \Big\{ & u \in C^\infty (M\setminus \mathcal{S}; \Sp^1) \cap W^{1,1}(M; \Sp^1) \, : \,  \star J u = \pi \mathcal{S} , \, \mathcal{S}=\bigcup_{j=1}^r \Sigma_j , \: \Sigma_j \mbox{ smooth, disjoint, } \\ & (n-2)\mbox{-submanifold of $M$, and } |\nabla^\ell u(x)| \le C(\ell,M) \dist(x, \mathcal{S})^{-\ell}, \, \forall \, \ell \ge 1, \, \forall \, x\in M \Big\} . 
   \end{align*}

\begin{teo}\label{thm: approx with standard singularities} Let $\sigma \in [1/2,1)$, then the class $\mathcal{R}$ is dense in $H^{\sigma}(M; \Sp^1)$.    
\end{teo}

\begin{proof} 
The density statement is proved, for instance, in \cite[Theorem~10.3]{Bre-Mi} and \cite[Theorem~1]{Muccidensity}. We only comment on the disjointness of the $\Sigma_j$ in $\mathcal{S} = \cup_{j=1}^r \Sigma_j$. 

In the construction in \cite[Theorem~10.3]{Bre-Mi}, the singular set $\mathcal{S}$ is obtained as the inverse image of a regular value of a smooth map. In particular, it can be written as the disjoint union of its connected components $\Sigma_j$. 
\end{proof}

\section{Improved regularity for $\sigma$-minimizing maps with null Jacobian}\label{sec: sec 3}

In this section we discuss the properties of minimizers $u$ in the definition of $\M_s(\Sigma)$. Let $u \in \mathfrak{F}_s^J(\Sigma)$ be a minimizer in the definition of $\M_s$ and set
\[
    \sigma := \frac{1+s}{2} \in \bigg(\frac12,1\bigg).
\]
Then we deduce that $u$ is a null-Jacobian minimizing $\sigma$-harmonic map in $M\setminus \supp \Sigma$, according to the following definition.

\begin{defn}[Null-Jacobian minimizing maps] \label{def:null-jac-min}
Let $\Omega \subseteq M$ be open. A map $u \in H^\sigma(\Omega; \Sp^1)$ is said to be a null-Jacobian minimizing $\sigma$-harmonic map if $\star J u = 0 $ in $\Omega$ and $\E_\sigma(u,\Omega) \le \E_\sigma(v,\Omega) $ for every $v \in H^\sigma(\Omega; \Sp^1)$ with $\star Jv=0$ and ${\rm spt}(u-v) \Subset \Omega$. 
    
\end{defn}

A null-Jacobian minimizing map lies between the two more familiar notions of minimizing $\sigma$-harmonic map with vanishing Jacobian and a stationary $\sigma$-harmonic map with vanishing Jacobian.

As a consequence, we obtain that the following properties hold.

\begin{itemize}
    \item[$(i)$] $u$ is a weakly $\sigma$-harmonic map in $M \setminus \supp \Sigma$, namely, for every $\Omega \Subset M \setminus \supp \Sigma $ we have 
\begin{equation*}
     \frac{d}{dt}\bigg|_{t=0} \E_\sigma \bigg( \frac{u + t \varphi }{|u + t \varphi|} , \Omega \bigg) =0 \qquad \forall \varphi \in C^\infty_c(\Omega). 
\end{equation*}

This follows directly from the minimality of $u$ away from $\Sigma$. In particular, see for example \cite[Section 3]{Partialreg}, $u$ is a solution of the $\sigma$-harmonic map equation
     \begin{equation*}
        (-\Delta)^\sigma u  = A_\sigma(u) u    \qquad \mbox{in } \, \mathscr{D}'(M \setminus \supp \Sigma) ,  
        \end{equation*}
         where
        \begin{equation*}
        A_\sigma(u)(x) :=  \int_M |u(x)-u(y)|^2 \K_\sigma(x,y) \, dV(y)  . 
    \end{equation*}

Moreover, if $\Sigma$ is a smooth $(n-2)$-dimensional surface, then it has zero $H^\sigma$-capacity and hence $u$ is a weakly $\sigma$-harmonic map in all of $M$ (i.e. also across $\Sigma$).

\item[$(ii)$] $u$ is a stationary $\sigma$-harmonic map in $M \setminus \supp \Sigma $. That is, for every $\Omega \Subset M \setminus \supp \Sigma $ and smooth vector field $X$ with $\supp X \Subset \Omega $, 
    \begin{equation*}
        \frac{d}{dt}\bigg|_{t=0} \E_\sigma(u \circ \phi_{-t}^X, \Omega ) = 0, 
    \end{equation*}
where $\phi_{t} ^X$ is the flow on $M$ generated by $X$. Indeed, the condition $\star Ju=0$ in $M\setminus \supp \Sigma$ is preserved under composition with flows supported away from $\Sigma$, so minimality in our class implies stationarity.

 \item[$(iii)$] $u \in W^{1,p}_{\rm loc}(M\setminus \supp \Sigma)$ for every $p\in [1,2)$. This follows by \cite[Theorem 1.5]{GlobRegFrac} since $u$ is stationary in $M\setminus \supp \Sigma$ by part $(ii)$. In particular, $u$ has a well-defined weak gradient $\nabla u \in L^1_{\rm loc}(M \setminus \supp \Sigma)$ away from $\Sigma$. We will use this in the proof of Proposition \ref{prop: compactness min}.

\end{itemize}

\subsection{The $0$-homogeneous case}

The goal of this subsection is to prove Theorem \ref{thm: classification R2}. As a first step, we prove a weaker version that we will later use in the proof of Theorem \ref{thm: classification R2}.

\begin{prop}\label{prop:small_energy_minimizers}
   Let $\sigma_\circ \in (1/2,1)$ and $\sigma\in (\sigma_\circ,1)$. Then, there is a constant $\kappa=\kappa(\sigma_\circ) >0 $ with the following property. Let $u \in H^\sigma_{\rm loc}(\R^2; \Sp^1)$ be a $0$-homogeneous weakly $\sigma$-harmonic map with
   \begin{equation}\label{eq: quant no Jacobian}
       (1-\sigma)^2 [u]^2_{H^\sigma(B_1)} \le \kappa .
   \end{equation}
    Then, $u$ is constant. 
\end{prop}

\begin{proof}
    Since $u$ is $0$-homogeneous, we can write $u(x)=g(x/\abs{x})$ for some $g \colon \Sp^1\to\Sp^1$.
    
    Recall the slicing formula (see~\cite[eq.~(2.7)]{codim2-frac})
        \begin{align}
            [u]_{H^{\sigma}(B_1)} ^2 
            =
            \frac{1}{4} \int_{\Sp^1} d \theta \int_{0} ^{1} \Big( [u]_{H^{\sigma}(S_{\theta,t})} ^2 + 
            [u]_{H^{\sigma}(S_{\theta^{\perp},t})} ^2 +
            [u]_{H^{\sigma}(S_{\theta,-t})} ^2 +
            [u]_{H^{\sigma}(S_{\theta^{\perp},-t})} ^2 \Big) \, dt, \label{eq:planar_slicing_squares-new}
        \end{align}
where $\theta^\perp$ denotes the counterclockwise rotation of $\theta \in \Sp^1$ by an angle equal to $\pi/2$, and the set $S_{\theta,t}:=\{t\theta^\perp +\xi\theta : \abs{\xi} \in [0,\sqrt{1-t^2})\}$ is the intersection of the line $t\theta^\perp+\theta \R$ with the unit ball $B_1 \subset \R^2$.

Since the projection $x\mapsto x/|x|$ is smooth on every segment $S_{\theta,t}$ with $t>0$, it follows that $g\in H^{\sigma}(\Sp^1;\Sp^1)$, and in particular $u$ is continuous away from the origin, by the embedding $H^{\sigma}(\Sp^1;\Sp^1) \hookrightarrow C^0(\Sp^1;\Sp^1)$ for $\sigma > 1/2$.

More quantitatively, by \cite[Lemma~2.8]{codim2-frac}, we obtain that
$$\osc(u,S_{\theta,t}) ^2\leq C_0 (1-\sigma) (2t)^{2\sigma-1} [u]_{H^{\sigma}(S_{\theta,t})} ^2,$$
where $C_0>0$ depends only on $\sigma_\circ$.

Observe that, for every $t \in (0,1/\sqrt{2})$, the union of the four segments in the right-hand side of (\ref{eq:planar_slicing_squares-new}) contains the boundary of a square of side length $2t$ centered at the origin. Let $L_t:=\partial (-t,t)^2$ and for every $\theta \in \Sp^1$ we denote by $\theta \cdot L_t$ the set $\{ \theta z : z \in L_t \}$. Then, since the continuity of $u$ prevents jumps in the vertices, for every $\theta\in \Sp^1$ and $t\in (0,1/\sqrt{2})$ it holds that
$$\osc(g,\Sp^1)^2 =\osc(u,\theta\cdot L_t)^2 \leq 4\bigl(\osc(u,S_{\theta,t}) ^2 + \osc(u,S_{\theta^{\perp},t}) ^2 + \osc(u,S_{\theta,-t}) ^2 + \osc(u,S_{\theta^{\perp},-t}) ^2 \bigr),$$
and hence
    \begin{equation*} \notag
            [u]_{H^{\sigma}(B_1)} ^2 
            \ge \frac{\osc(g,\Sp^1)^2}{16 C_0 (1-\sigma)} \int_{\Sp^1} d \theta \int_{0} ^{\frac{1}{\sqrt{2}}}
            (2t)^{1-2\sigma} \, dt 
            =\frac{\pi }{2^{4+\sigma} C_0 (1-\sigma)^2} \mathrm{osc}(g, \Sp^1)^2.
        \end{equation*}

    Recalling the assumption~\eqref{eq: quant no Jacobian} we obtain 
    \begin{equation*}
        \osc(g, \Sp^1)^2 \le \frac{32 C_0}{\pi} (1-\sigma)^2 [u]^2_{H^{\sigma}(B_1)} \le \frac{32 C_0}{\pi}\kappa,
    \end{equation*}
    and choosing $\kappa:= \pi/(64C_0)$ this gives $ \osc(g, \Sp^1) \le 1/2$, so the map $g$ cannot be surjective.
    
    Therefore $\mathrm{deg}(g) = 0$, so Theorem~\ref{teo:fractional_lifting} gives $g = e^{i\varphi}$, for some (continuous) $\varphi \in H^{\sigma}(\Sp^1)$. Moreover,
        \[
            \abs*{ \sin \bigg( \frac{\varphi(x)-\varphi(y)}{2} \bigg) } = \frac{\abs{e^{i\varphi(x)}- e^{i\varphi(y)}}}{2} \le \frac{1}{4}, \qquad \forall x,y \in \Sp^1.
        \]
    By continuity, this readily implies that $\abs{\varphi(x)-\varphi(y)} \le 2 \arcsin(1/4)$ for every $x, y \in \Sp^1$.
        
    Finally, we exploit the fact that $u$ is weakly $\sigma$-harmonic. This property translates into the following equation for the phase $\varphi$: 
    \begin{equation*}
        \int_{\Sp^1} \sin(\varphi(x)-\varphi(y)) \widetilde K_\sigma(x,y) \, dy = 0, \qquad \forall x \in \Sp^1, 
    \end{equation*}
    where $\widetilde K_\sigma$ is a suitable positive kernel. Applying this identity at a maximum point $x_0 \in \Sp^1$ for $\varphi$, we deduce that $\sin(\varphi(x_0) - \varphi(y)) \ge 0$. Using again that the oscillation of $\varphi$ is bounded by $2\arcsin(1/4) < \pi/2$, we obtain that $\varphi(x_0) - \varphi(y) = 0$, for every $y \in \Sp^1$. This concludes the proof.
\end{proof}

The next two results are well-known estimates for fractional norms, in which we need to emphasize the correct behavior of the constants when $\sigma\sim 1$. The first one is the following interpolation inequality.
\begin{lemma}\label{lemma:interpolation}
There exists a dimensional constant $C>0$ such that for every $\sigma \in (0,1)$ it holds
\begin{equation*}
    \sigma(1-\sigma)[f]_{H^\sigma(B_1)}^2\leq C \|f\|_{L^2(B_1)} ^{2-2\sigma} \|\nabla f\|_{L^2(B_1)} ^{2\sigma}\qquad \forall f\in H^1(B_1).
\end{equation*}
\end{lemma}
\begin{proof} By \cite[Theorem 5.7]{foghem2026interp} applied with $(p, \eta, s , \sigma ) = (2,\sigma,1,0)$ we have that 
\begin{equation*}
    \sigma(1-\sigma)\|g\|_{H^\sigma(B_1)}^2 \leq C \|g \|_{L^2(B_1)} ^{2-2\sigma} \| g \|_{H^1(B_1)} ^{2\sigma}\qquad \forall g\in H^1(B_1).
\end{equation*}
Applying this inequality to $g=f-\fint_{B_1} f$ and using the Poincaré inequality for $H^1(B_1)$ gives the desired result. 
\end{proof}

The second estimate that we need is the following fractional version of the Hardy inequality.
\begin{lemma}[Fractional Hardy inequality]\label{lemma:frac-hardy}
Let $\Omega\subset \R^n$ be a convex open set and let $\sigma_0\in (1/2,1)$, $\sigma\in[\sigma_0,1)$. Then there exists a constant $C=C(\sigma_0,n)$ such that
    \[
        \int_{\Omega} \frac{f(x)^2}{d_\Omega(x)^{2\sigma}}\,dx\leq C(1-\sigma)[f]_{H^\sigma(\Omega)} ^2, \qquad \forall f\in H^\sigma _0(\Omega),
    \]
where $d_\Omega(x):=\dist(x,\partial \Omega).$
\end{lemma}

\begin{proof}
For $\sigma\in (1/2,1)$ equation (16) in \cite{frac-Hardy-ineq} states that
    \[
        \int_{\Omega} \frac{f(x)^2}{d_\Omega(x)^{2\sigma}}\,dx\leq k_{n,\sigma} [f]_{H^\sigma(\Omega)} ^2, \qquad \forall f\in H^\sigma _0(\Omega),
    \]
where
    \[
        k_{n,\sigma}:= \bigg(2\pi^{\frac{n-1}{2}}\frac{\Gamma(\frac{1+2\sigma}{2})}{\Gamma(\frac{n+2\sigma}{2})}\int_0 ^1 \frac{\big(1-r^{\frac{2\sigma-1}{2}}\big)^2}{(1-r)^{1+2\sigma}}\,dr\bigg)^{-1}.
    \]

Hence the claim follows just by observing that the factor in front of the integral is strictly positive for $\sigma>1/2$ and
    \[
        \int_0 ^1 \frac{\big(1-r^{\frac{2\sigma-1}{2}}\big)^2}{(1-r)^{1+2\sigma}}\,dr \ge \int_0 ^1 \frac{\frac{(2\sigma-1)^2}{4}(1-r)^2}{(1-r)^{1+2\sigma}}\,dr=\frac{(2\sigma-1)^2}{8(1-\sigma)},
    \]
where we used the elementary inequality $1-r^\alpha \ge \alpha (1-r)$, valid for all $(\alpha,r) \in (0,1)^2$.
\end{proof}

The last ingredient that we need for the proof of Theorem~\ref{thm: classification R2} is a relation between the $H^\sigma$ seminorm of a $0$-homogeneous map and the $H^\sigma$ seminorm of its trace on the sphere.

\begin{lemma}
\label{lem:reduction-fractional-energy-homogeneous}
Let $n\geq 2$, $\sigma\in (0,1)$ and let $u(x)=g(x/|x|)$ for every $x \in \R^n \setminus \{0\}$. Then,
\begin{equation*}
    [u]_{H^\sigma(B_R)}^2
    =
    \frac{R^{n-2\sigma}}{n-2\sigma}
    \iint_{\Sp^{n-1}\times \Sp^{n-1}}
    |g(\omega)-g(\theta)|^2 \widetilde K_{n,\sigma}(\omega, \theta)
    \, d\omega d\theta,
\end{equation*}
where
\begin{equation*}
   \widetilde K_{n,\sigma}(\omega, \theta)
    :=
    2\int_0^1
    \frac{t^{n-1}}
    {(1+t^2-2t \omega\cdot\theta)^{\frac{n+2\sigma}{2}}}
    \,dt .
\end{equation*}
Moreover,
\begin{equation*}
    \widetilde{K}_{n,\sigma}(\omega,\theta) =   \frac{c_{n,\sigma} + o(1)}{|\omega-\theta|^{n-1+2\sigma}}  ,  \quad \mbox{as } \,  |\omega-\theta| \to 0 ,
\end{equation*}
where the $o(1)$ is independent of $\sigma$ and
\begin{equation*}
    c_{n,\sigma}
    =
    \int_{-\infty}^{\infty}
    \frac{d\tau}
    {(1+\tau^2)^{\frac{n+2\sigma}{2}}}
    =
    \sqrt{\pi}\,
    \frac{\Gamma\big(\frac{n+2\sigma-1}{2}\big)}
    {\Gamma\big(\frac{n+2\sigma}{2}\big)}.
\end{equation*}
\end{lemma}

\begin{proof}
By scaling we can assume $R=1$. Since $u$ is $0$-homogeneous with profile $g$, by polar coordinates writing $x=r\omega$ and $y=\rho\theta$, with
$r,\rho\in(0,1)$ and $\omega,\theta\in \Sp^{n-1}$, we have that
\[
    [u]_{H^\sigma(B_1)}^2
    =
    \iint_{\Sp^{n-1}\times \Sp^{n-1}}
    |g(\omega)-g(\theta)|^2 I_{n,\sigma}(\omega,\theta)
    \,d\omega d\theta,
\]
where
\[
    I_{n,\sigma}(\omega,\theta)
    :=
    \int_0^1\int_0^1
    \frac{r^{n-1}\rho^{n-1}}
    {(r^2+\rho^2-2r\rho\,\omega\cdot\theta)^{\frac{n+2\sigma}{2}}}
    \,d\rho\,dr .
\]
We compute this last integral. We split the square $(0,1)^2$ into the two regions
$\rho<r$ and $r<\rho$. On the first one we set $\rho=rt$. Then $t\in(0,1)$ and
$d\rho=r dt$, and therefore
\begin{align*}
    \int_0^1\int_0^r
    \frac{r^{n-1}\rho^{n-1}}
    {(r^2+\rho^2-2r\rho\,\omega\cdot\theta)^{\frac{n+2\sigma}{2}}}
    \,d\rho\,dr
    &=
    \int_0^1\int_0^1
    \frac{r^{n-1}(rt)^{n-1}r}
    {(r^2(1+t^2-2t\,\omega\cdot\theta))^{\frac{n+2\sigma}{2}}}
    \,dt\,dr                                                     \\
    &=
    \bigg(\int_0^1 r^{n-1-2\sigma}\,dr\bigg)
    \bigg(
    \int_0^1
    \frac{t^{n-1}}
    {(1+t^2-2t\,\omega\cdot\theta)^{\frac{n+2\sigma}{2}}}
    \,dt
    \bigg) 
   \\  &=
    \frac{1}{n-2\sigma}
    \bigg(
    \int_0^1
    \frac{t^{n-1}}
    {(1+t^2-2t\,\omega\cdot\theta)^{\frac{n+2\sigma}{2}}}
    \,dt
    \bigg) . 
\end{align*}
The region $r<\rho$ gives the same contribution by symmetry. Hence
\[
    \widetilde K_{n,\sigma}(\omega, \theta) = (n-2\sigma) I_{n,\sigma}(\omega,\theta)
    =
     2
    \int_0^1
    \frac{t^{n-1}}
    {(1+t^2-2t\,\omega\cdot\theta)^{\frac{n+2\sigma}{2}}}
    \,dt.
\]
It remains to identify the singular behavior of $\widetilde K_{n,\sigma}$ near the diagonal. Let
\[
    \delta:=|\omega-\theta|.
\]
Since
\[
    1+t^2-2t\,\omega\cdot\theta
    =
    (1-t)^2+t|\omega-\theta|^2
    =
    (1-t)^2+t \delta^2,
\]
we have
\[
    \widetilde K_{n,\sigma}(\omega,\theta)
    =
    2\int_0^1
    \frac{t^{n-1}}
    {((1-t)^2+t\delta^2)^{\frac{n+2\sigma}{2}}}
    \,dt .
\]
We set $t=1-\delta\tau$. Then
\[
    \widetilde K_{n,\sigma} (\omega,\theta)
    =
    \frac{2}{\delta^{n+2\sigma-1}}
    \int_0^{1/\delta}
    \frac{(1-\delta\tau)^{n-1}}
    {(\tau^2+1-\delta\tau)^{\frac{n+2\sigma}{2}}}
    \,d\tau .
\]
Since
\[
    \frac{(1-\delta\tau)^{n-1}}
    {(\tau^2+1-\delta\tau)^{\frac{n+2\sigma}{2}}}
    \leq
    \frac{C}{(1+\tau^2)^{\frac{n+2\sigma}{2}}} \in L^1(0,\infty),
\]
by dominated convergence we can pass to the limit as $\delta\to 0$, and we obtain that
\[
   \lim_{|\omega-\theta|\to 0}\widetilde K_{n,\sigma}(\omega,\theta)|\omega-\theta|^{n-1+2\sigma}=
   \lim_{\delta\to 0} 2\int_0 ^{1/\delta}
    \frac{(1-\delta\tau)^{n-1}}
    {(\tau^2+1-\delta\tau)^{\frac{n+2\sigma}{2}}}
    \,d\tau
    =  
    \int_{-\infty}^\infty
    \frac{d\tau}
    {(1+\tau^2)^{\frac{n+2\sigma}{2}}},
\]
and the limit is uniform with respect to $\sigma\in (0,1)$. Finally, using
\[
    \int_{-\infty}^{\infty}
    \frac{d\tau}{(1+\tau^2)^a}
    =
    \sqrt{\pi}\,
    \frac{\Gamma(a-\frac12)}{\Gamma(a)}
    \quad\mbox{with } a = \frac{n+2\sigma}{2} >\frac{1}{2},
\]
we obtain the stated formula for $c_{n,\sigma}$.
\end{proof}

\begin{cor}\label{cor:Mattia}
    Let $\sigma\in [1/2,1)$ and $u \in H^\sigma_{\rm loc}(\R^2;\Sp^1)$ be a $0$-homogeneous map with $u(x)=g(x/|x|)$. Then $g \in H^\sigma(\Sp^1;\Sp^1)$, and there exists an absolute $c > 0$ such that 
        \[
            (1-\sigma) [u]_{H^\sigma(B_1)}^2 \ge c [g]_{H^\sigma(\Sp^1)}^2 . 
        \]
\end{cor}

\begin{proof} 
    This follows directly by Lemma \ref{lem:reduction-fractional-energy-homogeneous} applied to $n=2$ and $R=1$. 
\end{proof}

We are now ready to prove our rigidity result for $0$-homogeneous maps.

\begin{proof}[Proof of Theorem~\ref{thm: classification R2}]
Let $g\colon\Sp^1\to \Sp^1$ be such that $u(x)=g(x/|x|)$. Then, Corollary~\ref{cor:Mattia} implies that $g\in H^\sigma(\Sp^1)$, and in particular $g$ is continuous. Moreover, since $\star Ju=0$, it follows that $g$ has degree zero, and hence it admits a continuous lifting $\varphi\colon\Sp^1\to \R$, for which $g=e^{i\varphi}$ and, by Theorem~\ref{teo:fractional_lifting}, it satisfies
\begin{equation}\label{nonlinear_est_lifting}
    [\varphi ]_{H^{\sigma}(\Sp^1)}^2 \le C  [g]_{H^{\sigma}(\Sp^1)}^2  + \frac{C}{(1-\sigma)^{1-1/(2\sigma)}} [g]_{H^{\sigma}(\Sp^1)}^{2/\sigma} . 
\end{equation}

Now let $\psi\in H^1(B_1)$ be the harmonic extension of $\varphi$, namely the minimizer of the Dirichlet energy in $B_1$ with trace equal to $\varphi$ on $\partial B_1=\Sp^1$. Then by standard Fourier expansion, it holds that
\begin{equation}\label{trace_eq}
\|\nabla \psi\|_{L^2(B_1)}^2 = [\varphi]_{H^{1/2}(\Sp^1)}^2.
\end{equation}

Finally, we set $v:=e^{i\psi}$, and we observe that $v\in H^1(B_1)$ has trace equal to $g$ at $\partial B_1$.

If we extend $v$ to $\R^2$ by setting $v:=u$ outside $B_1$, then the local minimality of $u$ (applied in $B_r$ with $r\searrow 1$) yields
\begin{equation}\label{eq:u_loc_min}
\E_\sigma(u,B_1)\leq \E_\sigma(v,B_1).
\end{equation}

Moreover, since $v$ coincides with $u$ outside $B_1$ it holds that
\begin{align}
\E_\sigma(v,B_1)&=[v]_{H^\sigma(B_1)} ^2 +2\iint_{B_1\times B_1^c} \frac{|u(y)-v(x)|^2}{|y-x|^{2+2\sigma}}\,dxdy \nonumber\\
&\leq [v]_{H^\sigma(B_1)}^2 + \biggl(2+\frac{2}{\ep}\biggr)\iint_{B_1\times B_1^c} \frac{|u(x)-v(x)|^2}{|y-x|^{2+2\sigma}}\,dxdy\nonumber\\
&\quad +(2+2\ep)\iint_{B_1\times B_1^c}\frac{|u(y)-u(x)|^2}{|y-x|^{2+2\sigma}}\,dxdy,\label{est:energy_v}
\end{align}
for every $\ep>0$.

Now we observe that
$$\int_{B_1^c} \frac{dy}{|y-x|^{2+2\sigma}}\leq \int_{B_{1-|x|}^c} \frac{dy}{|y|^{2+2\sigma}}=\frac{\pi}{\sigma}(1-|x|)^{-2\sigma},$$
and hence, since $u-v \in H^\sigma _0(B_1)$, from Lemma~\ref{lemma:frac-hardy} we deduce that
\begin{multline*}
\iint_{B_1\times B_1^c} \frac{|u(x)-v(x)|^2}{|y-x|^{2+2\sigma}}\,dxdy \leq \frac{\pi}{\sigma}\int_{B_1} \frac{|u(x)-v(x)|^2}{(1-|x|)^{2\sigma}}\,dx\\
\leq C(1-\sigma) [u-v]_{H^\sigma(B_1)} ^2\leq C(1-\sigma)\bigl([u]_{H^\sigma(B_1)}^2 +[v]_{H^\sigma(B_1)}^2\bigr).
\end{multline*}

Combining this estimate with (\ref{eq:u_loc_min}) and (\ref{est:energy_v}) we obtain that
\begin{align}
[u]_{H^\sigma(B_1)}^2 &=\E_\sigma(u,B_1)-2\iint_{B_1\times B_1^c}\frac{|u(y)-u(x)|^2}{|y-x|^{2+2\sigma}}\,dxdy\nonumber\\
&\leq [v]_{H^\sigma(B_1)}^2 + \frac{C}{\ep}(1-\sigma)\bigl([u]_{H^\sigma(B_1)}^2 +[v]_{H^\sigma(B_1)}^2\bigr) +2\ep \iint_{B_1\times B_1^c}\frac{|u(y)-u(x)|^2}{|y-x|^{2+2\sigma}}\,dxdy\label{s*1}
\end{align}

Moreover, since $u$ is $0$-homogeneous, we deduce that we can choose $\ep$ (independently of $u$ and $\sigma$) so that
\begin{equation*}
2\ep\iint_{B_1\times B_1^c}\frac{|u(y)-u(x)|^2}{|y-x|^{2+2\sigma}}\,dxdy\leq \frac{1}{4} [u]_{H^\sigma(B_1)}^2.
\end{equation*}

As a consequence, if $\sigma$ is sufficiently close to $1$ so that $ C(1-\sigma)\leq \ep/4$, we conclude that
$$[u]_{H^\sigma(B_1)}^2\leq 3 [v]_{H^\sigma(B_1)}^2.$$

At this point, the following chain of inequalities holds true
\begin{align}
[u]_{H^\sigma(B_1)}^2 &\leq 3[v]_{H^\sigma(B_1)}^2
\stackrel{(1)}{\leq} C_1(1-\sigma)^{-1}\|\nabla v\|_{L^2(B_1)} ^{2\sigma} = C_1(1-\sigma)^{-1}\|\nabla \psi\|_{L^2(B_1)} ^{2\sigma}\nonumber\\
&\stackrel{(2)}{=} C_1(1-\sigma)^{-1}[\varphi]_{H^{1/2}(\Sp^1)} ^{2\sigma}
\stackrel{(3)}{\leq} C_2(1-\sigma)^{\sigma-1} [\varphi]_{H^\sigma(\Sp^1)} ^{2\sigma}\nonumber\\
&\stackrel{(4)}{\leq} C_3 (1-\sigma)^{\sigma-1} [g]_{H^{\sigma}(\Sp^1)}^{2\sigma}  + C_3(1-\sigma)^{-1/2} [g]_{H^{\sigma}(\Sp^1)}^{2}\nonumber\\
&\stackrel{(5)}{\leq} C_4 (1-\sigma)^{\sigma} [u]_{H^{\sigma}(B_1)}^{2\sigma}  + C_4(1-\sigma)^{1/2} [u]_{H^{\sigma}(B_1)}^{2}.\label{s*2}
\end{align}

Here inequality (1) follows from Lemma~\ref{lemma:interpolation}, equality (2) is (\ref{trace_eq}), inequality (3) follows from \cite[Remark~5]{BBM01} and inequality (4) follows from (\ref{nonlinear_est_lifting}) and the subadditivity of the map $x\mapsto x^\sigma$, and inequality (5) follows from Corollary~\ref{cor:Mattia}.

Therefore, if $\sigma$ is sufficiently close to $1$ so that $C_4 (1-\sigma)^{1/2}<1/2$, we obtain that
$$[u]_{H^\sigma(B_1)} ^2 \leq  2C_4 (1-\sigma)^{\sigma} [u]_{H^{\sigma}(B_1)}^{2\sigma},$$
and, finally,
\begin{equation}\label{s*3}
[u]_{H^\sigma(B_1)}^2 \leq (2C_4)^{\frac{1}{1-\sigma}} (1-\sigma)^{\frac{\sigma}{1-\sigma}},
\end{equation}
so by Proposition~\ref{prop:small_energy_minimizers} we conclude that $u$ must be constant if $\sigma$ is sufficiently close to $1$.
\end{proof}

\begin{rem}\label{rem:s*}
Following the proof of Theorem~\ref{thm: classification R2} one may find an explicit value for $\sigma_*$. Indeed, we exploited the smallness of $(1-\sigma)$ exactly three times: in order to absorb the squared seminorm of $u$ in the right-hand sides of (\ref{s*1}) and (\ref{s*2}) into the respective left-hand side (notice that $\ep$ might also be made explicit), and then to ensure that the right-hand side of (\ref{s*3}) is smaller than the constant in Proposition~\ref{prop:small_energy_minimizers}. In all these steps it would be possible to find an explicit value of $\sigma_*$ (depending only on the sharp constants in some inequalities) for which the required estimate holds.
\end{rem}

\subsection{Higher integrability for $\sigma$-minimizing maps: proof of Theorem \ref{thm: improved W1p p<3}}

We recall that $\mathcal{B}_r(x)$ denotes the geodesic ball centered at $x$ with radius $r$ in $M$. The goal of this subsection is to prove the following result, from which Theorem \ref{thm: improved W1p p<3} follows.

\begin{teo}\label{thm: local improved W1p p<3} There exists $\sigma_* \in (1/2,1)$ such that, for every $\sigma\in (\sigma_*,1)$, if $u \colon \B_{4R}(x) \to \Sp^1$ is a null-Jacobian minimizing $\sigma$-harmonic map satisfying $\E_\sigma(u, \B_{4R}(x)) \le \Lambda $, then
\begin{itemize}
   
    \item[$(i)$] The Minkowski dimension of ${\rm sing}(u)$ is at most $n-3$.
     \item[$(ii)$] $u \in W^{1,p} (\B_R(x))$ for every $ p \in [1,3)$, with a quantitative estimate depending on $\sigma,p,\Lambda$, and the geometry of $M$ in $\B_{4R}(x)$. 
\end{itemize}
Moreover, by elliptic regularity, $u \in W^{2\sigma,\frac{p}{2\sigma}} (\B_R(x))$ for every $ p \in [1,3)$. 
\end{teo}

 In order to prove this theorem, we recall two definitions from \cite{GlobRegFrac} inspired by a work by Cheeger and Naber \cite{CheegerNaberQuant1}. 

We will give complete details of this local regularity result in the Euclidean setting, since the proof of the general case follows with minor modifications using the results in Appendix~\ref{app1: kernel and comparability}.

\begin{defn}[{$k$-symmetric map}] We say that a map $h \colon \R^n \to \R^d$ is $k$-symmetric if $h$ is $0$-homogeneous and it has $k$ independent directions of translation invariance. 
    
\end{defn}

\begin{defn}[Quantitative symmetry] Let $\sigma \in (1/2,1)$ and let $\Omega \subset \R^n$ be an open set. Given $u \in  H^{\sigma}(\Omega ; \Sp^1)$, $\epsilon > 0$, and a nonnegative integer $k$, we say that $u$ is $(k,\epsilon)$-symmetric in $B_r(x_0) \Subset \Omega$ if there exists a $k$-symmetric map $h \in  H^{\sigma}(B_{2r}(x_0) ; \Sp^1)$ such that
\[
\fint_{B_1^+(0,0)} |U(\mathbf{x} + r \mathbf{y} )-H(r \mathbf{y} )|^{2} \,d\mathbf{y}
=
\fint_{B_r^+(x_0,0)} |U( \mathbf{y} )-H(\mathbf{y}- \mathbf{x} )|^{2}\,d \mathbf{y}
\le \epsilon  , 
\]
where $ \mathbf{x}:=(x,z), \mathbf{y}:=(y,z) \in \R^{n+1}_+$, and $U$ and $H$ are the Caffarelli--Silvestre extensions of $u$ and $h$ respectively. 
\end{defn}

\begin{defn}[Quantitative singular set]
Let $u \in H^{\sigma}(\Omega;\Sp^1)$, $r,\eta>0$, and $k \in \{0,1,\cdots,n\}$. We define the
$k$-quantitative singular stratum $\mathcal{S}^k_{\eta,r}(u)\subset \Omega$ as
    \[
        \mathcal{S}^k_{\eta,r}(u)
        =
        \big\{
        x \in \Omega :
        u \text{ is not } (k+1,\eta)\text{-symmetric in } B_\rho(x)
        \text{ for every } \rho \in [r,1]
        \big\}.
    \]
\end{defn}

With this terminology, we can give the proof of Theorem \ref{thm: local improved W1p p<3}. 

\begin{proof}[Proof of Theorem \ref{thm: local improved W1p p<3}] The proof follows the same lines as in \cite{GlobRegFrac}, with the only new input being the stronger symmetry self-improvement (i.e. Lemma \ref{lem: sym self improvement n-2}) under the assumptions of minimality and $\star J u =0$. Let us briefly recall the argument in \cite{GlobRegFrac}. 

With no loss of generality, assume $x=0$ and $R=1$. The starting point is the quantitative stratification estimate, namely the volume estimate for the quantitative singular set \(\mathcal{S}^k_{\eta,r}(u)\) in \cite[Theorem~1.8]{GlobRegFrac}. This part is completely general and yields, for every $\eta>0$, that
\begin{equation}\label{eq: quant strat estimate}
\mathcal{H}^n \big(T_r(\mathcal{S}^k_{\eta,r}(u))\cap B_1\big)\le C r^{n-k-\eta}, \qquad \forall \, r \in (0,1) , 
\end{equation}
where $T_r(A)$ denotes an $r$-tubular neighborhood of the set $A$.

To convert this into regularity, the key input is the \(\varepsilon\)-regularity theorem \cite[Theorem~4.1]{GlobRegFrac}, which is obtained by combining the \((n,\varepsilon)\)-regularity statement of \cite[Lemma~4.2]{GlobRegFrac} with the symmetry self-improvement result \cite[Lemma~4.3]{GlobRegFrac}. More precisely, \cite[Lemma~4.3]{GlobRegFrac} shows that for every $\ep>0$ there exists $\delta>0$ such that \((n-1,\delta)\)-symmetry implies \((n,\varepsilon)\)-symmetry. Geometrically, this is because an \((n-1)\)-symmetric tangent map depends on only one variable, and in one dimension for $\sigma \in (1/2,1)$ every \(0\)-homogeneous tangent map in $H^\sigma_{\rm loc}(\R)$ is necessarily constant. 

As a consequence, \cite[Theorem~4.1]{GlobRegFrac} implies there exists a $\delta>0$ such that if $u$ is \((n-1,\delta)\)-symmetric in $B_2$ then $u$ is regular in $B_{1}$. Hence, for $\eta$ fixed suitably small, there are no points with bad regularity scales 
\begin{equation*}
     \B_r(u) :=\{x \in B_1 :  r_u(x) < r \}, \quad  r_u(x) = \max \Big \{ r \in (0,1] \, : \, \sup_{y \in B_r(x)} |\nabla u(y)| \le 1/r \Big\} , 
\end{equation*}
inside the strata $\mathcal{S}^n_{\eta, c r}(u)$ and $\mathcal{S}^{n-1}_{\eta, cr}(u)$, for every $r \in (0,1)$. In particular, all points with bad regularity scales lie in the $(n-2)$-quantitative stratum, that is $\B_r(u) \subset \mathcal{S}^{n-2}_{\eta,cr}(u)$.

Inserting \(k=n-2\) into the general volume estimate \eqref{eq: quant strat estimate} gives \cite[Theorem~1.6]{GlobRegFrac} namely
\[
\mathcal{H}^n \big(T_r(  \B_r(u) ) \big)\le C r^{2-\eta}.
\]
Since $\sing(u)\subset \B_r(u)$ for every $r \in (0,1)$, this volume bound already implies that the Minkowski dimension of $\sing(u)$ is $\le n-2$. 
Moreover, from this estimate also \(\nabla u\in L^p (B_1) \) for every $ p \in [1,2)$ immediately follows, since
\begin{equation*}
    \{x \in B_1  :  |\nabla u(x)|>1/r \} \subset  \{x \in B_1  :  r_u(x) < r \} =  \B_r(u) , 
\end{equation*}
and by the layer cake representation 
\begin{align*}
    \int_{B_1}|\nabla u|^p & = p \int_0^\infty t^{p-1} \mathcal{H}^n(\{x\in B_1 : |\nabla u| > t \}) \, dt \\ & \le  \mathcal{H}^n(B_1) + \int_1^\infty t^{p-1} \mathcal{H}^n(\{x\in B_1 : |\nabla u| > t \}) \, dt \\& \le \mathcal{H}^n(B_1) + C \int_1^\infty t^{p-1+\eta-2} dt , 
\end{align*}
which is finite provided that $\eta < 2-p$. 

Thus, the only point that must be improved in order to achieve \(L^p\) for every $ p \in [1,3)$ and Minkowski dimension $\le n-3$ is precisely \cite[Lemma~4.3]{GlobRegFrac}: one would need a stronger symmetry self-improvement statement showing that \((n-2,\delta)\)-symmetry already forces \((n,\varepsilon)\)-symmetry. Equivalently, one must rule out nontrivial \(2\)-dimensional \(0\)-homogeneous tangent maps under our additional hypothesis of minimality and $\star J u =0$. This is done in Lemma \ref{lem: sym self improvement n-2}, which we prove below, and is the substitute for \cite[Lemma 4.3]{GlobRegFrac} in our case. 

Once this is available, the same argument above gives \(   \B_r(u) \subset \mathcal{S}^{n-3}_{\eta,cr}(u)\), and then \eqref{eq: quant strat estimate} with \(k=n-3\) would immediately imply the volume estimate 
\begin{equation*}
   \mathcal{H}^n(T_r( \{x \in B_1  :  |\nabla u(x)|>1/r \})) \le  \mathcal{H}^n(T_r(  \B_r(u) ) )\le C r^{3-\eta} , 
\end{equation*}
for every $\eta$ suitably small. This implies \(\nabla u\in L^p(B_1)\) for every \(p<3\) and Minkowski dimension of the singular set $\le n-3$. 

From here, the higher order regularity follows by Lemma \ref{lem: A in W^1r}. 
\end{proof}

\begin{lemma}\label{lem: sym self improvement n-2} Let $\sigma_*$ be the constant of Theorem \ref{thm: classification R2}. Let $\sigma\in (\sigma_* ,1)$ and $u \in {H}^{\sigma}(B_2; \Sp^1)$ be a null-Jacobian minimizing $\sigma$-harmonic map. Then, for every $\varepsilon > 0$ there exists $\delta > 0$ such that if $u$ is
$(n-2,\delta)$-symmetric on $B_1$, then $u$ is also $(n,\varepsilon)$-symmetric on $B_1$.
\end{lemma}

\begin{proof}
The proof is similar to that of \cite[Lemma 4.3]{GlobRegFrac} but using our improved classification Theorem \ref{thm: classification R2} after the compactness argument.

Suppose, to the contrary, that there exists $\ep_0>0$ and a sequence of null-Jacobian minimizing $\sigma$-harmonic maps $u_k \in {H}^{\sigma}(B_2; \Sp^1) $ such that $\star J u_k =0$ and $u_k$ is $(n-2, 1/k)$-symmetric but not $(n, \ep_0)$-symmetric on $B_1$. By the compactness of null-Jacobian minimizing $\sigma$-harmonic maps Proposition \ref{prop: compactness min}, up to subsequences, $u_k  \rightharpoonup  u$ weakly in $ H^\sigma(B_2; \Sp^1)$ and $u_k  \to  u$ strongly in $ H^\sigma(B_1; \Sp^1)$ for some null-Jacobian minimizing $\sigma$-harmonic map $u \in H^\sigma(B_2; \Sp^1)$.

Moreover, the limit map $u$ is $0$-homogeneous, it satisfies $\star J u = 0$ in $B_1$ by Theorem \ref{teo:BN-Invent fractional}, and it is $(n-2, 0)$-symmetric but not $(n, \ep_0)$-symmetric in $B_1$. 

Since $u$ has $(n-2)$ independent directions of translation invariance, by Lemma \ref{lem: dim reduction minimality} it can be written (up to a rotation) as $u(x)=v(x_1, x_2)$ where $v \in H^\sigma_{\rm loc}(\R^2; \Sp^1)$ is a null-Jacobian minimizing $\sigma$-harmonic map. Clearly $\star J v =0$, otherwise we would get $\star J u \neq 0$ in $B_1$, which is a contradiction. 

Hence $v \in H^\sigma_{\rm loc}(\R^2; \Sp^1)$ is a $0$-homogeneous, null-Jacobian minimizing $\sigma$-harmonic map with $\star J v = 0 $. By Theorem \ref{thm: classification R2}, we get that $v$ is constant. Hence, also $u$ is constant, and this contradicts that $u$ is not $(n, \ep_0)$-symmetric in $B_1$.
\end{proof}

\section{The fractional $s$-mass}\label{sec: frac s mass}

In \cite{codim2-frac} we introduced and studied a fractional notion of area for codimension-two
surfaces in $\R^n$ (or on closed manifolds), which we call the \emph{fractional $s$-mass}
and denote by $\M_{s}$. In analogy with the fractional perimeter, $\M_{s}$ provides a
fractional (and nonlocal) counterpart to the classical $(n-2)$-dimensional Hausdorff measure with multiplicity.

In the present work, we adopt a formulation of $\M_{s}$ based on prescribing the distributional Jacobian of an $\Sp^1$-valued map, in the sense described in Section \ref{sbs: factor Jac}. This is slightly different from what we did in \cite{codim2-frac}; we will prove in Subsection \ref{sbs: equivalence} that these two formulations are completely equivalent.

Let $\Sigma$ be an $(n-2)$-dimensional admissible boundary on $M$. We define the $s$-mass of $\Sigma$ as 
\begin{equation}\label{eq: M_s and F^J def}
      \mathbb{M}_{s}(\Sigma) := \inf_{u \in \mathfrak{F}_s ^J(\Sigma)}  [u]^2_{H^{\frac{1+s}{2}}(M)} , \quad \mbox{where } \, \mathfrak{F}_s ^J(\Sigma):=\Big\{ u \in H^{\frac{1+s}{2}}(M; \Sp^1) \, \colon  {\star Ju}= \pi \Sigma \Big\}
\end{equation}
and $[u]^2_{H^{\frac{1+s}{2}}(M)}$ is the fractional Sobolev seminorm given in \eqref{eq:Hs-kernel}.

\subsection{Equivalence with the weak-linking definition}\label{sbs: equivalence}

In our previous work \cite{codim2-frac} that focused on \emph{smooth} $(n-2)$-dimensional surfaces $\Sigma$ with integer multiplicity, we defined the $s$-mass by minimization of the fractional seminorm \eqref{eq:Hs-kernel} on a different class of maps $\mathfrak{F}_{\hspace{-1pt} s}^{\hspace{.9pt} w}  (\Sigma)$ with a weak-linking condition around $\Sigma$. More precisely, let
\begin{equation}\label{eq:sigma smooth}
\parbox{0.8\linewidth}{\centering
$\Sigma\subset M$ be a $C^{2}$, closed, oriented, $(n-2)$-dimensional boundary with locally constant integer multiplicity.
}
\end{equation}
That is, $\Sigma := d_1 \Sigma_1 \cup \dots \cup d_m \Sigma_m$ where $m\geq 1$, $d=(d_1,\dots,d_m)\in \N_+^m$, and $\Sigma_1,\dots,\Sigma_m$ are pairwise disjoint, closed, connected, oriented, $(n-2)$-dimensional boundaries.

In \cite{codim2-frac}, we minimized over the class of weak-linking maps 
\begin{equation}\label{defn:Fweak}
        \mathfrak{F}_{\hspace{-1pt} s}^{\hspace{.9pt} w}  (\Sigma):=\biggl\{ u \in H^{\frac{1+s}{2}}(M; \Sp^1) \colon \exists \{u_k\}\subset \mathfrak{F}_{\hspace{-.7pt} s}  (\Sigma),\ u_k\to u \sms \mbox{ in } H^{\alpha}(M;\Sp^1) \ \forall \alpha \in \biggl( \hspace{-0.05cm} 0, \frac{1+s}{2}\biggr) \biggr\} , 
    \end{equation}
where (up to a sign)
\begin{multline}
    \mathfrak{F}_{\hspace{-.7pt} s}(\Sigma)  := \Big\{ u \in C^1(M \setminus \Sigma; \Sp^1) \cap H^{\frac{1+s}{2}}(M; \Sp^1) \colon \\
    \deg(u, \gamma) = d_1\link(\gamma,\Sigma_1)+\dots+d_m \link(\gamma,\Sigma_m) \ \text{for any} \ \gamma:\Sp^1\to M\setminus \Sigma \Big\}, \label{eq: F old def}
\end{multline}

In the present work, we minimize instead on the class
$$ \mathfrak{F}_s ^J(\Sigma):=\Big\{ u \in H^{\frac{1+s}{2}}(M; \Sp^1) \, \colon  {\star Ju}= \pi \Sigma \Big\},$$
so we define $\M_{s}(\Sigma)$ by prescribing the distributional Jacobian. A priori, these two classes do not need to coincide, since they encode singularities in different ways. The following lemma shows that, for smooth $\Sigma$, they coincide.

\begin{lemma}\label{lem: equivalence linking-Jacobian} Let $\Sigma$ be smooth (i.e., as in \eqref{eq:sigma smooth}) and $s\in (0,1)$. Then 
\[ 
\mathfrak{F}_{\hspace{-1pt} s}^{\hspace{.9pt} w}  (\Sigma) \equiv \mathfrak{F}_s ^J(\Sigma).
\]
In particular, the $s$-mass defined via weak linking and the one defined via prescribed Jacobian coincide.
\end{lemma}
\begin{proof}
First of all, for $\mathfrak{F}_{\hspace{-.7pt } s}(\Sigma)$ given by \eqref{eq: F old def}, we claim that
\begin{equation*}
\mathfrak{F}_{\hspace{-.7pt } s}(\Sigma)  = \Big\{ u \in C^1(M \setminus \Sigma; \Sp^1) \cap H^{\frac{1+s}{2}}(M; \Sp^1)  \, \colon {\star Ju} = \pi \Sigma \Big\} .
\end{equation*}

Indeed, it follows from Step~15 in the proof of \cite[Theorem~7.1]{Bre-Mi} that the two factors in (\ref{eq: Mir factorization}) are continuous where $u$ is continuous, so the characterization of the Jacobian in terms of the degree for maps in $ C^1(M \setminus \Sigma; \Sp^1) \cap H^{\frac{1+s}{2}}(M; \Sp^1)$ follows from the corresponding statement for maps in $ C^1(M \setminus \Sigma; \Sp^1) \cap W^{1,1+s}(M; \Sp^1)$, that is \cite[Theorem~3.14]{Bre-Mi}.

Now we show the two inclusions separately.

\vsp 
\noindent 
$ \bullet \,\, \mathfrak{F}_{\hspace{-1pt} s}^{\hspace{.9pt} w}  (\Sigma) \subseteq \mathfrak{F}_s ^J(\Sigma)$. Let $u \in \mathfrak{F}_{\hspace{-1pt} s}^{\hspace{.9pt} w}(\Sigma)$. By definition, there exists a sequence $u_k \in C^1(M \setminus \Sigma;\Sp^1)$ such that $\star J u_k = \pi \Sigma$ and $u_k \to u$ in $H^{1/2}(M;\Sp^1)$. By Theorem \ref{teo:BN-Invent fractional} the map $u \mapsto \star Ju$ is continuous from $H^{1/2}(M;\Sp^1)$ into $(n-2)$-currents endowed with the flat topology, and therefore $\star Ju = \lim_{k \to \infty} \star J u_k = \pi \Sigma$, so $u\in \mathfrak{F}_s ^J(\Sigma)$.

\vsp 
\noindent 
$  \bullet  \,\, \mathfrak{F}_s ^J(\Sigma) \subseteq   \mathfrak{F}_{\hspace{-1pt} s}^{\hspace{.9pt} w}  (\Sigma)$. Let $u \in \mathfrak{F}_s ^J(\Sigma)$. By \cite[Proposition~3.1]{codim2-frac} (see also Step 2 in the proof of the $\Gamma$-limsup, Section \ref{sec: gamma limsup currents}) there exists a standard-vortex $u_\circ$ around $\Sigma$ such that 
\begin{equation*}
    u_\circ \in C^1(M \setminus \Sigma ; \Sp^1) \cap  H^{\frac{1+s}{2}}(M ; \Sp^1)\cap W^{1,1}(M;\Sp^1), \qquad \star Ju_\circ = \pi \Sigma.
\end{equation*}
Consider the quotient $u u_\circ^{-1}$. Then, since $H^{\frac{1+s}{2}} \cap L^\infty$ is an algebra, the map $u u_\circ^{-1}$ belongs to $H^{\frac{1+s}{2}}(M; \Sp^1)$ and satisfies
\begin{equation*}
    \star J(u u_\circ^{-1}) = \star J u - \star J u_\circ = \pi \Sigma - \pi \Sigma =0 . 
\end{equation*}
By Theorem \ref{thm: density on M} applied to $u u_\circ^{-1}$ there exists a smooth sequence $u_k \in C^{\infty}(M; \Sp^1)$ such that $u_k \to u u_\circ^{-1}$ in $H^{\frac{1+s}{2}}(M; \Sp^1)$ as $k \to + \infty$. In particular
\begin{equation*}
    u_k u_\circ \in \mathfrak{F}_{\hspace{-1pt} s}^{\hspace{.9pt}}  (\Sigma) \quad \mbox{and} \quad u_k u_\circ \to u \s \mbox{in } H^{\frac{1+s}{2}}(M; \Sp^1) . 
\end{equation*}
This implies that $u \in \mathfrak{F}_{\hspace{-1pt} s}^{\hspace{.9pt} w}  (\Sigma)$ by the very definition of $\mathfrak{F}_{\hspace{-1pt} s}^{\hspace{.9pt} w}  (\Sigma)$. 
\end{proof}

In the proof of the second inclusion above, we have also proved the following. 

\begin{cor}\label{cor: equivalence inf on smooth} Let $\Sigma$ be smooth (precisely, as in \eqref{eq:sigma smooth}) and $s\in (0,1)$. Then
\begin{equation*}
\mathfrak{F}_s ^J(\Sigma) \, \mbox{ is the closure of } \, \mathfrak{F}_{\hspace{-.7pt } s}(\Sigma) \,  \mbox{ in } \, H^\frac{1+s}{2}(M; \Sp^1). 
\end{equation*}  
In particular, the $s$-mass coincides with the one given by taking the infimum over $\mathfrak{F}_{\hspace{-1pt} s}^{\hspace{.9pt}}  (\Sigma)$. 
\end{cor}

\subsection{Further properties at fixed $s$}

\begin{prop}[Coercivity and semicontinuity of $\M_{s}$]\label{prop: fixed s coerc and lsc}
Let $\{\Sigma_k\}$ be a sequence of $(n-2)$-boundaries in $M$ and assume there exists $s\in (0,1)$ such that
$$\limsup_{k\to \infty} \M_{s}(\Sigma_k)<+\infty.$$

Then there exists an $(n-2)$-boundary $\Sigma$ in $M$ such that (up to a subsequence) $\Sigma_k\to \Sigma$ in the flat sense and
$$\liminf_{k\to \infty} \M_{s}(\Sigma_k)\geq \M_{s}(\Sigma).$$
\end{prop}
\begin{proof}
For every $k\in\N$ let $u_k\in \mathcal{A}(\Sigma_k)$ be a minimizer in the definition of $\M_{s}(\Sigma_k)$, so that $\star J u_k =\pi \Sigma_k$ and
$$\limsup_{k\to \infty} \,  [u_k]_{H^\frac{1+s}{2}(M)}<+\infty.$$

Therefore, there exists a map $u\in H^{\frac{1+s}{2}}(M;\Sp^1)$ and a subsequence (not relabeled) such that $u_k\rightharpoonup u$ weakly in $H^{\frac{1+s}{2}}(M;\Sp^1)$ and strongly in $H^{1/2}(M;\Sp^1)$. Thus, setting $\Sigma:=\star Ju/\pi$, Theorem~\ref{teo:BN-Invent fractional} and the weak lower semicontinuity of the fractional seminorm imply that
$$\lim_{k\to \infty}\F(\Sigma-\Sigma_k)=\lim_{k\to \infty}\frac{1}{\pi}\F(\star Ju_k - \star Ju) =0, $$
and 
$$\liminf_{k\to \infty} \M_{s}(\Sigma_k)= \liminf_{k\to \infty} \, [u_k]_{H^\frac{1+s}{2}(M)}^2 \geq [u]_{H^\frac{1+s}{2}(M)}^2 \geq \M_{s}(\Sigma).$$
\end{proof}

\begin{prop}[Fractional isoperimetric inequality]\label{prop: frac iso}
There exists a constant $C_{s}>0$ such that for every admissible $(n-2)$-boundary $\Sigma$ in $M$ there exists an integral $(n-1)$-current $S$ such that $\Sigma= \partial S$ and
\begin{equation*}
    \M(S) \leq C_s |M|^{\frac{s}{s+1}} \M_{s}(\Sigma)^{\frac{1}{1+s}} . 
\end{equation*}
Moreover, $S$ can be taken to be a regular level set $S=v^{-1}( \theta )$ (in the sense of \cite[Section 7.5]{ABO2-singularities}) of some map $v\in W^{1,1+s}(M; \Sp^1)$ with $\star Jv = \pi (\partial S )= \pi \Sigma$. 
\end{prop}

\begin{proof}
    Let $u_\Sigma \in H^{\frac{1+s}{2}}(M; \Sp^1)$ be a minimizer in the definition of $\M_{s}(\Sigma)$, and $u_\Sigma = v e^{i\varphi}$ be the factorization \eqref{eq: Mir factorization}. By \cite[Theorem 3.8]{ABO2-singularities} applied to $v$, there exists a regular level set $S_\theta := v^{-1}( \theta )$ of $v$ such that 
    \begin{equation*}
        \star J v = \pi (\partial S_\theta )  \qquad \mbox{and}  \qquad \M(S_\theta) \le C \int_{M} |\nabla v| . 
    \end{equation*}
    Since $M$ is compact, by Hölder's inequality
    \begin{equation*}
\M(S_\theta) \leq C |M|^{\frac{s}{s+1}} \|\nabla v\|_{L^{1+s}(M)} \le C_s |M|^{\frac{s}{s+1}}  [u_\Sigma]_{H^{\frac{1+s}{2}}(M)}^{\frac{2}{1+s}} =  C_s |M|^{\frac{s}{s+1}} \M_{s}(\Sigma)^{\frac{1}{1+s}}  , 
\end{equation*}
where the second inequality follows from~\cite[equation~(7.5)]{Bre-Mi}.
\end{proof}

Proposition~\ref{prop: frac iso} is to be understood as a genuinely nonlocal isoperimetric inequality, with no direct counterpart for the classical mass.
Indeed, if one attempts to replace $\M_s(\Sigma)$ by $\M(\Sigma)$, then no estimate of the same type can hold uniformly with respect to the multiplicity $m\in \Z \setminus\{0\}$ of $\Sigma$. This is because the left-hand side would grow as $\sim |m|$ while the right-hand side would grow as $\sim |m|^{1/2}$.

\subsubsection{The domain of the $s$-mass}\label{sbs: domain of the s-mass} We collect a few remarks on which $(n-2)$-currents have finite $s$-mass.

\begin{itemize}
    \item[$(i)$] \textit{Characterization for fixed $s\in (0,1)$.} By the Gagliardo--Nirenberg inequality, if there exists a map $u \in W^{1,1+s}(M; \Sp^1)$ such that $\star Ju = \pi \Sigma$, then $\M_{s}(\Sigma)<+\infty$. Conversely, by the factorization theorem in \cite[Theorem 7.1]{Bre-Mi}, if $\M_{s}(\Sigma)<+\infty$ then there exists a map $v \in W^{1,1+s}(M; \Sp^1)$ with $\star Jv = \pi \Sigma$. In particular, 
\begin{equation*}
     \big\{\Sigma \, : \, \M_{s}(\Sigma) <+\infty \big\} = \Big\{  \tfrac{1}{\pi} {\star Ju} \, : \, u\in W^{1,1+s}(M; \Sp^1)\Big\} . 
\end{equation*}

Moreover, in \cite{Bousquet} (see also \cite[Theorem 3.18]{Bre-Mi}), Bousquet provided a complete characterization of the Jacobians of maps $u\in W^{1,p}(M;\Sp^1)$ for $p\in(1,2)$. In our setting, this characterization states that 
\begin{equation}\label{Bousquet_char}
\Big\{  \tfrac{1}{\pi} {\star Ju} \, : \, u\in W^{1,1+s}(M; \Sp^1)\Big\} 
= \Big\{  \tfrac{1}{\pi} {\star Ju} \, : \, u\in W^{1,1}(M; \Sp^1)\Big\} 
\cap \big( W^{1,\frac{1+s}{s}}(M; \Lambda^{n-2})\big)^* .
\end{equation}
If $\Sigma$ is an integral $(n-2)$-boundary with finite mass, \cite[Theorem 4.4]{ABO2-singularities} ensures that $\Sigma=\frac{1}{\pi}\star Jv$ for some $v\in W^{1,1}(M;\Sp^1)$. Hence, for integral boundaries with finite mass, acting continuously on $(n-2)$ forms of class $W^{1,\frac{1+s}{s}}$ is sufficient. However, this condition is still quite abstract.

\item[$(ii)$] \textit{Examples.} 
\begin{itemize}
\item \emph{Polyhedral boundaries.} By \cite[Proposition~5.8]{ABO2-singularities}, every $(n-2)$-dimensional polyhedral boundary
is the Jacobian of a map in $W^{1,p}(M;\Sp^1)$ for every $p<2$. Hence it has finite $s$-mass for every $s\in(0,1)$.
\item \emph{Smooth boundaries.} By \cite{Bousquet} (see also \cite[Proposition~3.1]{codim2-frac}),
every finite union of smooth, oriented, connected $(n-2)$-dimensional boundaries has finite $s$-mass.
\item \emph{Integral boundaries with bounded density.} More generally, any $(n-2)$-dimensional integral boundary $T$ such that $\sup\{\|T\|(B(x,r))/r^{n-2}:x\in M, \, r>0\}<+\infty$ has finite $s$-mass for every $s\in (0,1)$. This follows from (\ref{Bousquet_char}) and \cite[Theorem~5]{Traces-rectif}.
\end{itemize}

\item[$(iii)$] \textit{Non-example: finite mass does not imply finite $s$-mass.} There exists an integral, unit-density, smooth $(n-2)$-current $T$ with $\M(T)<+\infty$ but $\M_{s}(T)=+\infty$
for every $s\in(1/(n-1),1)$; see \cite[Lemma~15.40]{Bre-Mi}.
In particular
\[
\big\{ \Sigma \text{ integral boundary} \, : \, \M(\Sigma)<+\infty \big\} \not\subset \big\{\Sigma \, : \, \M_{s}(\Sigma) <+\infty \big\} . 
\]

\item[$(iv)$] \textit{The normalized regime as $s \to 1$.} Consider the class of admissible boundaries $\Sigma$ such that
\begin{equation}\label{eq: rem domain mass 4th point}
\liminf_{s \to 1} \, (1-s)^2 \M_{s}(\Sigma)<+\infty.
\end{equation}
By Corollary~\ref{cor: bdd rescled mass -> integral}, every such $\Sigma$ is integral.
But this class is strictly smaller than the set of all integral boundaries: indeed, the current $T$ from $(iii)$ is an integral boundary but satisfies
$\M_{s}(T)=+\infty$ for all $s$ close to $1$, hence it does not satisfy \eqref{eq: rem domain mass 4th point}.

\end{itemize}

\subsection{Proof of the envelope formula for the first variation of $\M_s$}

In this section we prove Proposition~\ref{lem:envelope-Ms}. For a smooth vector field $X$ on $M$, we denote by $\phi_t^X \colon M \to M $ its integral flow at time $t\in \R $. We will need the following result \cite[Lemma 3.10]{CFSfrac}, about the control of the $H^\sigma$ seminorm under inner variations.

\begin{lemma}\label{lemma: energy under flow}
Let $\sigma \in (0,1)$ and $ v\in H^{\sigma}(M)$ be such that $|v|\le 1$. Let $X$ be a smooth vector field on $M$ and $v_t := v \circ \phi_{-t}^X$. Then, for all $T>0$ and integer $k \ge 0$ there holds
\begin{equation*}
    \sup_{t\in (-T,T)} \bigg| \frac{d^k}{d  t^k} [v \circ \phi^X_{-t}]^2_{H^\sigma(M)} \bigg|\leq C\big( 1 + [v ]^2_{H^\sigma(M)}  \big)\, ,
\end{equation*}
for some constant $C=C(M,\sigma,k,T, \| X\|_{C^k(M)} )>0$.
\end{lemma}

\begin{lemma}\label{lem: conv first variation}
    Let $\sigma \in (0,1)$, $u\in H^{\sigma}(M; \R^m)$, and let $X$ be a smooth vector field on $M$. Then  
    \begin{equation*}
        u_k \to u \,\, \mbox{ in } H^\sigma(M) \quad \implies \quad \frac{d}{dt}\bigg|_{t=0}[u_k \circ\phi_{-t}^X]^2_{H^\sigma(M)} \to \frac{d}{dt} \bigg|_{t=0}[u\circ\phi_{-t}^X]^2_{H^\sigma(M)} . 
    \end{equation*}
\end{lemma}
\begin{proof}
    The fact that the derivatives in the first inner variations exist follows from Lemma \ref{lemma: energy under flow}; we just need to justify continuity with respect to the strong convergence in $H^\sigma$.

It is well known that such inner variation can be expressed in terms of the Caffarelli--Silvestre extension and the relative weighted stress-energy tensor in the extension;  see \cite[Prop.~2.15]{Partialreg} or \cite[Lemma 4.16]{MillotSire}. If \(U_k\) and \(U\) denote the component-wise Caffarelli--Silvestre extensions of \(u_k\) and \(u\) respectively, this formula yields (up to a constant)

\begin{align}
& \frac{d}{dt}\bigg|_{t=0}  [u_k \circ\phi_{-t}^X]^2_{H^\sigma(M)}
 \nonumber \\ & = \int_{\widetilde M}
\Big(
|\widetilde\nabla U_k |^2\,\widetilde{\dive}\mathbf X
-
2 \sum_{i,j=1}^{n+1}(\partial_iU_k \cdot\partial_j U_k) \partial_j\mathbf X_i
\Big) z^{1-2\sigma} +  (1-2\sigma) \int_{\widetilde M} z^{-2\sigma} |\widetilde\nabla U_k|^2 \mathbf X_{n+1} , \label{eq: fvar Caf extension} 
\end{align}
where $\widetilde M := M \times (0,\infty)$ with the product metric, and $\mathbf X =(\mathbf X_1, \dotsc, \mathbf X_{n+1}) \in C^1_c(\widetilde M)$ is any smooth vector field with $\mathbf X(x,0)  =(X(x),0)$. Choosing \(\mathbf X\) of the form \(\mathbf X(x,z)=(X(x), z\eta(z))\) where \(\eta \in C^1_c([0, \infty))\) is such that \(\eta \equiv 1\) in a neighborhood of $z=0$, all the coefficients in the integral above are bounded and compactly supported in the \(z\)-variable. Moreover, the strong convergence \(u_k\to u\) in \(H^\sigma(M)\) implies
\[
\widetilde\nabla U_k \to \widetilde\nabla U
\quad\mbox{strongly in } \, L^2(\widetilde M, z^{1-2\sigma} dV(x)dz).
\]
Indeed, this follows from the characterization of the \(H^\sigma(M)\) seminorm in terms of the weighted $H^1(\widetilde M)$ energy of the extension (see \cite{CFSfrac}). Since \eqref{eq: fvar Caf extension} is a (weighted) quadratic expression in \(\widetilde\nabla U\) with bounded coefficients, we may pass it to the limit as $k \to \infty $ and conclude the convergence of the first variations.
\end{proof}

\begin{proof}[Proof of Proposition~\ref{lem:envelope-Ms}] For brevity we set $\Sigma_t := \phi_t^X(\Sigma)$ and $\sigma=(1+s)/2$. Since for every $u\in\mathfrak{F}_s^J(\Sigma)$ we have
\begin{equation*}
        \star J (u \circ \phi^X_{-t}) = (\phi^X_{t})_\sharp (\star J u ) , 
\end{equation*}
by Lemma \ref{lemma: energy under flow} (applied with $k=0$) for small $t\le t_0$ the composition with $\phi_{-t}^X$ yields a bijection between the admissible classes
\begin{equation*}
u\in\mathfrak{F}_s^J(\Sigma)
\quad\Longleftrightarrow\quad
u\circ\phi_{-t}^X \in\mathfrak{F}_s^J(\Sigma_t).
\end{equation*}
Hence, for every minimizer $v_t \in \mathcal{A}(\Sigma_t)$ and fixed $v_0 \in \mathcal{A}(\Sigma)$ we have
\begin{equation}\label{bound_vt}
    \sup_{|t|\le t_0} \M_s(\Sigma_t) = \sup_{|t|\le t_0} [v_t]^2_{H^\sigma(M)} \le \sup_{|t|\le t_0} [v_0 \circ\phi_{-t}^X ]^2_{H^\sigma(M)} \le C(1+[v_0]^2_{H^\sigma(M)})=:C_0 , 
\end{equation}
and setting $f(u,t):=[u\circ\phi_{-t}^X]^2_{H^\sigma(M)}$ we can write (for $|t|\le t_0$)
\begin{equation*}
\M_s(\Sigma_t) = \inf_{u\in\mathfrak{F}_s^J(\Sigma)} f(u,t) = \inf_{u\in K } f(u,t) , 
\end{equation*}
where 
\begin{equation*}
    K:= \big\{u\in H^\sigma(M; \Sp^1) \, : \, [u]^2_{H^\sigma(M)}\le C_0,\, \star Ju = \pi \Sigma \big\} . 
\end{equation*}
Now we aim to apply a standard envelope-type theorem like \cite[Corollary 4 (ii)]{MilgromSegal} to conclude \eqref{eq: right der}-\eqref{eq: left der}.

Unfortunately, we cannot apply \cite[Corollary 4 (ii)]{MilgromSegal} directly, as its assumptions are not satisfied in our setting. However, we can still exploit its proof with minor adaptations.

To this end, we first observe that for every choice of a minimizer $u\in \mathcal{A}(\Sigma)$ it holds that $\M_s(\Sigma_t)\leq f(u,t)$ and $\M_s(\Sigma) = f(u,0)$, which yields
$$\limsup_{t\to 0^+} \frac{\M_s(\Sigma_t)-\M_s(\Sigma)}{t} \leq \lim_{t\to 0^+} \frac{f(u,t)-f(u,0)}{t}=\frac{d}{dt}\bigg|_{t=0^+} [u\circ\phi_{-t}^X]^2_{H^\sigma(M)},$$
and hence also
\begin{equation}\label{est_right-deriv}
\limsup_{t\to 0^+} \frac{\M_s(\Sigma_t)-\M_s(\Sigma)}{t} \leq \inf_{u\in\mathcal{A}(\Sigma)} \frac{d}{dt}\bigg|_{t=0^+} [u\circ\phi_{-t}^X]^2_{H^\sigma(M)}
\end{equation}

In order to prove the opposite inequality, we first note that Lemma~\ref{lemma: energy under flow} (applied with $k=2$) implies that the family $\{f(u, \cdot)\}_{u\in K}$ is equidifferentiable for $|t|\le t_0$, that is:
\begin{equation*}
    \forall \, t \in (-t_0, t_0) \mbox{ the ratio } \frac{f(u, t')-f(u,t)}{t'-t} \mbox{ converges uniformly for } u\in K, \mbox{ as } t'\to t. 
\end{equation*}
Thus, from \cite[Theorem 3]{MilgromSegal} we deduce
$$\frac{d}{dt}\bigg|_{t=0^+} \M_s(\phi_t^X(\Sigma))=\lim_{\tau\to 0^+} \frac{\partial f}{\partial t} (v_\tau,0)$$
for every choice of minimizers $v_\tau\in \mathcal{A}(\Sigma_\tau)$, and both the derivative and the limit exist.

Now we claim that, up to a subsequence, $v_\tau$ converges strongly in $H^\sigma(M;\Sp^1)$ to some minimizer $v_0\in \mathcal{A}(\Sigma)$. Indeed, \eqref{bound_vt} yields weak compactness of $\{v_\tau\}$ in $H^\sigma(M;\Sp^1)$ and, by lower semicontinuity of the seminorm, any weak limit $v_0$ must be a minimizer in $\mathcal{A}(\Sigma)$. Moreover, since $\tau \mapsto \M_s(\Sigma_\tau)$ is continuous, we also have that
$$\lim_{\tau\to 0^+} [v_\tau]_{H^\sigma(M)}^2 =\lim_{\tau\to 0^+} \M_s(\Sigma_\tau) = \M_s(\Sigma) = [v_0]_{H^\sigma(M)}^2 ,$$
so the convergence is actually strong in $H^\sigma$.

By Lemma \ref{lem: conv first variation} we know that $\partial f/\partial t (\cdot,0)$ is continuous with respect to the strong $H^\sigma$ convergence, and hence 
$$\frac{d}{dt}\bigg|_{t=0^+} \M_s(\phi_t^X(\Sigma))= \frac{\partial f}{\partial t} (v_0,0)=\frac{d}{dt}\bigg|_{t=0^+} [v_0\circ\phi_{-t}^X]^2_{H^\sigma(M)}.$$

Together with \eqref{est_right-deriv}, this proves \eqref{eq: right der}, and also that the infimum is attained. The proof of \eqref{eq: left der} is similar.

\end{proof}

\begin{rem} Assume that $\Sigma$ minimizes the $s$-mass $\M_s$ in a domain $\Omega$, that is 
\begin{equation*}
    \M_s (\Sigma) \le \M_s (\Sigma') \quad \mbox{for every } \Sigma' \mbox{ such that } \Sigma \cap \Omega^c = \Sigma' \cap \Omega^c .
\end{equation*}
Then, for every fixed smooth vector field \(X\) with ${\rm spt} X \subset \Omega$, \(t=0\) is a local minimum of \(f(t) := \M_s (\phi_t^X(\Sigma))\). Hence, the right derivative $\partial_+ f (0) $ is nonnegative and the left derivative $\partial_- f (0) $ is nonpositive. Moreover, the envelope formulas \eqref{eq: right der} and \eqref{eq: left der} imply that always (i.e. even for non-minimizers)
\begin{equation*}
    \partial_+ f (0) \coloneqq  \frac{d}{dt}\bigg|_{t=0^+} \M_s (\phi_t^X(\Sigma)) \le \frac{d}{dt}\bigg|_{t=0^-} \M_s (\phi_t^X(\Sigma)) \eqcolon \partial_- f (0) , 
\end{equation*}
which gives 
\begin{equation*}
    0 \le \partial_+ f (0) \le \partial_- f (0) \le  0 . 
\end{equation*}
Hence both one-sided derivatives exist and vanish, and therefore \( f'(0)=\frac{d}{dt}\big|_{t=0}\M_s (\phi_t^X(\Sigma))\) exists and is equal to \(0\).
\end{rem}

\section{Equi-coercivity and $\Gamma$-convergence on currents}\label{sec: gamma comp}

\subsection[Equi-coercivity of the $s$-mass: proof of $(i)$ in Theorem \ref{thm: new main asymptotics}]{Equi-coercivity of the $s$-mass: proof of $(i)$ in Theorem \ref{thm: new main asymptotics}} \label{subsec:5-1}

In this section we prove the equi-coercivity for the $s$-mass, relying on the corresponding result for the Ginzburg--Landau functionals proved by Alberti--Baldo--Orlandi \cite{2005-Indiana-ABO}.
We will use the following elementary inequalities
\begin{alignat}{2}
& (1-e^{-\lambda})^2  \le \lambda^\alpha, 
&\qquad \forall \lambda>0,\ \alpha\in(0,2], 
\label{eq: elem ineq 1}\\
& \lambda e^{-2\lambda t}  \le C\, t^{\sigma-1}\lambda^\sigma,
&\qquad \forall \lambda>0,\ \sigma\in[0,1],\ t>0 .
\label{eq: elem ineq 2}
\end{alignat}

\begin{proof}[Proof of {\it (i)} in Theorem \ref{thm: new main asymptotics}] Let $\sigma := \frac{1+s}{2}\in(1/2,1)$, and let $u_\sigma\in H^\sigma(M;\Sp^1)$ be a minimizer in the definition of $\M_{s}(\Sigma_s)$. In particular, by hypothesis
\begin{equation}\label{eq: bound limsup u}
    \limsup_{\sigma \to 1 } \, (1-\sigma)^2 [u_\sigma]_{H^\sigma(M)}^2 =  \limsup_{s \to 1} \, \frac{1}{4}(1-s)^2 \M_{s}(\Sigma_s) <+\infty .
\end{equation}

For $t>0$, denote by
\[
P_t u_\sigma \in C^\infty(M;\R^2)
\]
the (component-wise) solution of the heat equation with initial datum $u_\sigma$. Our goal is to estimate the Ginzburg--Landau energy of these maps. Since $|u_\sigma|=1$, by the maximum principle we have $|P_t u_\sigma|\le 1$ on $M$. We obtain
\begin{align*}
\frac{1}{4}\int_M (1-|P_t u_\sigma|^2)^2
&= \frac{1}{4}\int_M (|u_\sigma|^2-|P_t u_\sigma|^2)^2 \\
&= \frac{1}{4}\int_M \big((u_\sigma-P_t u_\sigma)\cdot(u_\sigma+P_t u_\sigma)\big)^2 \\
&\le \int_M |u_\sigma-P_t u_\sigma|^2 .
\end{align*}
We will repeatedly use the spectral representation in Proposition \ref{prop: spectral characterization}. Writing 
\begin{equation*}
    u_\sigma = \sum_{k\ge 0} a_k \varphi_k , 
\end{equation*}
we have
\[
P_t u_\sigma = \sum_{k\ge 0} a_k e^{-\lambda_k t}\varphi_k,
\quad \mbox{and} \quad
u_\sigma-P_t u_\sigma = \sum_{k\ge 0} a_k(1-e^{-\lambda_k t})\varphi_k .
\]
Hence
\[
\int_M |u_\sigma-P_t u_\sigma|^2 
= \sum_{k\ge 0} |a_k|^2 (1-e^{-\lambda_k t})^2 .
\]
Using the elementary inequality \eqref{eq: elem ineq 1} we can estimate
\[
\int_M |u_\sigma-P_t u_\sigma|^2 
\le t^\sigma \sum_{k\ge 0} |a_k|^2 \lambda_k^\sigma
= \frac{\alpha_{n,\sigma}}{2} t^\sigma [u_\sigma]_{H^\sigma(M)}^2 ,
\]
where we used the spectral representation of Proposition \ref{prop: spectral characterization}. 

Similarly
\[
\int_M |\nabla (P_t u_\sigma)|^2
= \sum_{k\ge 0} |a_k|^2 \lambda_k e^{-2\lambda_k t} , 
\]
and by the elementary inequality \eqref{eq: elem ineq 2} we get
\[
\int_M |\nabla (P_t u_\sigma)|^2 
\le C t^{\sigma-1} \sum_{k\ge 0} |a_k|^2 \lambda_k^\sigma
= C \alpha_{n,\sigma} t^{\sigma-1} [u_\sigma]_{H^\sigma(M)}^2 .
\]

By the definition of the Ginzburg--Landau energy and \eqref{eq: alpha-est}, we have
\begin{align*}
{GL}_{\sqrt{t}}^n(P_t u_\sigma, M)
& := \frac{1}{2}\int_M |\nabla (P_t u_\sigma)|^2
   + \frac{1}{4t}\int_M (1-|P_t u_\sigma|^2)^2 \le C (1-\sigma)\, t^{\sigma-1} [u_\sigma]_{H^\sigma(M)}^2 ,
\end{align*}
and therefore
\begin{equation}\label{eq: GL-bound-sigma}
\frac{{GL}_{\sqrt{t}}^n(P_t u_\sigma,M)}{|\log \sqrt{t} |}
\le C \frac{(1-\sigma) t^{\sigma-1}}{|\log \sqrt{t} |}[u_\sigma]_{H^\sigma(M)}^2 .
\end{equation}

We now choose
\[
t=t(\sigma):= e^{-\frac{2}{1-\sigma}} \to 0 
\quad\text{as }\sigma\to 1.
\]
Then $|\log(\sqrt{t(\sigma)})| = \frac{1}{1-\sigma}$ and
\[
t(\sigma)^{\sigma-1}
= e^{-\frac{2(\sigma-1)}{1-\sigma}} = e^2.
\]
Plugging this into \eqref{eq: GL-bound-sigma} we obtain
\begin{equation}\label{eq: en bound to optimize}
\frac{{GL}_{\sqrt{t(\sigma)}}^n(P_{t(\sigma)} u_\sigma, M)}{|\log \sqrt{t(\sigma)}|}
\le C  (1-\sigma)^2 [u_\sigma]_{H^\sigma(M)}^2.
\end{equation}
By hypothesis \eqref{eq: bound limsup u}, the right-hand side is uniformly bounded as $\sigma \to 1$. Hence, by Theorem~\ref{thm: ABO comp}, we get that (up to subsequences) the Jacobians $\star J(P_{t(\sigma)} u_\sigma )$ converge to an integral $(n-2)$-boundary in the flat topology of $M$. 

We are left to show that the flat limits of $\star J u_\sigma$ and $\star J(P_{t(\sigma)} u_\sigma)$ coincide as $\sigma\to 1$. Using again the spectral representation and \eqref{eq: elem ineq 1} applied with $\lambda=\lambda_k t$ and $\alpha = \sigma-1/2$, we get
\begin{align*}
     [ u_\sigma  -P_t u_\sigma]_{H^{1/2}(M)}^2 & = \frac{2}{\alpha_{n,1/2}}  \sum_{k \ge 0} |a_k|^2 (1-e^{-\lambda_k t})^2 \lambda_k^{1/2} \\ & \le \frac{2}{\alpha_{n,1/2}} \,  t^{\sigma-1/2} \sum_{k \ge 0} |a_k|^2 \lambda_k^{\sigma} \le  C (1-\sigma)t^{\sigma-1/2} [u_\sigma]^2_{H^\sigma(M)} . 
\end{align*}
By interpolation between $L^2$ and $H^\sigma$ (see \cite[Lemma 7.1]{CG24}) and Young's inequality we also have
\[
[ u_\sigma ]_{H^{1/2}(M)} + [P_t u_\sigma]_{H^{1/2}(M)}
\le C ( 1+ [ u_\sigma ]_{H^\sigma(M)} + [P_t u_\sigma]_{H^\sigma(M)} ),
\]
and, by the spectral representation,
\[
[ P_t u_\sigma ]_{H^\sigma(M)}^2
= \frac{2}{\alpha_{n,\sigma}} \sum_{k\ge 0} |a_k|^2 e^{-2\lambda_k t} \lambda_k^\sigma
\le \frac{2}{\alpha_{n,\sigma}} \sum_{k\ge 0} |a_k|^2 \lambda_k^\sigma
=  [u_\sigma]_{H^\sigma(M)}^2 .
\]
Thus 
\[
[ u_\sigma ]_{H^{1/2}(M)} + [P_t u_\sigma]_{H^{1/2}(M)}
\le C (1+ [u_\sigma]_{H^\sigma(M)}) .
\]

Hence, for $\sigma > 3/4$, combining these estimates with Theorem~\ref{teo:BN-Invent fractional}, we obtain
\begin{align*}
    \F(\star J u_\sigma - \star J(P_{t(\sigma)} u_\sigma))  & \le C [ u_\sigma  -P_t u_\sigma]_{H^{1/2}(M)}\Big( 1+ [ u_\sigma ]_{H^{1/2}(M)} + [P_t u_\sigma]_{H^{1/2}(M)} \Big) \\ & \le C (1-\sigma)^{1/2} t(\sigma)^{1/8} ( [u_\sigma]_{H^\sigma(M)} + [u_\sigma]^2_{H^\sigma(M)}  ) \\ & = C \frac{e^{-\frac{1}{4(1-\sigma)}}}{(1-\sigma)^{3/2}} \Big( (1-\sigma)^{2} [  u_\sigma ]^2_{H^{\sigma}(M)} +(1-\sigma)^{2} [  u_\sigma ]_{H^{\sigma}(M)} \Big)  \to 0 , 
\end{align*}
and this concludes the proof. 
\end{proof} 

\begin{rem} Even optimizing all the constants in the proof above, the authors were unable to obtain the sharp constant $C$ in the right-hand side of \eqref{eq: en bound to optimize}, which would yield the correct liminf estimate in Theorem~\ref{thm: new main asymptotics}. Hence, it seems that one cannot obtain the $\Gamma$-liminf part of our result (proved in Section \ref{sbs: liminf}) from the analogous one for the Ginzburg--Landau energy in \cite{2005-Indiana-ABO}.     
\end{rem}

\begin{rem}
Using the same strategy as in the previous section, one can also deduce a very short proof of the equi-coercivity for the $p$-energy of Sobolev maps taking values in spheres of general dimension. This equi-coercivity property was recently studied by the authors in \cite{p-energy} and, in the context of more general target manifolds, in \cite{p-energy-Verona}; see also \cite[Chapter~12]{Stern:PhD} for related formulations. Thus, the present approach yields an alternative proof of the compactness component (for spherical targets) that underlies these results. 
\end{rem}

\subsection{Proof of the liminf inequality}\label{sbs: liminf}

The goal of this subsection is to prove the liminf inequality corresponding to Theorem~\ref{thm: new main asymptotics}, namely
    \begin{equation}
        \label{eq:liminf-1}
            \lim_{s \to 1} \F (\Sigma_s - \Sigma) = 0 \quad \implies \quad \liminf_{s \to 1} \hspace{0.03cm} (1-s)^2 \M_s(\Sigma_s) \geq \frac{2\pi \omega_{n-1}}{n} \M(\Sigma),
    \end{equation}
where, for every $s \in (0,1)$, $\Sigma_s$ and $\Sigma$ are admissible boundaries as defined in~\eqref{eq:sigma class def}.

\medskip

The proof of the liminf inequality relies on several technical tools, but in many places it builds on ideas previously introduced by the authors in~\cite{codim2-frac}. For this reason, we first outline the four main steps of the proof, originally used in~\cite{codim2-frac} in the case where $\Sigma$ was smooth, and indicate how they must be adapted to treat the general case. We then give a detailed proof of each step, making use of the results proved in~\cite{codim2-frac} whenever possible.

\medskip

\noindent \ding{182} (Reductions) First, we may assume that $\Sigma$ in~\eqref{eq:liminf-1} is an integral $(n-2)$-boundary. Indeed, the implication is nontrivial only when the liminf is finite. In this case, the equi-coercivity result proved in Section~\ref{subsec:5-1} implies that $\Sigma$ is an integral $(n-2)$-boundary.

Moreover, using an approximation argument and working in local charts, we reduce the problem to proving the following statement. 

\begin{tcolorbox}[colback=blue!3!white]
  Let $\Omega \subset \R^n$ be a bounded open set and let $\Sigma$ be an integral $(n-2)$-boundary in $\Omega$. Let $\{\Sigma_s\}$ be a family of smooth, closed, oriented $(n-2)$-submanifolds converging to $\Sigma$ in the flat sense, and let $w_s\in C^1(\Omega \setminus \Sigma_s;\Sp^1) \cap W^{1,1}(\Omega; \Sp^1)$ be maps with $\star Jw_s=\pi \Sigma_s$. Then,
\begin{equation}\label{eucl_liminf1}
\liminf_{s \to 1} \hspace{0.03cm} (1-s)^2 [w_s]^2_{H^{\frac{1+s}{2}}(\Omega)}\geq \frac{2\pi \omega_{n-1}}{n} \M_\Omega(\Sigma). 
\end{equation}
\end{tcolorbox}

\noindent This result is the analogue of Proposition~1.2 in~\cite{codim2-frac} in the case of non-smooth $\Sigma$.

\medskip

\noindent \ding{183} (Discretization) In order to prove~\eqref{eucl_liminf1}, we consider discretized approximations of the functions $w_s$ appearing on the left-hand side. For every $(\eps,R,z) \in (0,1) \times O(n) \times (0,1)^n$, where $O(n)$ denotes the orthogonal group in dimension $n$, we define a suitable piecewise affine function $w_s^{\eps, R, z}$ coinciding with $w_s$ on the lattice $R(\eps \mathbb{Z}^n + \eps z)$ and determined by the parameters $\eps, R$ and~$z$. These piecewise affine interpolations are not defined on the whole $\Omega$, but only on a union of cubes inside $\Omega$. This union always contains $\Omega_{\sqrt{n} \eps}$, where for $\eta > 0$ we set
    \begin{equation}
        \label{eq:Omega-eta}
        \Omega_{\eta} :=\{ x \in \Omega : \mathrm{dist}(x, \R^n \setminus \Omega) > \eta  \}.
    \end{equation}
Moreover, these piecewise affine interpolations no longer take values in the circle.

\medskip

\noindent \ding{184} (Comparison) Here, we compare the fractional seminorm of $w_s$ to the average of the Ginzburg--Landau energy (in a form slightly different from~\eqref{eq:def-GL-1}) of the interpolating maps $w_s^{\eps, R, z} \colon \Omega_{\sqrt{n} \eps} \to \R^2$ as the parameters $\eps, R$ and $z$ vary. More precisely, combining~\cite[Lemma~4.2]{codim2-frac} (with the change of variable $\rho=r^2$) and~\cite[Lemma~2 and eq.~(4.21)]{2009-ARMA-AC}, we derive the following estimate.
\begin{tcolorbox}[colback=blue!3!white]
  For every $\eta \in (0,1)$ there exists $s_0 \in (0,1)$ such that for every $s \in (s_0, 1)$ we have
    \begin{equation} \label{eq:key-est1}
        (1-s)^2 [w_s]_{H^{\frac{1+s}{2}}(\Omega)} ^{2} \geq \frac{2\omega_{n-1}}{n} \int_0^{1-\eta} 
        \abs{\log \rho} \, d\rho \fint_{O(n)} dR \int_{(0,1)^n} {GL}_{\rho_s, \beta}^n(w_s^{\rho_s,R,z},\Omega_{\eta}) \, dz,
    \end{equation}
    where $\rho_s:=\rho^{\frac{1}{1-s}}$, $\beta \in (0,1)$ is arbitrary, and
    \begin{equation} \label{eq:GL-beta}
        {GL}_{\eps, \beta}^n(v,\Omega) 
        := 
        \frac{1}{\abs{\log \eps}}\int_{\Omega} \frac{1-\beta}{2} \abs{\nabla v}^2 + \beta c_0 \frac{(1-|v|^2)^2}{\eps^2}  \, dx.
    \end{equation}
    Here, $c_0 > 0$ is the constant appearing in~\cite[Lemma~2]{2009-ARMA-AC}.
\end{tcolorbox}

\noindent The quantity $\rho_s$ relates the parameter $s$ of the fractional seminorm with the scaling parameter $\eps$ in the definition of the Ginzburg--Landau energy. In fact, we note that for every $\rho \in (0,1)$, it holds that $\rho_s \to 0$ as $s \to 1$. 

\medskip

\noindent \ding{185} (Degree) In this step, we exploit the connection between the Ginzburg--Landau energy and the topological degree. It follows from~\cite{HanLi96, Hong96} that for any continuous map $w \colon (-\delta,\delta)^2 \to \R^2$ such that $\abs{\mathrm{deg}(w, (-\delta,\delta)^2, 0)} = d$, we have
    \begin{equation} \label{eq:GL-degree-12}
        {GL}_{\eps, \beta}^2(w,(-\delta,\delta)^2) \ge (1-\beta)\pi \cdot d + O \big( \abs{\log \eps}^{-1} \big), \qquad \text{as $\eps \to 0^+$},
    \end{equation}
where the remainder term depends on $\delta$ and the trace of $w$ on the boundary of $(-\delta,\delta)^2$.

At this point, the strategy differs depending on whether $\Sigma$ is smooth or not. We first sketch the argument in the smooth case and then explain how to adapt the proof to the general case.

\medskip

\emph{Case of smooth $\Sigma$.} Since $\Sigma$ is a boundary, its normal bundle is trivial, and therefore the tubular neighborhood $\mathcal{U}_{\delta}(\Sigma) := \{ x \in \Omega : \mathrm{dist}(x, \Sigma) < \delta \}$ is diffeomorphic to the product $(\Sigma \cap \Omega) \times (-\delta,\delta)^2$.

Let $w \colon \Omega \to \R^2$ be a continuous map such that the restriction of $w$ to any small two-dimensional square orthogonal to $\Sigma$ has a well-defined degree at $0 \in \R^2$. Up to considering a connected component, we may further assume $\Sigma$ connected, so that the degree is constant. We denote by $d$ the absolute value of this number.

Putting together these observations, the previous estimate~\eqref{eq:GL-degree-12}, and Fubini's theorem, we formally derive
    \[
        {GL}_{\eps, \beta}^n(w,\mathcal{U}_{\delta}(\Sigma)) \gtrsim \int_{\Sigma \cap \Omega} {GL}_{ \eps, \beta}^2(w_{\sigma},(-\delta,\delta)^2) \, d \mathcal{H}^{n-2}(\sigma) \ge (1-\beta) \pi \cdot d \cdot \mathcal{H}^{n-2}(\Sigma \cap \Omega) + O \big( \abs{\log \eps}^{-1} \big),
    \]
where $w_{\sigma}$ denotes the restriction of $w$ to the two-dimensional square $\{ \sigma \} \times (-\delta,\delta)^2$. 

By combining this formula with the choices $\eps = \rho_s$ and $w = w_s^{\rho_s,R,z}$, along with inequality~\eqref{eq:key-est1} from the third step, we can formally conclude the proof (after letting $\beta \to 0^+$, $\eta \to 0^+$, and noticing that $\int_0^1 \abs{\mathrm{log} \rho} \, d\rho = 1$).

To make this argument precise, we need to prove that for almost every $\sigma \in \Sigma$, and for sufficiently many choices of the parameters $(r,R,z)$, the topological degree at $0 \in \R^2$ of $w_s^{\rho_s,R,z}$ on $\{ \sigma \} \times (-\delta,\delta)^2$ is~$\pm d$. This fact is not a direct consequence of the identity ${\star Jw_s}=\pi \Sigma_s$, which only encodes information on the degree of $w_s$ on two-dimensional squares orthogonal to $\Sigma_s$. In particular, it provides no control on the degree of $w_s$ on squares centered at points of $\Sigma$, where a priori the degree may not even be well defined. However, as a consequence of the flat convergence of $\Sigma_s$ to $\Sigma$, it is possible to prove that the restriction of $w_s$ to $\{ \sigma \} \times (-\delta, \delta)^2$ has degree~$\pm d$ for every $\sigma \in \Sigma$ outside an exceptional set whose measure tends to zero as $s\to 1$. We then show that this property is stable under the discretization procedure leading to $w_s^{\rho_s,R,z}$, again for all parameters outside an exceptional set of asymptotically vanishing measure.

There is a second delicate point, namely the behavior of the remainder term in inequality~\eqref{eq:GL-degree-12}. Indeed, for the choice $\eps= \rho_s :=\rho^{\frac{1}{1-s}}$, it is not immediately clear that the remainder term is $(1-s) O ( \abs{\log \rho}^{-1} )$ uniformly with respect to $w_s^{\rho_s,R,z}$ and $\sigma \in \Sigma$. This difficulty is resolved by using sharper estimates on the remainder term established in~\cite{1999-SIAM-Jerrard, 2005-Indiana-ABO}. 

\medskip

\emph{General case of non-smooth $\Sigma$.} The strategy outlined in Step~\ding{185} cannot be applied as stated when $\Sigma$ is not associated with a smooth submanifold, since in that case the notion of a tubular neighborhood becomes problematic. However, since $\Sigma$ is an integral current, it is close, at small scales and in the flat distance, to its tangent plane at almost every point. Moreover, the mass of $\Sigma$ inside a small cube of radius $r$ centered at such a point is comparable to $r^{n-2}$ times the multiplicity of $\Sigma$ at that point.

Consequently, in order to prove inequality~\eqref{eucl_liminf1}, we first establish a more robust estimate in which $\Sigma$ is replaced by its tangent plane, while allowing the functions $w_s$ appearing in the statement of Step~\ding{182} to be singular on a larger set. This estimate is the content of Proposition~\ref{prop:local-liminf} below, and represents the main novelty compared to the proof of the liminf inequality in the smooth case. It can be interpreted as a quantitative version of~\cite[Proposition~4.1]{codim2-frac}.

At this point, to prove Proposition~\ref{prop:local-liminf}, we can combine Steps~\ding{183},~\ding{184}, and~\ding{185}, since the limiting current is now a codimension-two plane, and therefore smooth.

To summarize, we have the chain of implications  
    \[
        \text{Proposition}~\ref{prop:local-liminf} \implies~\eqref{eucl_liminf1} \implies~\eqref{eq:liminf-1}.
    \]

The rest of the section is organized as follows. First, we introduce the relevant notation and recall the precise definition of the discretized approximations used in~\cite[Subsection~4.1]{codim2-frac}. We then turn to the proof of Proposition~\ref{prop:local-liminf}, which constitutes the bulk of the section. Then, we establish the two implications  
Proposition~\ref{prop:local-liminf} $\implies$~\eqref{eucl_liminf1} and~\eqref{eucl_liminf1} $\implies$~\eqref{eq:liminf-1}. These require less work and are independent from the previous arguments.

\medskip

As anticipated in Step~2, every choice of parameters $(\eps, R, z) \in (0,1) \times O(n) \times (0,1)^n$ defines the lattice $R(\eps \Z^n + \eps z) \subset \R^n$. Given an open set $F \subset \R^n$, we denote by 
    \begin{equation*}
            \mathscr{C}_{\eps,R,z}(F) := \Big\{ j + R((0,\eps)^n) : j \in R(\eps \Z^n + \eps z) \ \text{and} \ j + R((0,\eps)^n) \Subset F \Big\},
    \end{equation*}
the family of cubes induced by the lattice and contained in $F$, and by
    \begin{equation*}
            \mathscr{E}_{\eps,R,z}(F) := \Big\{ I : I \ \text{is a closed one-dimensional edge of some cube $Q \in \mathscr{C}_{\eps,R,z}(F)$} \Big\}.
    \end{equation*}

We recall that the union of the cubes in $\mathscr{C}_{\eps,R,z}(F)$ contains the set $F_{\sqrt{n} \eps}$ of points whose distance from the complement of $F$ is greater than $\sqrt{n} \eps$ (see~\eqref{eq:Omega-eta}). Let now $w \colon F \to \R^2$. We want to define a piecewise affine function $w^{\eps,R,z} \colon F_{\sqrt{n} \eps} \to \R^2$ that coincides with $w$ on the lattice $R(\eps \Z^n + \eps z)$. An affine function is uniquely determined by its values at $n+1$ distinct points. However, each cube $Q \in \mathscr{C}_{\eps,R,z}(F)$ has $2^n$ vertices. For this reason, we consider the Kuhn decomposition of $Q$ into $n!$ simplices with vertices among those of $Q$. On each simplex $T$, we then define $w^{\eps,R,z}$ as the unique affine function that coincides with $w$ at the vertices of $T$; we refer to~\cite[Section~4]{2009-ARMA-AC},~\cite[Subsection~2.2]{2023-Solci}, and~\cite[Subsection~4.1]{codim2-frac} for more details.

\medskip

We are now in a position to prove the main result of this section.

\begin{prop}\label{prop:local-liminf}
Let $n\geq 2$ be an integer. Let $E\subset \R^{n-2}$ be a bounded open set with Lipschitz boundary, and let $\delta \in (0,1/2)$. Let us set
    
    \[
        F:=E\times (-\delta,\delta)^2
        \qquad\mbox{and}\qquad
        F ^{*} :=F \setminus (E \times \{(0,0)\} ).
    \]

We also consider for every $\ell \in (0,\delta)$ and every $\sigma \in E$ the closed curve
    \[
        \gamma_{\sigma,\ell}:= \{\sigma\}\times \partial ([-\ell,\ell]^2) \subset F ^{*}.
    \]
    
Moreover, for every $s\in (0,1)$ let $N_{s} \subset F$ be a relatively closed, $(n-1)$-rectifiable set, and let $w_s\in H^{\frac{1+s}{2}}(F; \Sp^1)\cap C^1(F\setminus N_s;\Sp^1)$.

Let also $d\in \N^+$, and let us assume that $\abs{\deg(w_{s},\gamma)}=d$ for every closed curve $\gamma \subset F^{*}\setminus N_{s}$ that is homotopic to some $\gamma_{\sigma,\ell}$ in $F^{*}$ and for every $s\in (0,1)$. Finally, let us set also
\[
    \lambda:=\limsup_{s \to 1} \biggl(\frac{\mathcal{H}^{n-1}(N_s)}{\delta\mathcal{H}^{n-2}(E)}\biggr)^{1/3}.
\]

Then, there exists a dimensional constant $C_n>0$ such that
\begin{equation}\label{th:local-liminf}
\liminf_{s\to 1^{-}} \hspace{0.03cm} (1-s)^2 [w_s]_{H^{\frac{1+s}{2}}(F)} ^2 \geq \frac{2\pi\omega_{n-1}}{n} d \cdot (1- C_n \lambda) \cdot \mathcal{H}^{n-2}(E).
\end{equation}
\end{prop}

\begin{proof}

Assume first $n\ge 3$. We first notice that it suffices to prove~\eqref{th:local-liminf} for $\lambda \in [0,\lambda_0]$, for some $\lambda_0 > 0$ depending only on $n$, that is, when $\mathcal{H}^{n-1}(N_s)$ is not asymptotically too large compared with $\delta \mathcal{H}^{n-2}(E)$. Indeed, if it holds in this range with a constant $C_n > 0$, then it also holds for all $\lambda \ge 0$, possibly with the larger constant $\max\{C_n,\lambda_0^{-1}\}$. The choice of $\lambda_0$ will be specified later in the proof.

Moreover, we observe that without loss of generality we can assume that the liminf in \eqref{th:local-liminf} is finite, and up to extracting a subsequence (that we do not relabel), we can also assume that it is a limit, so that in particular
\begin{equation}\label{eq:limsup_finite}
\limsup_{s\to 1^{-}} \hspace{0.03cm} (1-s)^2 [w_s]_{H^{\frac{1+s}{2}}(F)} ^{2} <+\infty.
\end{equation}

We fix $\eta \in (0,\delta/8)$ and choose $s_0 = s_0(\eta)$ such that~\eqref{eq:key-est1} holds. Then we let $\rho \in (0,1-\eta)$ and $s \in (s_0,1)$. For convenience, we also set $\Cube := (0,1)^n$.

At this point, we select the ``good'' cubes in $\mathscr{C}_{\rho_s,R,z}(F)$, namely those cubes for which we have an estimate on the oscillation of $w_s$ on each of their one-dimensional edges, and such that $N_s$ does not intersect any of these edges.

To this end, for every $(R,z)\in O(n) \times \Cube$, we consider the set
    \begin{equation} \label{eq:O}
        \mathscr{O}_{s,\rho,R,z}:=\Big\{Q\in \mathscr{C}_{\rho_s,R,z}(F) : \osc(w_s,I)> \frac{\sqrt{2}}{n} \text{ for some edge $I$ of $Q$} \Big\},
    \end{equation}
of cubes for which we do not have a good oscillation estimate on all their edges, and the set
    \begin{equation} \label{eq:N}
        \mathscr{N}_{s,\rho,R,z}:=\Big\{Q\in \mathscr{C}_{\rho_s,R,z}(F) : I \cap N_s \neq \varnothing \text{ for some edge $I$ of $Q$}\Big\},
    \end{equation}
of cubes with at least one edge intersecting $N_s$.

We consider also the set
    \begin{equation} \label{eq:B}
        \mathcal{B}_{s,\rho}^{\lambda}:=
            \Big\{(R,z)\in O(n) \times \Cube
            : \lvert \mathscr{O}_{s,\rho,R,z}\cup \mathscr{N}_{s,\rho,R,z} \rvert  > \frac{2\lambda^2 \delta}{\rho_{s} ^{n-1}}\mathcal{H}^{n-2}(E)\Big\}.
    \end{equation}
For convenience, we set $f_{s,\rho}(R,z):=\lvert \mathscr{O}_{s,\rho,R,z} \rvert$ and $g_{s,\rho}(R,z):=\lvert \mathscr{N}_{s,\rho,R,z} \rvert$. Then, by Markov's inequality, we obtain
    \begin{equation} \label{eq:Markov-1}
        \mu \otimes \Leb^n \bigg( \Big\{ f_{s,\rho} > \frac{\lambda^2 \delta}{\rho_{s} ^{n-1}}\mathcal{H}^{n-2}(E) \Big\} \bigg) \le \frac{\rho_s^{n-1}}{\lambda^2 \delta} \frac{1}{\mathcal{H}^{n-2}(E)} \int_{O(n) \times \Cube} f_{s,\rho} \, d\mu dz, 
    \end{equation}
where $\mu$ denotes the standard probability measure on $O(n)$, namely the bi-invariant Haar measure normalized to have total mass one. The same formula holds with $g_{s,\rho}$ in place of $f_{s,\rho}$.

We define $\mathscr{F} \subset \mathscr{E}_{\rho_s,R,z}(F)$ as the subset of one-dimensional edges of cubes in $\mathscr{O}_{s,\rho,R,z}$. Then,
    \begin{equation} \label{eq:Markov-2}
        f_{s,\rho}(R,z) \le 2^{n-1} \sum_{I \in \mathscr{E}_{\rho_s,R,z}(F)} \chi_{\mathscr{F}}(I) \le n^2 2^{n-2} \sum_{I \in \mathscr{E}_{\rho_s,R,z}(F)} \osc(w_s,I)^2.
    \end{equation}
Here, $\chi_{\mathscr{F}}(I) = 1$ if $I \in \mathscr{F}$ and $0$ otherwise. The first inequality then follows from the fact that each edge belongs to at most $2^{n-1}$ cubes, while the second one uses that $1 < n \cdot \osc(w_s,I)/\sqrt{2}$ for every edge $I$ belonging to a cube in $\mathscr{O}_{s,\rho,R,z}$, by~\eqref{eq:O}. Combining~\eqref{eq:Markov-2} with~\cite[Lemma~4.3]{codim2-frac} we deduce that
    \begin{equation}
        \label{eq:Markov-3}
        \int_{O(n) \times \Cube} f_{s,\rho} \, d\mu dz \le \frac{n^2 2^{n-2}}{\rho_s^{n-1-s}} \cdot c_{n,s_0} (1-s) [w_s]_{H^{\frac{1+s}{2}}(F)}^2,
    \end{equation}
where $c_{n,s_0} > 0$ is the positive constant appearing in~\cite[Lemma~4.3]{codim2-frac}.

From~\eqref{eq:Markov-1} and~\eqref{eq:Markov-3} we have
    \begin{equation}
        \label{eq:Markov-4}
            \mu \otimes \Leb^n \bigg( \Big\{ f_{s,\rho} > \frac{\lambda^2 \delta}{\rho_{s} ^{n-1}}\mathcal{H}^{n-2}(E) \Big\} \bigg) \le \frac{\rho_s^s}{\lambda^2 \delta} \cdot \frac{n^2 2^{n-2}}{\mathcal{H}^{n-2}(E)} \cdot c_{n,s_0} (1-s) [w_s]_{H^{\frac{1+s}{2}}(F)}^2.
    \end{equation}

Moreover, it holds that
    \begin{equation} \label{eq:Markov-5}
        g_{s,\rho}(R,z) \le 2^{n-1} \sum_{I \in \mathscr{E}_{\rho_s,R,z}(F)} \mathcal{H}^0(I \cap N_s),
    \end{equation}
since each edge belongs to at most $2^{n-1}$ cubes, and $\mathcal{H}^0(I \cap N_s) \ge 1$ for every edge $I$ belonging to a cube in $\mathscr{N}_{s,\rho,R,z}$, by~\eqref{eq:N}. Combining~\eqref{eq:Markov-5} with~\cite[Lemma~4.5]{codim2-frac} we deduce that
    \begin{equation}
        \label{eq:Markov-6}
        \int_{O(n) \times \Cube} g_{s,\rho} \, d\mu dz \le \frac{2^{n-1}}{\rho_s^{n-1}} \cdot c_{n} \mathcal{H}^{n-1}(N_s),
    \end{equation}
where $c_{n} > 0$ is the positive constant appearing in~\cite[Lemma~4.5]{codim2-frac}.

From~\eqref{eq:Markov-1} (with $g_{s,\rho}$ in place of $f_{s,\rho}$) and~\eqref{eq:Markov-6} we have
    \begin{equation}
        \label{eq:Markov-7}
            \mu \otimes \Leb^n \bigg( \Big\{ g_{s,\rho} > \frac{\lambda^2 \delta}{\rho_{s} ^{n-1}}\mathcal{H}^{n-2}(E) \Big\} \bigg) \le \frac{1}{\lambda^2 \delta} \cdot \frac{2^{n-1}}{\mathcal{H}^{n-2}(E)} \cdot c_{n} \mathcal{H}^{n-1}(N_s).
    \end{equation}

Putting together the definition~\eqref{eq:B} and inequalities~\eqref{eq:Markov-4} and~\eqref{eq:Markov-7}, we obtain
    \begin{equation*}
            \mu \otimes \Leb^n(\mathcal{B}^{\lambda}_{s,\rho})
            \le
            \frac{c_{n,s_0} n^2 2^{n-2}}{\lambda^2 } \cdot \frac{\rho_s^s}{\delta \mathcal{H}^{n-2}(E)} \cdot (1-s) [w_s]_{H^{\frac{1+s}{2}}(F)}^2
            +
            \frac{c_n 2^{n-1}}{\lambda^2 } \cdot \frac{\mathcal{H}^{n-1}(N_s)}{\delta \mathcal{H}^{n-2}(E)}.
    \end{equation*}
In particular, using~\eqref{eq:limsup_finite}, we also have
    \begin{equation}
        \label{eq:KestA}
        \limsup_{s \to 1} \mu \otimes \Leb^n(\mathcal{B}^{\lambda}_{s,\rho}) \le \frac{c_n 2^{n-1}}{\lambda^2 } \limsup_{s \to 1} \bigg( \frac{\mathcal{H}^{n-1}(N_s)}{\delta \mathcal{H}^{n-2}(E)} \bigg) = c_n 2^{n-1}\lambda.
    \end{equation}
The previous estimate is nontrivial only if $\lambda_0 \le c_n^{-1} 2^{-(n-1)}$, which we assume from now on. 

\smallskip

Let $\pi_{1} \colon \R^{2} \times \R^{n-2} \to \R^{2}$ and $\pi_{2} \colon \R^2 \times \R^{n-2} \to \R^{n-2}$ be the orthogonal projections onto the two factors. We then consider the set
    \[
        \mathfrak{S}_{s,\rho,R,z}^{\lambda}
        :=
        \Big\{\sigma\in E_{\eta} : \lvert \{ Q\in \mathscr{O}_{s,\rho,R,z} \cup \mathscr{N}_{s,\rho,R,z} : \sigma\in \pi_{2}(Q) \} \rvert > \frac{\lambda\delta}{\rho_{s}} \Big\},
    \]
namely the set of points $\sigma\in E_\eta$ (see~\eqref{eq:Omega-eta} for the notation) for which the number of ``bad'' cubes intersecting $\{\sigma\}\times (-\delta,\delta)^2$ exceeds $\lambda\delta / \rho_s$. We observe that
\begin{align}
    \notag
    \frac{\lambda\delta}{\rho_{s}} \mathcal{H}^{n-2}(\mathfrak{S}_{s,\rho,R,z}^{\lambda}) & \leq 
    \int_{E_\eta} \sum_{Q\in \mathscr{O}_{s,\rho,R,z} \cup \mathscr{N}_{s,\rho,R,z}} \chi_{\pi_{2}(Q)}(\sigma) \, d\sigma \\[1.5ex] \notag
    & = 
    \sum_{ Q \in \mathscr{O}_{s,\rho,R,z} \cup \mathscr{N}_{s,\rho,R,z}} \mathcal{H}^{n-2}(\pi_{2}(Q)) \\[1.5ex] \label{eq:KestB}
    & \leq 
    p_{n} \rho_{s}^{n-2} \lvert \mathscr{O}_{s,\rho,R,z}\cup \mathscr{N}_{s,\rho,R,z} \rvert,
\end{align}
where $\chi_{\pi_2(Q)}(\sigma) =1$ if $\sigma \in \pi_2(Q)$ and $0$ otherwise, and $p_n:=\max\{\mathcal{H}^{n-2}(\pi_{2}(R(\Cube))) : R \in O(n)\}$ is a constant depending only on the dimension.

As a consequence of~\eqref{eq:B} and~\eqref{eq:KestB}, for every $(R,z)\in (O(n) \times \Cube) \setminus \mathcal{B}^{\lambda}_{s,\rho}$ we have
    \begin{equation}\label{est:meas_Y}
        \mathcal{H}^{n-2}(\mathfrak{S}_{s,\rho,R,z}^{\lambda}) \leq 2 p_n \lambda \mathcal{H}^{n-2}(E).
    \end{equation}

We are now in a position to use the relation between the Ginzburg--Landau energy (see~\eqref{eq:GL-beta}) and the topological degree, as anticipated in the description of Step~4. 

We introduce the following notation:
    \begin{equation*}
        e_{\eps, \beta}(v)   
            := 
        \frac{1-\beta}{2\abs{\log \eps}} \left( \abs{\nabla v}^2 + \frac{\beta c_0}{\eps^2} (1-|v|^2)^2 \right), \qquad v \colon \Omega \to \R^2,
    \end{equation*}
where $\Omega \subset \R^n$ is an open set, and $(\eps, \beta) \in (0,1)^2$. With this notation, the functional ${GL}_{\eps, \beta}^n(v,\Omega)$ defined in~\eqref{eq:GL-beta} is simply the integral of $e_{\eps, \beta}(v)$ over $\Omega$.

\smallskip

We claim that if $\lambda_0 \le \frac{1}{8\sqrt{n}}$, then there exists a universal constant $K > 0$ such that, for every $\nu \in (0,1/K)$, there exist constants $D_1 > 0$ (depending only on $\beta$ and $\nu$) and $s_1 \in [s_0,1)$ (depending on $\eta,\delta,\beta, d$, and $\nu$) such that
    \begin{equation}
        \label{eq:main-final-1}
            \int_{\{\sigma\}\times [-\delta+\eta,\delta-\eta]^2} e_{\rho_s,\beta}(w_s^{\rho_s,R,z}) \, d\mathcal{H}^2
            \ge
            (1-\beta)(1-K\nu)\pi \cdot d \cdot \bigg(1-\frac{\abs{\log \delta} +D_1(1+\log d)}{\abs{\log \rho_s}}\bigg),
    \end{equation}
for every $(\rho,R,z) \in (0,1-\eta) \times O(n) \times \Cube$, every $\sigma\in E_{\eta} \setminus \mathfrak{S}_{s,\rho,R,z}^{\lambda}$, and every $s \in (s_1,1)$. Here, $w_{s} ^{\rho_s,R,z}$ denotes the discretized approximation of $w_s$ introduced in Step~2.

This is a refined version of~\eqref{eq:GL-degree-12} from Step~4, with an explicit dependence of the remainder term on the parameters. Estimate~\eqref{eq:main-final-1} is essentially a consequence of~\cite[Lemma~3.10]{2005-Indiana-ABO}. For this reason, we first prove~\eqref{th:local-liminf} assuming~\eqref{eq:main-final-1}, and we verify~\eqref{eq:main-final-1} at the end of the proof. 

\smallskip

We recall that $F = E \times (-\delta,\delta)^2$ and that $F_{\eta}$ denotes the set of points whose distance from the complement of $F$ is greater than $\eta$. Clearly, we have
    \begin{equation}
        \label{eq:main-final-2}
        {GL}^n_{\rho_s,\beta}(w_s^{\rho_s,R,z}, F_{\eta})
        \ge
        \int_{E_{\eta} \setminus \mathfrak{S}_{s,\rho,R,z}^{\lambda}} d\sigma
        \int_{\{\sigma\} \times [-\delta+\eta, \delta-\eta]^2} e_{\rho_s,\beta}(w_s^{\rho_s,R,z}) \, d\mathcal{H}^2.
    \end{equation}

From~\eqref{est:meas_Y}, for every $(R,z)\in (O(n) \times \Cube) \setminus \mathcal{B}^{\lambda}_{s,\rho}$ we have
    \[
        \mathcal{H}^{n-2}(E_{\eta} \setminus \mathfrak{S}_{s,\rho,R,z}^{\lambda}) \ge \mathcal{H}^{n-2}(E_{\eta}) - 2 \lambda p_n \mathcal{H}^{n-2}(E),
    \]
where the right-hand side is independent of $R$ and $z$. Moreover, choosing $\lambda_0 \le \frac{1}{4 p_n}$, there exists $\eta_0=\eta_0(E) \in (0,\delta/8)$ such that for every $\eta \in (0,\eta_0)$ and every $\lambda \in (0,\lambda_0)$ the right-hand side of the above inequality is strictly positive.

Consequently, from~\eqref{eq:main-final-1} and~\eqref{eq:main-final-2}, we deduce that, for every $\rho \in (0,1-\eta)$,
    \begin{multline}
        \label{eq:main-final-3}
         \liminf_{s \to 1} \int_{(O(n) \times \Cube) \setminus \mathcal{B}_{s,\rho}^{\lambda}} {GL}^n_{\rho_s,\beta}(w_s^{\rho_s,R,z}, F_{\eta}) \, d\mu dz \\[1.5ex]
        \ge
        \big(\mathcal{H}^{n-2}(E_{\eta}) - 2 \lambda p_n \mathcal{H}^{n-2}(E) \big) (1-\beta)(1-K\nu)\pi d \cdot \liminf_{s \to 1^{-}} \mu \otimes \Leb^n \big((O(n) \times \Cube) \setminus \mathcal{B}_{s,\rho}^{\lambda} \big).
    \end{multline}

We can now combine all the estimates to conclude the proof. We set
    \[
        \lambda_0 := \min\bigg\{ \frac{1}{c_n2^{n-1}}, \frac{1}{4p_n}, \frac{1}{8\sqrt{n}} \bigg\}, \quad \lambda \in ( 0,\lambda_0], \quad \text{and} \quad \eta \in (0,\eta_0(E)).
    \]
By restricting the inner integral on the right-hand side of~\eqref{eq:key-est1} to the complement of $\mathcal{B}_{s,\rho}^{\lambda}$, taking the liminf as $s \to 1$, applying Fatou's lemma, and using~\eqref{eq:KestA} and~\eqref{eq:main-final-3}, we obtain
    \begin{multline}
        \label{eq:recall-1}
            \liminf_{s \to 1} \hspace{0.03cm} (1-s)^2 [w_s]_{H^{\frac{1+s}{2}}(F)} ^{2} \geq \\[1ex]
            \frac{2\omega_{n-1}}{n} 
            \big(\mathcal{H}^{n-2}(E_{\eta}) - 2 \lambda p_n \mathcal{H}^{n-2}(E) \big) (1-\beta)(1-K\nu)\pi d (1-c_n 2^{n-1} \lambda)
            \int_0^{1-\eta} 
            \abs{\log \rho} \, d \rho.
        \end{multline}

Finally, letting $\nu \to 0^{+}$, $\beta\to 0^{+}$, and $\eta\to 0^{+}$ in~\eqref{eq:recall-1}, and using that $\int_0^1 \abs{\log \rho} \, d\rho =1$, we conclude that
    \begin{align*}
        \liminf_{s \to 1} \hspace{0.03cm} (1-s)^2 [w_s]_{H^{\frac{1+s}{2}}(F)} ^{2}
        & \geq 
        \frac{2 \pi \omega_{n-1}}{n} d \cdot
            \mathcal{H}^{n-2}(E) \cdot ( 1 - 2 \lambda p_n ) (1-c_n 2^{n-1} \lambda) \\[1.5ex]
        & \geq
        \frac{2 \pi \omega_{n-1}}{n} d \cdot
            \mathcal{H}^{n-2}(E) \cdot (1-C_n \lambda),
    \end{align*}
where $C_n:=2p_n+c_n 2^{n-1}$. This concludes the proof of~\eqref{th:local-liminf}.

\smallskip

It remains to prove~\eqref{eq:main-final-1}. For every $\sigma\in E_{\eta}$ we define
    \[
        \mathcal{L}_{s,\rho,R,z}(\sigma):=\Big\{ \ell \in (\delta/2,\delta-\eta) : \exists Q \in \mathscr{O}_{s,\rho,R,z} \cup \mathscr{N}_{s,\rho,R,z} \text{ such that } \gamma_{\sigma,\ell} \cap Q \neq \varnothing \Big\},
    \]
namely the set of $\ell \in (\delta/2,\delta-\eta)$ for which $\gamma_{\sigma,\ell}$ intersects at least one ``bad'' cube. We recall that we have chosen $\eta<\delta/8$, so that $(\delta/2,\delta-\eta)$ is not empty.

For every $\sigma\in E_{\eta} \setminus \mathfrak{S}_{s,\rho,R,z}^{\lambda}$ it holds that
\begin{equation}\label{est:meas_L}
    \mathcal{H}^{1}(\mathcal{L}_{s,\rho,R,z}(\sigma)) \leq \sum_{Q\in\mathscr{O}_{s,\rho,R,z}\cup \mathscr{N}_{s,\rho,R,z}} \diam (Q\cap (\{\sigma\}\times(-\delta,\delta)^2))\leq \frac{\lambda\delta}{\rho_s} \sqrt{n} \rho_s = \lambda \sqrt{n} \delta.    
\end{equation}
Since $\lambda \le \lambda_0 \le \frac{1}{8\sqrt{n}}$, we have $\lambda \sqrt{n} \delta \le \delta/8 < \delta/2 - \eta$. Hence, for every $\sigma \in E_{\eta} \setminus \mathfrak{S}_{s,\rho,R,z}^{\lambda}$, there exist many curves $\gamma_{\sigma,\ell}$ that avoid all ``bad'' cubes.


\smallskip

For every such curve $\gamma_{\sigma,\ell}$ with $\ell \notin \mathcal{L}_{s,\rho,R,z}(\sigma)$, let $U_{\sigma,\ell}$ be the union of the cubes in $\mathscr{C}_{\rho_s,R,z}(F)$ intersecting $\gamma_{\sigma,\ell}$, and let $S_{\sigma,\ell}$ be the union of their edges.

We claim that for every $\ell\notin \mathcal{L}_{s,\rho,R,z}(\sigma)$, it holds that
    \begin{equation}\label{claim_v}
        \abs*{w_s^{\rho_s,R,z}(x)}>1/2 \quad \forall x\in U_{\sigma,\ell}
            \qquad \text{and} \qquad
        \abs*{\deg(w_s^{\rho_s,R,z}/|w_s^{\rho_s,R,z}|, \gamma_{\sigma,\ell})} = d.
    \end{equation}

Indeed, if $Q \notin \mathscr{O}_{s,\rho,R,z}$, the oscillation estimate for $w_s$ yields $\abs{w_s(x)-w_s(y)} \leq \sqrt{2}/n$ for every pair of vertices $x,y$ of $Q$ connected by an edge. Since any two vertices of a $n$-dimensional cube can be joined by a path consisting of at most $n$ edges, the triangle inequality implies that $\abs{w_s(x)-w_s(y)} \leq \sqrt{2}$ for all vertices $x,y$ of $Q$. Hence, the image of the vertices of $Q$ under $w_s$ is contained in a quarter of the circle. By construction, $w_{s} ^{\rho_s,R,z}(Q)$ lies in the convex hull of the images of the vertices of $Q$, and $w_{s} ^{\rho_s,R,z}=w_s$ on these vertices. It follows that $\abs{w_{s} ^{\rho_s,R,z}} \geq \sqrt{2}/2>1/2$ on every cube $Q\notin \mathscr{O}_{s,\rho,R,z}$, and therefore on $U_{\sigma,\ell}$.

Moreover, $\gamma_{\sigma,\ell}$ can be continuously deformed inside $U_{\sigma,\ell}$ (and hence inside $F^*$, since $\ell > \delta/2$ and $\rho_s$ is small) into a curve $\gamma \subset S_{\sigma,\ell}$. In particular, $\gamma \cap N_{s} = \varnothing$, because the cubes in $U_{\sigma,\ell}$ do not belong to $\mathscr{N}_{s,\rho,R,z}$. Hence, by our hypothesis, $\abs{\deg(w_{s},\gamma)}=d$.

Finally, we observe that the oscillation estimate implies that $w_s ^{\rho_s,R,z}/\abs{w_s^{\rho_s,R,z}}$ is homotopic to $w_s$ on each edge in $S_{\sigma,\ell}$. Indeed, on these edges $w_s$ takes values in an arc of length less than $\pi/2$, while $w_s^{\rho_s,R,z}$ takes values in its convex hull, which does not contain the origin. Hence $w_s ^{\rho_s,R,z}/\abs{w_s^{\rho_s,R,z}}$ is homotopic to $w_s$ on $\gamma$, and we conclude that
    \[
        \abs*{\deg(w_s ^{\rho_s,R,z}/\abs{w_s^{\rho_s,R,z}}, \gamma_{\sigma,\ell})} =\abs*{\deg(w_s ^{\rho_s,R,z}/\abs{w_s^{\rho_s,R,z}}, \gamma)}
        =
        \abs*{\deg(w_s,\gamma)}=d,
    \]
which concludes the proof of~\eqref{claim_v}.

\smallskip

At this point, we observe that for every $\sigma \in E_{\eta}$ and every $\beta\in (0,1)$ it holds that
    \[
        \int_{0}^{\delta-\eta} d\ell \int_{\gamma_{\sigma,\ell}} e_{\rho_s,\beta}(w_{s} ^{\rho_s,R,z})\,d\mathcal{H}^1 = \int_{\{\sigma\}\times [-\delta+\eta,\delta-\eta]^2} e_{\rho_s,\beta}(w_{s}^{\rho_s,R,z})\,d\mathcal{H}^2.
    \]
In particular, since $\eta \in (0,\delta/8)$ and $\lambda_0 \le \frac{1}{8\sqrt{n}}$, by~\eqref{est:meas_L} it follows that for every $\sigma \in E_\eta \setminus \mathfrak{S}_{s,\rho,R,z}^{\lambda}$ there exists $\ell_\sigma\in (\delta/2,\delta-\eta)\setminus \mathcal{L}_{s,\rho,R,z}(\sigma)$ such that
    \begin{equation} \label{eq-GL-b1}
        \int_{\gamma_{\sigma,\ell_\sigma}} e_{\rho_s,\beta}(w_{s} ^{\rho_s,R,z})\,d\mathcal{H}^1 \leq \frac{4}{\delta} \int_{\{\sigma\}\times [-\delta+\eta,\delta-\eta]^2} e_{\rho_s,\beta}(w_{s} ^{\rho_s,R,z})\,d\mathcal{H}^2.
    \end{equation}

As a consequence of~\eqref{claim_v}, the restriction $\widetilde{w}_{s} ^{\rho_s,R,z}$ of $w_{s} ^{\rho_s,R,z}$ to $\{ \sigma \} \times [-\ell_{\sigma}, \ell_{\sigma}]^2$ satisfies the assumptions of~\cite[Lemma~4.6]{codim2-frac} (which is a reformulation of~\cite[Lemma~3.10]{2005-Indiana-ABO}) for every $(\rho,R,z) \in (0,1-\eta) \times O(n) \times \Cube$ (see also~\cite[Remark~4.7]{codim2-frac}).

Let $K, D_0, D_1 > 0$ be the constants given by~\cite[Lemma~4.6]{codim2-frac}. Since $2 \ell_{\sigma} > \delta$, we can choose $s_1 \in [s_0,1)$ such that
    \[
       \frac{\rho_s}{2\ell_{\sigma}} < 1 \quad \text{and} \quad \frac{\rho_s}{2\ell_{\sigma}} \abs*{ \log \frac{\rho_s}{2\ell_{\sigma}}} < \frac{D_0}{d}, \quad \forall \rho \in (0,1-\eta).
    \]
Then, for every $s \in [s_1,1)$, combining~\eqref{eq-GL-b1} with~\cite[Lemma~4.6]{codim2-frac} (applied with $\nu/4$ in place of~$\rho$) we obtain
    \begin{align}
        \notag
        (1 + K \nu) \int_{\{\sigma\}\times [-\delta+\eta,\delta-\eta]^2} & e_{\rho_s,\beta}(w_{s} ^{\rho_s,R,z}) \, d\mathcal{H}^2 \\ \notag
        &\geq
        \int_{\{\sigma\}\times [-\ell_\sigma,\ell_\sigma]^2} e_{\rho_s,\beta}(w_{s} ^{\rho_s,R,z}) \, d\mathcal{H}^2 + \frac{K \nu \ell_\sigma}{4} \int_{\gamma_{\sigma,\ell_\sigma}} e_{\rho_s,\beta}(w_{s} ^{\rho_s,R,z})\,d\mathcal{H}^1 \\[1ex] \notag
        &\geq 
        \int_{\{\sigma\}\times [-\ell_\sigma,\ell_\sigma]^2} e_{\rho_s,\beta}(\widetilde{w}_{s} ^{\rho_s,R,z}) \, d\mathcal{H}^2 + \frac{K \nu \ell_\sigma}{4} \int_{\gamma_{\sigma,\ell_\sigma}} e_{\rho_s,\beta}(\widetilde{w}_{s} ^{\rho_s,R,z})\,d\mathcal{H}^1 \\[1ex] \notag
        &\geq 
        (1-\beta) \pi d \bigg(1-\frac{\abs{\log \delta} +D_1(1+\log d)}{\abs{\log \rho_s}}\bigg),
    \end{align}
where in the last line we used that $2\ell_{\sigma} \in (\delta,2\delta)$ and $\delta < 1/2$. Finally,~\eqref{eq:main-final-1} follows from the elementary inequality $(1+K\nu)^{-1} \ge 1-K\nu$. This concludes the proof for $n \ge  3$. 

For $n=2$ the same argument applies with many simplifications. The set $E$ is finite and the argument applies separately at each point of $E$, with no need for slicing in the $\sigma$-variable, and the estimate on the exceptional set $\mathfrak S_{s,\rho,R,z}^{\lambda}$ is unnecessary. For every $\sigma \in E$, one just chooses $\ell\in(\delta/2,\delta-\eta)$ so that the square loop $\gamma_{\sigma, \ell}$ avoids the bad squares, which follows directly from the bound on their total number. The degree argument and the two-dimensional Ginzburg--Landau lower bound are exactly the same, giving \eqref{th:local-liminf} with $n=2$. 
\end{proof}


In order to prove that $\text{Proposition}~\ref{prop:local-liminf} \implies~\eqref{eucl_liminf1}$, we will need the following lemma.

\begin{lemma}\label{lemma:blow-up}
Let $\Sigma=\lcurr E,\theta,\tau\rcurr$ be an integral $k$-current with zero boundary in some open set $\Omega\subset \R^n$. For every $x\in\R^n$ and every $r>0$ set $\Sigma_{x,r}:=\lcurr E_{x,r},\theta_{x,r},\tau_{x,r}\rcurr$, where
    \[
        E_{x,r}:=\frac{E-x}{r},\qquad \theta_{x,r}(y):=\theta(x+ry),\qquad \tau_{x,r}(y):=\tau(x+ry).
    \]

For every $x\in E$ at which $E$ admits an approximate tangent plane let us consider also the following integral $k$-current on $\R^n$:
    \[
        \Sigma_{x,0}:=\lcurr T_x E,\theta_{x,0},\tau_{x,0}\rcurr,\quad\text{where}\quad
        \theta_{x,0}(y)\equiv\theta(x),\qquad \tau_{x,0}(y)\equiv\tau(x).
    \]

Then, for $\mathcal{H}^{k}$-almost every $x\in E$ and every $R>0$ it holds that
    \[
        \lim_{r\to 0^+} \F_{B_R}(\Sigma_{x,r}-\Sigma_{x,0})=0,
    \]
and $\|\Sigma_{x,r}\|\rightharpoonup \theta(x) \mathcal{H}^{k}\mres T_x E$ locally in the weak sense of measures.
\end{lemma}

\begin{proof}
We show that the claim holds whenever $x\in E$ is such that all the following holds.
\begin{itemize}
    \item $E$ admits an approximate tangent plane at $x$.
    \item $x$ is an approximate continuity point for both $\theta$ and $\tau$, with respect to the measure $\mathcal{H}^{k}\mres E$, namely
    \begin{gather}
    \lim_{r\to 0} \frac{1}{\mathcal{H}^{k}(E\cap B(x,r))}\int_{B(x,r)} \abs{\theta(y)-\theta(x)}d\mathcal{H}^{k}(y)=0,\label{appr_cont_theta}\\
    \lim_{r\to 0} \frac{1}{\mathcal{H}^{k}(E\cap B(x,r))}\int_{B(x,r)} \abs{\tau(y)-\tau(x)}d\mathcal{H}^{k}(y) =0.\label{appr_cont_tau}
    \end{gather}
    \item The $k$-density of $E$ at $x$ is bounded, namely
    \begin{equation}\label{eq:bounded_density}
    \limsup_{r\to 0^+} \frac{\mathcal{H}^{k}(E\cap B(x,r))}{r^k}<+\infty.
    \end{equation}
\end{itemize} 

We point out that all the previous properties hold for $\mathcal{H}^{k}$-almost every $x\in E$.

Now we show that for every such $x\in E$ it holds that $\Sigma_{x,r}\rightharpoonup \Sigma_{x,0}$ as $r\to 0^+$ in the weak sense of currents. To this end, we fix a test $k$-form $\omega \in \D^{k}(B_R)$, and we compute
\begin{align*}
\langle \Sigma_{x,r}-\Sigma_{x,0} , \omega \rangle &= \int_{E_{x,r}} \theta_{x,r}(y) \tau_{x,r}(y)\cdot\omega(y) \,d\mathcal{H}^{k}(y) - \int_{T_{x}E}\theta(x)\tau(x)\cdot\omega(y) \,d\mathcal{H}^{k}(y)\\
&= \int_{E_{x,r}} \theta_{x,r}(y) \tau_{x,r}(y)\cdot\omega(y) \,d\mathcal{H}^{k}(y) - \int_{E_{x,r}}\theta(x)\tau_{x,r}(y)\cdot\omega(y) \,d\mathcal{H}^{k}(y)\\
&\quad +\int_{E_{x,r}} \theta(x) \tau_{x,r}(y)\cdot\omega(y) \,d\mathcal{H}^{k}(y) - \int_{E_{x,r}}\theta(x)\tau(x)\cdot\omega(y) \,d\mathcal{H}^{k}(y)\\
&\quad +\int_{E_{x,r}} \theta(x) \tau(x)\cdot\omega(y) \,d\mathcal{H}^{k}(y) - \int_{T_{x}E}\theta(x)\tau(x)\cdot\omega(y) \,d\mathcal{H}^{k}(y),
\end{align*}
so it follows
\begin{align*}
\abs{\langle & \Sigma_{x,r}  -\Sigma_{x,0} ,  \omega \rangle} \\  & \le  \|\omega\|_{\infty} \bigg( \int_{E_{x,r} \cap B_R} \abs{\theta_{x,r}(y)-\theta(x)} \,d\mathcal{H}^{k}(y) + \abs{\theta(x)}\int_{E_{x,r} \cap B_R} \abs{\tau_{x,r}(y)-\tau(x)} \,d\mathcal{H}^{k}(y) \bigg) \\[1ex]
& + \bigg|\int_{E_{x,r}} \theta(x) \tau(x)\cdot\omega(y) \,d\mathcal{H}^{k}(y) - \int_{T_{x}E}\theta(x)\tau(x)\cdot\omega(y) \,d\mathcal{H}^{k}(y)\bigg|.
\end{align*}

Observe that the second line in the right-hand side tends to zero as $r\to 0^+$ because of the definition of approximate tangent plane.

Moreover, observe that for every $\mathcal{H}^{k}$-measurable function $f \colon E\to [0,+\infty)$ it holds that
$$\int_{E_{x,r}} f(x+ry)\,d\mathcal{H}^{k}(y)=\frac{1}{r^{k}}\int_{E} f(y)\,d\mathcal{H}^{k}(y),$$
hence also the first line tends to zero because of \eqref{appr_cont_theta}, \eqref{appr_cont_tau} and \eqref{eq:bounded_density}. Thus, we have proved that $\Sigma_{x,r}\rightharpoonup \Sigma_{x,0}$ as $r\to 0^+$ in the weak sense of currents.

Similarly, we can show that for every $x\in E$ as above it holds that $\|\Sigma_{x,r}\|\rightharpoonup \theta(x) \mathcal{H}^{k}\mres T_x E$. Indeed, if we fix $\psi \in C^0_c (B_R)$, we get
\begin{align*}
\bigg|\int_{B_R} \psi(y)\, d(\|\Sigma_{x,r}\|- \theta(x)\mathcal{H}^{k}\mres T_x E )(y)\bigg|=\bigg|\int_{E_{x,r}} \psi(y)\theta_{x,r}(y)\, d\mathcal{H}^{k}(y) - \int_{T_x E} \psi(y)\theta(x)\, d\mathcal{H}^{k}(y)\bigg|\\
\leq \|\psi\|_\infty \int_{E_{x,r} \cap B_R} |\theta_{x,r}(y)-\theta(x)| d\mathcal{H}^k(y) + \theta(x) \bigg|\int_{E_{x,r}} \psi(y)\,d\mathcal{H}^k(y) - \int_{T_x E}\psi(y)d\mathcal{H}^k(y) \bigg|,
\end{align*}
and the right-hand side tends to zero as $r\to 0^+$ for the same reasons as before.

As a consequence of the weak convergence, we also deduce that $\M_{B_R}(\Sigma_{x,r})$, which is the total mass of the measure $\|\Sigma_{x,r}\|$, is bounded as $r\to 0^+$.

Since we know that $\M_{B_R}(\partial \Sigma_{x,r})=0$ for $r$  sufficiently small (so that $B(x,rR)\subset \Omega$), the boundedness of the mass is sufficient to upgrade weak convergence of currents to convergence in the flat norm (see, for example, \cite[Theorem 31.2]{GMT-Simon}), so the proof is concluded.
\end{proof}

We are now ready to prove the implication $\text{Proposition}~\ref{prop:local-liminf} \implies~\eqref{eucl_liminf1}$.

\begin{proof}[Proof that \text{Proposition}~\ref{prop:local-liminf} implies~\eqref{eucl_liminf1}.]

Let $\Sigma$ be as in~\eqref{eucl_liminf1} and $E,\theta,\tau$ be such that $\Sigma=\lcurr E,\theta,\tau \rcurr$. Moreover, we fix a small positive number $\ep>0$. By Lemma~\ref{lemma:blow-up} and Vitali covering theorem, we can find a countable (or finite) family $\{Q(x_i,r_i)\}_{i\in I}$ of pairwise disjoint cubes such that the following hold:
\begin{itemize}
    \item $Q(x_i,r_i)$ has sides of length $r_i$ and some faces (of dimension $n-2$) parallel to $T_{x_i}E$.
    \item $\mathcal{H}^{n-2}(E\setminus \bigcup_{i\in I}Q(x_i,r_i))=0.$
    \item $\F_{\widehat{Q}_i}(\Sigma_{x_i,r_i} -\Sigma_{x_i,0})< \ep$ for every $i\in I$, where $\widehat{Q}_i:=(Q(x_i,r_i)-x_i)/r_i$.
    \item $\M_{\widehat{Q}_i}(\Sigma_{x_i,r_i})\leq (1+\ep)\M_{\widehat{Q}_i}(\Sigma_{x_i,0})=(1+\ep)\theta(x_i)$ for every $i\in I$. 
\end{itemize}
We remark that, by scaling, the third and fourth properties imply that
    \[
        \F_{Q(x_i,r_i)}(\Sigma - \lcurr x_i+ T_{x_i}E , \theta(x_i), \tau(x_i)\rcurr)< \ep r_i^{n-1},
        \qquad
        \M_{Q(x_i,r_i)}(\Sigma)\leq (1+\ep)\theta(x_i)r_i^{n-2}.
    \]

As a consequence, we also have that
    \[
        \limsup_{s \to 1} \F_{Q(x_i,r_i)}(\Sigma_s - \lcurr x_i+ T_{x_i}E , \theta(x_i), \tau(x_i)\rcurr)< \ep r_i^{n-1},
    \]
and, since both $\Sigma_s$ and $T_x E$ are smooth manifolds, there is an integral $(n-1)$-current $T_s$, supported on a \emph{closed} $(n-1)$-rectifiable set $N_s$ (in fact, $N_s$ can be chosen to be a smooth manifold by \cite[Theorem~3.16]{Bre-Mi}), such that
    \[
        \partial T_s=\Sigma_s - \lcurr x_i+ T_{x_i}E , \theta(x_i), \tau(x_i)\rcurr \quad\text{in } Q(x_i,r_i), \qquad \text{and}\qquad \mathcal{H}^{n-1}(N_s)\leq \M(T_s)<\ep r_i^{n-1}.
    \]

As a consequence, for every closed Lipschitz curve $\gamma\subset Q(x_i,r_i)\setminus N_s$, we have that $[\Sigma_s]=\theta(x_i)[x_i+T_{x_i}E]$ in the $(n-2)$-dimensional homology of $Q(x_i,r_i)\setminus\gamma$. Therefore, if $A$ is any surface with $\partial A=\gamma$, then \cite[equation~(3.6)]{ABO2-singularities} yields
    \[
        \abs{\deg(w_s,\gamma)}=\abs{\operatorname{int}(\theta(x_i)\Sigma_s,A)} =\theta(x_i)\abs{\operatorname{int}(x_i+T_{x_i}E,A)},
    \]
and in particular we find that $\abs{\deg(w_s,\gamma)}=\theta(x_i)$ for every loop in $Q(x_i,r_i)$ which winds once around $x_i + T_{x_i} E$ and does not intersect $N_s$.

Thus, identifying $T_{x_i} E \cap Q(x_i,r_i)\cong [0,r_i]^{n-2}$, we can apply Proposition~\ref{prop:local-liminf} with $\delta=r_i/2$ and $d=\theta(x_i)$, so we obtain that
    \begin{align*}
        \liminf_{s\to 1^{-}} \hspace{0.03cm} (1-s)^2 [w_s]_{H^{\frac{1+s}{2}}(Q(x_i,r_i))} ^2 &\geq \frac{2\pi\omega_{n-1}}{n}  \theta(x_i) \cdot r_i^{n-2} \cdot \liminf_{s\to 1^{-}}\bigg[1-C_n\biggl(\frac{\mathcal{H}^{n-1}(N_s)}{r_i \cdot r_i^{n-2}}\biggr)^{1/3}\bigg]\\
        &\geq \frac{2\pi\omega_{n-1}}{n}  \frac{\M_{Q(x_i,r_i)}(\Sigma)}{1+\ep}  \Big[1-C_n \ep^{1/3}\Big].
    \end{align*}

Summing over $i\in I$, since the cubes $Q(x_i,r_i)$ are disjoint and cover $E$ (up to an $\mathcal{H}^{n-2}$-negligible set) we obtain that
    \[
        \liminf_{s\to 1^{-}} \hspace{0.03cm} (1-s)^2 [w_s]_{H^{\frac{1+s}{2}}(\Omega)}^2 \geq \frac{2\pi\omega_{n-1}}{n}  \frac{\M_\Omega(\Sigma)}{1+\ep}  \Big[1-C_n \ep^{1/3}\Big].
    \]
Letting $\ep \to 0^+$ we obtain exactly~\eqref{eucl_liminf1}.
\end{proof}

Finally, we show that $\eqref{eucl_liminf1} \implies~\eqref{eq:liminf-1}$.

\begin{proof}[Proof that \eqref{eucl_liminf1} implies~\eqref{eq:liminf-1}.]

Without loss of generality, we can assume that the currents $\{\Sigma_s\}$ are smooth submanifolds for every $s$.

Indeed, if $u_s$ is the minimizer in the definition of $\mathbb{M}_s(\Sigma_s)$, by Theorem~\ref{thm: approx with standard singularities} we can find a sequence $\{u_{s,k}\}$ of maps that are smooth outside some smooth $(n-2)$-submanifold $\Sigma_{s,k} \subset M$, for which ${\star J u_{s,k}} = \pi \Sigma_{s,k}$, and such that $u_{s,k} \to u_s$ strongly in $H^{\frac{1+s}{2}}$.

By Theorem~\ref{teo:BN-Invent fractional}, we have $\Sigma_{s,k} \to \Sigma_s$ in the flat topology, and
	\[
		\limsup_{k\to \infty} \mathbb{M}_s(\Sigma_{s,k})
		\le
		\lim_{k\to \infty} [u_{s,k}]_{H^{\frac{1+s}{2}}(M)}^2
		=
		[u_s]_{H^{\frac{1+s}{2}}(M)}^2
		=
		\mathbb{M}_s(\Sigma_s).
	\]

Therefore, we can find a diagonal sequence $\Sigma_{s,k(s)}$ that still converges to $\Sigma$ in the flat topology and for which the liminf of the fractional mass has not increased. Therefore, it is enough to prove the liminf inequality for this new sequence, which now consists of smooth submanifolds (and will be denoted just by $\Sigma_s$).

By Corollary~\ref{cor: equivalence inf on smooth}, we can find maps $v_s \in C^\infty(M\setminus \Sigma_s;\Sp^1) \cap H^{\frac{1+s}{2}}(M;\Sp^1)$ such that
	\[
		\star J v_s = \pi \Sigma_s, \quad \text{and} \quad \mathbb{M}_s(\Sigma_s) \ge [v_s]_{H^{\frac{1+s}{2}}(M)}^2 - (1-s).
	\]
In particular, it is enough to prove that
	\[
		\liminf_{s \to 1} \hspace{0.03cm} (1-s)^2 [v_s]_{H^{\frac{1+s}{2}}(M)}^2
		\ge \frac{2\pi \omega_{n-1}}{n} \mathbb{M}(\Sigma).
	\]

At this point, we can further reduce ourselves to the Euclidean setting since, for every fixed $\eps>0$, by Proposition \ref{prop: comparison kernel sharp} we can find a family of pairwise disjoint open sets $\{U_j\}_{j \in J}$ in $M$ and parameterizations $\phi_j \colon V_j \to U_j$, where $V_j \subset \mathbb{R}^n$, such that
	\begin{gather*}
		\liminf_{s \to 1} \hspace{0.03cm} (1-s)^2 [v_s]_{H^{\frac{1+s}{2}}(U_j)}^2
		\ge
		(1-\varepsilon) \liminf_{s \to 1} \hspace{0.03cm} (1-s)^2
		[v_s \circ \phi_j]_{H^{\frac{1+s}{2}}(V_j)}^2, \\[1ex]
		\mathbb{M}_{V_j}(\phi_j^* \Sigma)
		\ge
		(1-\varepsilon)\,\mathbb{M}_{U_j}(\Sigma),
		\quad \text{and} \quad
		\sum_{j\in J} \mathbb{M}_{U_j}(\Sigma)
		\ge
		(1-\varepsilon)\,\mathbb{M}(\Sigma),
	\end{gather*}
so that it is enough to show that
	\[
		\liminf_{s \to 1} \hspace{0.03cm} (1-s)^2
		[v_s \circ \phi_j]_{H^{\frac{1+s}{2}}(V_j)}^2
		\ge
		\frac{2\pi \omega_{n-1}}{n} \mathbb{M}_{V_j}(\phi_j^* \Sigma),
	\]
and this follows from~\eqref{eucl_liminf1}.
\end{proof}

\subsection[Proof of the limsup inequality]{Proof of the limsup inequality}\label{sec: gamma limsup currents}

In this subsection, we prove the $\Gamma$-limsup inequality of Theorem~\ref{thm: new main asymptotics}, that is: for every admissible boundary $\Sigma$ there exists a (discrete) sequence of admissible boundaries $\{\Sigma_s\}$ such that
\[
\lim_{s \to 1} \F (\Sigma_s - \Sigma) = 0  \qquad \mbox{and} \qquad  \limsup_{s \to 1} \hspace{0.03cm} (1-s)^2 \M_{s}(\Sigma_s) \leq \frac{2\pi\omega_{n-1}}{n} \M(\Sigma) . 
\]
Note that if $\M(\Sigma) = +\infty$ there is nothing to prove, so we assume that $\M(\Sigma)<+\infty$ from now on, which in particular implies that $\Sigma$ is an integral $(n-2)$-boundary.

We will use a particular case of a deep result by Almgren--Browder--Caldini--De~Lellis on the smooth approximation of integral cycles \cite{almgren2024optimal}. This is not strictly necessary (see Remark \ref{rem: Gamma-limusp secondo modo} below for an alternative construction), but it allows us to reuse in a very direct way our construction in \cite{codim2-frac} of the recovery sequence for smooth surfaces. 

\begin{teo}[{\cite[Theorem 1.1]{almgren2024optimal}}]\label{thm: Caldini approx} Let $M$ be a smooth, closed, oriented, Riemannian manifold, and let $\Sigma$ be an $m$-dimensional integral boundary in $M$. Then, for every $\ep>0$, there exists a smooth, oriented, embedded, $m$-dimensional submanifold $S$ homologous to $\Sigma$ such that 
\begin{equation*}
    \F(\Sigma-S) < \ep, \qquad \mathcal{H}^m(S) < \M(\Sigma)+\ep . 
\end{equation*}
\end{teo}

With this result, we can easily construct the recovery sequence as follows.

\medskip 

  \noindent \ding{182} (Reductions) Let us make some reductions to reduce the proof to a particular case. By a straightforward diagonal argument, it is enough to prove the $\Gamma$-limsup inequality for a collection $\mathcal{S}$ of integral boundaries that is dense in energy, i.e. such that for every integral $(n-2)$-boundary $\Sigma$ there exists a sequence $\{S_k\} \subset \mathcal{S} $ with $ \F(S_k -\Sigma)\to 0$ and $\limsup_{k\to \infty} \, \M(S_k) \le \M(\Sigma) $. By Theorem \ref{thm: Caldini approx}, the collection of smooth, oriented, embedded, $(n-2)$-boundaries (with unit density) is dense in energy. Hence, we assume that $\Sigma$ lies in this class. In particular, by \cite[Theorem VIII.3]{Kirby} (and \cite[Theorem~5.11]{1960-FF}), $\Sigma = \partial N$ for some smooth, embedded $(n-1)$-dimensional surface $N$.

Second, by \eqref{eq: comparability limit dist space} and the definition of $\M_{s}$ it is enough to provide a map $u_\Sigma \in \mathfrak{F}_s ^J(\Sigma) $ such that (recalling the definition of the $H^\sigma_{\rm dist}$ seminorm \eqref{eq: H_dist seminorm def})
\begin{align*}
     \limsup_{s \to 1} & \hspace{0.03cm} (1-s)^2 [u_\Sigma]^2_{H^{\frac{1+s}{2}}_{\rm dist}(M) } \\ & =  \limsup_{s \to 1} \hspace{0.03cm} (1-s)^2 \iint_{M  \times M } \frac{|u_\Sigma(x)-u_\Sigma(y)|^2}{\dist(x,y)^{n+1+s}} \,dV(x)dV(y) \le \frac{2\pi\omega_{n-1}}{n} \mathcal{H}^{n-2}(\Sigma  ) . 
\end{align*}

\begin{rem}\label{rem: limsup bdd domain}
    If we were to have a regular domain $\Omega$ in place of $M$ (see Remark \ref{rem: localized stuff}), it would still be sufficient to show this property only for $\Omega \equiv M$. Indeed, once this has been proved, for a general $\Omega$ we could bound 
\begin{equation*}
    [u_\Sigma]^2_{H^{\frac{1+s}{2}}_{\rm dist}(\Omega)} \le [u_\Sigma]^2_{H^{\frac{1+s}{2}}_{\rm dist}(M)} - [u_\Sigma]^2_{H^{\frac{1+s}{2}}_{\rm dist}(M\setminus \overline{\Omega})} . 
\end{equation*}
Then, by \eqref{eq: comparability limit dist space} and the $\Gamma$-liminf inequality we would have
\begin{equation*}
    \liminf_{s \to 1} \hspace{0.03cm} (1-s)^2 [u_\Sigma]^2_{H^{\frac{1+s}{2}}_{\rm dist}(M\setminus \overline{\Omega})} =  \liminf_{s \to 1} \hspace{0.03cm} (1-s)^2 [u_\Sigma]^2_{H^{\frac{1+s}{2}}(M\setminus \overline{\Omega})} \ge \frac{2\pi\omega_{n-1}}{n} \mathcal{H}^{n-2}(\Sigma \setminus \overline{\Omega}) , 
\end{equation*}
which would allow to conclude the desired localized form:
\begin{equation*}
    \limsup_{s \to 1} \hspace{0.03cm} (1-s)^2 [u_\Sigma]^2_{H^{\frac{1+s}{2}}_{\rm dist}(\Omega)} \le  \frac{2\pi\omega_{n-1}}{n} \Big( \mathcal{H}^{n-2}(\Sigma) -  \mathcal{H}^{n-2}(\Sigma \setminus \overline{\Omega}) \Big) =  \frac{2\pi\omega_{n-1}}{n} \mathcal{H}^{n-2}(\Sigma \cap \Omega) , 
\end{equation*}

\end{rem}

\medskip  
\noindent \ding{183} (Construction of $u_\Sigma$) The construction of such $u_\Sigma \in \mathfrak{F}_s ^J(\Sigma) $ is identical to the one in \cite[Proposition 3.1]{codim2-frac}, where the authors build such a map in the case of a smooth $\Sigma \subset \R^n$. Indeed, in the proof of \cite[Proposition 3.1]{codim2-frac}, the only properties used are that $\Sigma$ is smooth and that $\Sigma$ is the boundary of a smooth (embedded) surface $N$, so that $\Sigma$ has a trivial normal bundle. We briefly recall the construction in words here. 

Since the normal bundle of $\Sigma$ is trivial, for $\delta>0$ sufficiently small, via the normal exponential map we can identify a tubular neighborhood $\mathcal{T}_\delta$ of $\Sigma$ with $\Sigma\times (-\delta,\delta)^2$, where $ \{\sigma\} \times (-\delta,\delta)^2$ represents a normal square at $\sigma \in \Sigma$.

On the normal square we choose a smooth vortex (with cutoff) profile $u_\star \colon (-\delta,\delta)^2\to \Sp^1$ with the following properties: $u_\star(p)=p/|p|$ in a neighborhood of the origin, $u_\star$ is constantly equal to $(1,0)$ in a neighborhood of the three sides of $\partial(-\delta,\delta)^2$ outside $N$, and $u_\star$ performs one full $2\pi$ turn around $\Sp^1$ along the remaining side (the one intersecting $N$), so that its degree is $1$. We then define $u_\Sigma$ on $\mathcal{T}_\delta$ by inserting $u_\star$ fiberwise on $\Sigma$; namely, $u_\Sigma$ depends only on the normal coordinates and restricts to a degree-one vortex in some small disk on each normal fiber.

Next, we extend $u_\Sigma$ across a neighborhood of $N$ by introducing a branch cut along $N$, in such a way that the phase jumps by $2\pi$ when crossing (a small normal neighborhood of) $N$. Finally, we set $u_\Sigma\equiv (1,0)$ outside a compact neighborhood of $N$. By construction, the junctions can be arranged to be $C^1$ away from $\Sigma$, and the resulting map links $\Sigma$ with degree one. By construction $\star Ju_\Sigma = \pi \Sigma$, hence $u_\Sigma \in \mathfrak{F}_s ^J(\Sigma) $.

\medskip 
\noindent \ding{184} (Completion of the proof) Split the domain of the integral 
\begin{equation*}
    \iint_{M  \times M} \frac{|u_\Sigma(x)-u_\Sigma(y)|^2}{\dist(x,y)^{n+1+s}} \,dV(x)dV(y) 
\end{equation*}
in 
\begin{equation*}
    M\times M =  \big( \mathcal{T}_\delta \times \mathcal{T}_\delta \big) \cup \big(  \mathcal{T}_\delta^c\times\mathcal{T}_\delta^c \big) \cup \big( \mathcal{T}_\delta^c \times \mathcal{T}_\delta \big)  \cup \big(  \mathcal{T}_\delta \times \mathcal{T}_\delta^c \big) . 
\end{equation*}
Since $u_\Sigma$ is $C^1$ on $\mathcal{T}_\delta^c$, using $|u_\Sigma(x)-u_\Sigma(y)| \le C \dist(x,y) $ yields
\begin{align*}
    \iint_{\mathcal{T}_\delta^c\times\mathcal{T}_\delta^c}  
\frac{|u_\Sigma(x)-u_\Sigma(y)|^2}{\dist(x,y)^{n+1+s}} &  \,dV(x)dV(y) 
\\ & \lesssim  \iint_{\mathcal{T}_\delta^c\times\mathcal{T}_\delta^c}
\frac{1}{\dist(x,y)^{n-1+s}}\,dV(x)dV(y) \lesssim \int_0^{{\rm diam}(M)} r^{-s} \, dr = \frac{C}{1-s}, 
\end{align*}
hence $(1-s)^2$ times this term tends to $0$ as $s \to 1$. The mixed term $\mathcal{T}_\delta^c \times \mathcal{T}_\delta $ (and the same for the other mixed term) can be dealt with by splitting 
\begin{equation*}
    \mathcal{T}_\delta^c \times \mathcal{T}_\delta = \mathcal{T}_\delta^c \times \mathcal{T}_{\delta/2} \cup \mathcal{T}_\delta^c \times (\mathcal{T}_\delta  \setminus \mathcal{T}_{\delta/2}) . 
\end{equation*}
On $\mathcal{T}_\delta^c \times \mathcal{T}_{\delta/2}$, ${\rm dist}(x,y) \ge c(\delta)>0$, so the kernel is bounded away from the diagonal, while for $(x,y) \in \mathcal{T}_\delta^c \times (\mathcal{T}_\delta  \setminus \mathcal{T}_{\delta/2})$ the map $u_\Sigma$ is smooth in these sets and thus satisfies $|u_\Sigma(x)-u_\Sigma(y)| \le C \dist(x,y) $ again. In both cases the contribution is $O(1)$ and $O((1-s)^{-1})$ respectively, hence it is killed by $(1-s)^2$ as $s \to 1$.

Therefore, for every $\delta$ small, we have 
\begin{equation*}
\limsup_{s \to 1} \hspace{0.03cm} (1-s)^2 [u_\Sigma]^2_{H^{\frac{1+s}{2}}_{\rm dist}(M)} 
= \limsup_{s \to 1} \hspace{0.03cm} (1-s)^2
\iint_{\mathcal{T}_\delta \times \mathcal{T}_\delta }\frac{|u_\Sigma(x)-u_\Sigma(y)|^2}{\dist(x,y)^{n+1+s}}\,dV(x)dV(y).
\end{equation*}
From here, proceeding exactly as the proof of \cite[eq. (4.45)]{codim2-frac} gives 
\begin{equation*}
    \limsup_{s \to 1} \hspace{0.03cm} (1-s)^2 [u_\Sigma]^2_{H^{\frac{1+s}{2}}_{\rm dist}(M)} \le \frac{2\pi\omega_{n-1}}{n} \mathcal{H}^{n-2}(\Sigma). 
\end{equation*}
Indeed, working in the tubular neighborhood $\mathcal{T}_\delta \simeq \Sigma \times (-\delta,\delta)^2$, the distance and volume form are uniformly comparable to their Euclidean counterparts for $\delta$ small; therefore, the computation leading to \cite[eq. (4.45)]{codim2-frac} applies verbatim (up to multiplicative errors that tend to $1$ as $\delta \to 0$). This completes the proof. 
 
\begin{rem}\label{rem: Gamma-limusp secondo modo}
In the proof of the $\Gamma$-limsup above, we invoked a deep density theorem of Almgren--Browder--Caldini--De~Lellis, Theorem~\ref{thm: Caldini approx}, in order to reduce the $\Gamma$-limsup construction to the case of smooth, embedded boundaries of unit density. 

We could have avoided using this result by following a more classical, albeit substantially longer, route based on a polyhedral approximation. Starting from an integral $(n-2)$-boundary $\Sigma = \partial T$, a Federer--Fleming type density yields a sequence of integral polyhedral $(n-2)$-boundaries $P_k$ converging to $\Sigma$ in flat norm and satisfying $\limsup_{k\to\infty} \M(P_k) \le \M(\Sigma)$.  On this polyhedral class, the existence of a competitor with prescribed Jacobian equal to $P_k$ can be obtained from \cite[Theorem 5.10]{ABO2-singularities}; see also \cite[Theorem 5.8]{p-energy}.

Since $P_k$ is polyhedral, it decomposes into finitely many top-dimensional $(n-2)$-dimensional faces together with a lower-dimensional skeleton. The contribution to the energy away from the top-dimensional faces can be estimated with the strategy developed for the $p$-energy in \cite{p-energy}, with $p=1+s \to 2$, together with the sharp version of the Gagliardo--Nirenberg inequality in \cite{Maz-Shap02}. Precisely, \cite[Remark 1]{Maz-Shap02} with $p=1+s$ and $\theta=(1+s)/2 $ gives
\begin{equation*}
    (1-s) [u]^2_{H^{\frac{1+s}{2}}} \le C \|\nabla u\|^{1+s}_{L^{1+s}} \|u\|_{L^\infty}^{1-s} =  C \|\nabla u\|^{1+s}_{L^{1+s}} . 
\end{equation*}
This, with the proof of the $\Gamma$-limsup in \cite{p-energy} for the $p$-energy, yields bounds on the off-face interactions (including mixed terms between neighborhoods of distinct faces) that become negligible after multiplication by $(1-s)^2$ as $s \to  1 $. 

On the other hand, on each $(n-2)$-dimensional face, one may work in coordinates and apply verbatim the Euclidean computations in \cite{codim2-frac}. The only additional input is to replace the Euclidean kernel by the distance kernel used here, which is achieved via the energy comparability estimates \eqref{eq: comparability limit dist space}. Consequently, each top-dimensional face contributes the correct leading-order constant to the rescaled energy, while the error due to the passage from the manifold distance and volume element to the flat case remains negligible in the $s \to 1$ limit. 

Putting together these steps, one obtains an alternative proof of the $\Gamma$-limsup inequality. However, this approach requires considerably more details (and notation), since one must track interactions across the polyhedral skeleta, and verify that all remainder terms are $o((1-s)^{-2})$ after rescaling. For this reason, we opted for the shorter argument based on Theorem~\ref{thm: Caldini approx}.
\end{rem}

\appendix

\section*{Appendix}\label{sec:appendix}
\addcontentsline{toc}{section}{Appendix}
\setcounter{section}{1}
\setcounter{teo}{0}
\setcounter{equation}{0}

\subsection{The distance kernel and comparability with Euclidean seminorms}\label{app1: kernel and comparability}

Throughout this subsection, it is convenient to allow $M$ to be a general complete, not necessarily compact, $n$-dimensional Riemannian manifold.

\begin{defn}[Uniformly flat ball]\label{def: unif flat ball}
 Let $p \in M$ and $R>0$. We say that the ball $\B_R(p)$ is \emph{uniformly flat with parametrization $\varphi_p$} if there exist an open neighborhood $\U_p$ of $p$ and a diffeomorphism $\varphi_p \colon B_R(0) \to \U_p$ such that
\[
\B_{\frac{9}{10}R}(p) \subset \U_p \subset \B_{\frac{11}{10}R}(p),
\qquad
\varphi_p(0)=p,
\]
and, in the coordinates induced by $\varphi_p^{-1}$, the metric coefficients $g_{ij}(x)$ satisfy
\begin{equation*}
    \tfrac{99}{100} |v|^2 \le g_{ij}(x)v^i v^j \le \tfrac{101}{100} |v|^2 ,
  \quad \forall \,  x \in B_R(0),\ v\in\mathbb{R}^n,
\end{equation*}
and 
\begin{equation*}
     R \abs*{ \nabla g_{ij}(x)} \le \tfrac{1}{100} , 
  \quad \forall \, x \in B_R(0),\ i,j\in\{1,\dots,n\}.
\end{equation*}
\end{defn}

 Observe that for every closed Riemannian manifold $M$ there exists $R_0=R_0(M)>0$ such that every ball $\B_R(p)$ with $R\le R_0$ is uniformly flat, with parametrization given by normal coordinates at $p$.

\begin{lemma}\label{lem: estimate k hat}
    Let $\sigma \in (1/2,1)$, $p\in M$, and assume that $\B_R(p) $ is uniformly flat with parametrization $\varphi \colon B_R(0) \to \U_p$. Denote $K(x,y): = \K_\sigma(\varphi(x), \varphi(y))$. Given $x\in B_R(0)$, let
$A(x)$ denote the positive symmetric square root of the matrix $g_{ij}(x)$ and, for $x,z\in B_{R/2}(0)$, define
\begin{equation}\label{eq: appendix kdef}
    k(x, z): = K(x,x+z) \qquad \mbox{and}\quad  \widehat k(x,z) := k(x,z) -  \frac{1}{|A(x)z|^{n+2\sigma}}.
\end{equation}
Then, there is a dimensional $C_n>0$ such that
\begin{equation}\label{remaining0}
\big|\widehat k(x,z)\big| \le  R^{-1}\frac{C_n}{|z|^{n+2\sigma -1}}\quad \mbox{ for all } x,z \in  B_{R/4}(0) . 
\end{equation}
\end{lemma}

\begin{proof}
    The proof is identical to that of \cite[eq. (18)]{CFSfrac} but tracking the dependence of the constants on $\sigma$. Since the statement is scaling invariant, we may assume $R=1$. 

    \vsp
\textbf{Claim.}  It is sufficient to prove the case when $M$ is replaced by $M=(\R^n,g)$, $p=0$, $\varphi = {\rm id}$, and $g_{ij}$ satisfying 
    \begin{equation}\label{eq: g assump Rn}
    \frac{1}{2} |v|^2 \le g_{ij}(x) v^i v^j \le 2 |v|^2 , \quad \abs*{D g_{ij}(x)} \le 1, \quad \forall  x\in \R^n . 
\end{equation}

\vsp\noindent
    Indeed, assume this case is settled. In the general case, fix a radial nonincreasing cutoff $\eta\in C_c^\infty(B_1)$ such that $\eta\equiv 1$ on $B_{7/8}$, $|\nabla \eta| \le 10$, and consider $\R^n$ with the extended metric 
\[
 g'_{ij} := g_{ij}\,\eta + \delta_{ij}(1-\eta).
\]
Clearly, by construction, $g'$ satisfies \eqref{eq: g assump Rn}. Then, $\widehat k '(x,z)$ (i.e., the one as in \eqref{eq: appendix kdef} relative to $(\R^n, g')$) satisfies \eqref{remaining0}. 

Moreover, the manifolds $(M,g)$ and $(\R^n,  g' )$ satisfy the assumptions of \cite[Lemma 2.18]{CFSfrac} with $M'=(\R^n,  g' )$ and $\varphi'={\rm id}$. For $x,y\in B_1$, recall 
\begin{equation}\label{wjhtihwi9h1}
  K(x,y) = \K_{\sigma}(\varphi(x),\varphi(y))
  = \alpha_{n,\sigma}^{-1}\frac{\sigma}{\Gamma(1-\sigma)}\int_0^\infty H_M(\varphi(x),\varphi(y),t)\,\frac{dt}{t^{1+\sigma}}, 
\end{equation}
and set
\[
  K'(x,y) := \alpha_{n, \sigma}^{-1} \frac{\sigma}{\Gamma(1-\sigma)}\int_0^\infty H_{M'}(\varphi'(x),\varphi'(y),t)\,\frac{dt}{t^{1+\sigma}} . 
\]
By \cite[Lemma 2.18]{CFSfrac}, for all $x,y\in B_{1/2}$, we have
\[
\begin{aligned}
\big|
       \big(K-K'\big)(x,y)\big|
&\le\alpha_{n, \sigma}^{-1} \frac{\sigma}{\Gamma(1-\sigma)} \int_0^\infty 
   \big| H_M(\varphi(x),\varphi(y),t) - H_{M'}(\varphi'(x),\varphi'(y),t) \big|\,\frac{dt}{t^{1+\sigma}} \\
&\le C \int_0^\infty e^{-c/t}\,\frac{dt}{t^{1+\sigma}} \\ & \le C . 
\end{aligned}
\]
Hence, for all $x,z\in B_{1/4}$, by the triangle inequality 
\begin{align*}
    \big|\widehat k(x,z)\big| \le \big| (K-K'\big)(x,x+z)\big| + \big| \widehat k '(x,z) \big| \le C + \frac{C}{|z|^{n+2\sigma-1}} \le \frac{\widetilde C}{|z|^{n+2\sigma-1}} ,  
\end{align*}
and the claim is proved. Thus, we have reduced to proving the desired estimate in the model situation $M=(\R^n,g)$, $p=0$, $\varphi={\rm id}$, and $g$ such that \eqref{eq: g assump Rn} holds. 

Define $h(z,x,t)$ by the identity 
\begin{equation*}
    H_M(x ,y,t) = \frac{1}{t^{n/2}} h\bigg( \frac{A(x)(y-x)}{\sqrt{t}} , x , t \bigg) . 
\end{equation*}
Let also $h_\circ (z,x,t) := (4\pi)^{-n/2}e^{-|z|^2/4}$. By \cite[Proposition 2.19]{CFSfrac} we have the estimate 
\begin{equation*}
    \big|(h-h_\circ)(z,x,t) \big| \le C \min\{1, \sqrt{t}\} e^{-c|z|^2} , \qquad \forall \, (x,z,t) \in \R^n \times \R^n \times (0,\infty) ,
\end{equation*}
where $c,C>0$ are dimensional constants. 

Recalling \eqref{wjhtihwi9h1} and the fact that $\varphi={\rm id}$, for $x, z \in B_{1/4}$ we have
\begin{equation*}
  k(x,z) := K(x,x+z)
  = \alpha_{n,\sigma}^{-1} \frac{\sigma}{\Gamma(1-\sigma)} \int_0^\infty h\bigg(\frac{A(x)z}{\sqrt t},x,t\bigg)\,
      \frac{dt}{t^{n/2+1+\sigma}}, 
\end{equation*}
which gives
\begin{align*}
\big| \widehat k(x,z) \big| 
&= \bigg| k(x,z)-\frac{1}{|A(x)z|^{n+2\sigma}} \bigg| \nonumber \\
&\le \alpha_{n,\sigma}^{-1} \frac{\sigma}{\Gamma(1-\sigma)} \int_0^\infty 
   \bigg| (h - h_\circ) \bigg(\frac{A(x)z}{\sqrt t},x,t\bigg) \bigg|
          \,
   \frac{dt}{t^{n/2+1+\sigma}} \nonumber \\ & \le C \int_0^\infty   e^{-c|A(x)z|^2/t} \frac{dt}{t^{n/2+1/2+\sigma}} \nonumber \\ & \le  C \int_0^\infty   e^{-c|z|^2/t} \frac{dt}{t^{n/2+1/2+\sigma}} \nonumber \\ & = \frac{C}{|z|^{n+2\sigma-1}} , 
\end{align*}
where we have also used that $|A(x)z|\ge \tfrac1{\sqrt 2}|z|$ for all $x,z$ by assumption \eqref{eq: g assump Rn} on the metric $g$. This concludes the proof.
\end{proof}

\begin{lemma}\label{lem: estimate K distant points}
    Let $\sigma \in (1/2,1)$, $p\in M$, and assume that $\B_R(p) $ is uniformly flat with parametrization $\varphi \colon B_R(0) \to \U_p$. Then, for all $x\in B_{R/4}(0)$ and $q\in M\setminus \varphi(B_R(0))$ there holds 
    \begin{equation*}
        \big| \K_\sigma(\varphi(x), q) \big| \le \frac{C_n}{R^{n+2\sigma}} , 
    \end{equation*}
    for some $C_n>0$ dimensional constant. 
\end{lemma}
\begin{proof}
    The proof is identical to that of \cite[eq. (20)]{CFSfrac} but tracking the dependence of the constants on $\sigma$, with our normalization of $\K_\sigma$ (see Remark \ref{rem: normalization K}). 

    Since the statement is scaling invariant, with no loss of generality, assume $R=1$. Under our hypothesis, by \cite[Lemma 2.16]{CFSfrac} we have the estimate 
    \begin{equation*}
        \big| H_M(\varphi(x), q, t)\big| \le C e^{-c/t} , \quad\mbox{for all } x \in B_{1/4}(0),\,  q\in M\setminus \varphi(B_1(0)), \, t>0 . 
    \end{equation*}
Using this inequality in the definition of $\K_\sigma$ gives
\begin{equation*}
  \big| \K_\sigma(\varphi(x), q) \big|
  \le \alpha_{n,\sigma}^{-1} \frac{\sigma}{\Gamma (1-\sigma )}
     \int_0^{\infty} \big| H_M(\varphi(x), q,t) \big| \,\frac{dt}{t^{1+\sigma}} \le C \int_0^{\infty}  e^{-c/t}  \,\frac{dt}{t^{1+\sigma}} \le C  , 
  \end{equation*}
  as desired. 
\end{proof}

\begin{prop}\label{prop: comparison kernel sharp}
    Let $M$ be a closed manifold and $\sigma \in (1/2,1)$. Then 
    \begin{equation*}
        \abs*{\K_\sigma(p,q) - \frac{1}{\dist(p,q)^{n+2\sigma}}} \le  \frac{C_M}{\dist(p,q)^{n+2\sigma-1}} ,  
    \end{equation*}
    for some $C_M>0$ depending only on $M$ and not on $\sigma$. 
\end{prop}

\begin{proof} This follows similarly as in \cite[Proposition 4.9]{Enric}, using Lemma \ref{lem: estimate K distant points} and Lemma \ref{lem: estimate k hat} to track precisely the dependence of the constants as $\sigma$ is close to $1$ (or $s=2\sigma$ close to $2$, in the notation of \cite{Enric}). 

Fix $R_{0}>0$ depending only on $M$ such that every ball $\{\B_{R_0}(p)\}_{p\in M} \subset M$ is uniformly flat, in the sense of Definition \ref{def: unif flat ball}. We distinguish two cases.

If $\dist(p,q)\le R_{0}/4$, we work in normal coordinates $\varphi_p$ at $p$ with $0=\varphi_{p}^{-1}(p)$ and $y=\varphi_{p}^{-1}(q)$, so that $A(0)=\delta_{ij}$ and $|y|=\dist(p,q)$. By Lemma \ref{lem: estimate k hat}, we obtain
\begin{equation*}
        \abs*{\K_\sigma(p,q) - \frac{1}{\dist(p,q)^{n+2\sigma}}} \le  \frac{C}{\dist(p,q)^{n+2\sigma-1}} .  
    \end{equation*}

On the other hand, if $\dist(p,q)\ge R_{0}/4$, then
\[
\frac{1}{\dist(p,q)^{n+2\sigma}}
  \le \frac{4/R_0}{\dist(p,q)^{n+2\sigma-1}},
\]
and Lemma \ref{lem: estimate K distant points} with $R=R_{0}/4$ yields
$|\K_{\sigma}(p,q)|\le C/R_0^{n+2\sigma}$.  Using also
$\dist(p,q)\le\diam(M)$, we conclude
\begin{align*}
    \abs*{\K_\sigma(p,q) - \frac{1}{\dist(p,q)^{n+2\sigma}}}  \le  \frac{C}{R_0^{n+2\sigma}} +  \frac{4/R_0}{\dist(p,q)^{n+2\sigma-1}} 
       \le \frac{C}{\dist(p,q)^{n+2\sigma-1}} ,
\end{align*}
also in this case. Combining the two regimes, we get the desired estimate for all $p,q \in M$.
\end{proof}

For $\sigma \in (0,1)$ and an open set $\Omega \subset M$ define
\begin{equation}\label{eq: H_dist seminorm def}
    [u]_{H^{\sigma}_{\rm dist}(\Omega)}^2
  := \iint_{\Omega\times \Omega} \frac{|u(x)-u(y)|^2}{\dist(x,y)^{n+2\sigma}} \,dV(x) dV(y) .
\end{equation}

\begin{lemma}\label{lem: equal distance kernel limit}
For $\sigma \in (1/2,1)$ and $\Omega \subset M $ open, the spaces $H^{\sigma}(\Omega)$ and $H^{\sigma}_{\rm dist}(\Omega)$ coincide, and the associated seminorms are equivalent. Moreover, for $ u\in H^{\sigma}(\Omega) \cap L^\infty(\Omega)$ and for every $\delta>0$,
    \begin{equation*}
        \big| [u]_{H^{\sigma}(\Omega)}^2 - [u]_{H^{\sigma}_{\rm dist}(\Omega)}^2 \big| \le C_M \big( \delta \min \Big\{ [u]_{H^{\sigma}_{\rm dist} (\Omega)}^2 , [u]_{H^{\sigma}(\Omega)}^2 \Big \} + C_\delta \|u\|^2_{L^\infty(\Omega)} \big) . 
    \end{equation*}
In particular, for every $ u \in L^\infty(\Omega)$, 
\begin{equation}\label{eq: comparability limit dist space}
    \limsup_{\sigma \to 1}  \, (1-\sigma)^2 [u]_{H^{\sigma}(\Omega)}^2 =   \limsup_{\sigma \to 1}  \, (1-\sigma)^2  [u]_{H^{\sigma}_{\rm dist}(\Omega)}^2 . 
\end{equation}
\end{lemma}
\begin{proof}
    The proof is identical to that of \cite[Lemma 4.10]{Enric}, using our Proposition \ref{prop: comparison kernel sharp} in place of \cite[Proposition 4.9]{Enric}. 
\end{proof}

\subsection{Some properties of circle-valued $\sigma$-harmonic maps}

\begin{prop}\label{prop: compactness min} Let $\sigma\in (1/2,1)$, $\Omega \subset \R^n$ be a bounded open set and $ \{u_k\}_k \subset H^\sigma(\Omega ; \Sp^1)$ be a sequence of minimizing (resp. null-Jacobian minimizing) $\sigma$-harmonic maps in $\Omega$ with 
 \begin{equation*}
       \sup_k \E_\sigma(u_k, \Omega) <\infty \s \mbox{and} \s u_k \to u \, \mbox{ a.e. in } \Omega . 
\end{equation*}
Then $u_k \rightharpoonup u $ in ${H}^\sigma (\Omega; \Sp^1)$ and $u_k \to u$ strongly in ${H}^\sigma(\Omega'; \Sp^1)$ for every open $\Omega'$ such that $\Omega' \Subset \Omega$. Moreover, $u$ is a minimizing (resp. null-Jacobian minimizing) $\sigma$-harmonic map in $\Omega$.
\end{prop}
\begin{proof}
   For minimizing $\sigma$-harmonic maps, the compactness and convergence directly follow from \cite[Theorem 7.1]{Partialreg}; in fact, they hold more generally for stationary $\sigma$-harmonic maps. We just have to justify that the limit $u$ is minimizing for $\sigma\in (1/2,1)$, which is the range left open in \cite{Partialreg}; see \cite[Remark 7.4]{Partialreg}. 

To this end, fix $v\in {H}^\sigma(\Omega ; \Sp^1)$ such that $K:=\supp(v-u)\Subset \Omega$, and we have to prove that
\begin{equation}\label{claim:minimality}
    \E_\sigma(u,\Omega)\leq \E_\sigma(v,\Omega).
\end{equation}

Let $\rho\in C^\infty _c(\Omega;[0,1])$ be such that $\rho=1$ in $K$, and let us set $w_k:=\rho v +(1-\rho)u_k$. Then $w_k\to \rho v + (1-\rho)u=v$ strongly in $H^\sigma$ on every compact subset of $\Omega$, since $u_k\to u$ and $\rho=1$ where $u$ and $v$ do not coincide. Moreover, the map $w_k$ coincides with $u_k$ outside the support of $\rho$, which is a compact subset of $\Omega$, but it might not take values into $\Sp^1$ on $ \{0<\rho<1\}$.

To fix this, we exploit the projection averaging technique of \cite{HL-CPAM} (see also \cite[Chapter~10]{Bre-Mi}). Let $D_1 \subset \R^2$ be the unit disk. For every $a\in D_1$ we consider the radial projection $\pi_a \colon D_1\setminus \{a\}\to \Sp^1$ from $a$ onto $\Sp^1$, and we claim that for every $k$ there exists $a_k \in D_{1/8}$ such that $v_k:=\pi_{a_k} (w_k ) $ satisfy
\begin{equation}\label{claim:energy_vk_to_v}
\liminf_{k\to \infty} \E_\sigma(v_k,\Omega')\leq \E_\sigma(v,\Omega)\qquad \text{for every }\Omega' \Subset\Omega.
\end{equation}

Once this is established, by minimality of $u_k$ we deduce that
$$\E_\sigma(u,\Omega')\leq \liminf_{k\to \infty} \E_\sigma(u_k,\Omega') \leq \liminf_{k\to \infty} \E_\sigma(v_k,\Omega')\leq \E_\sigma(v,\Omega),$$
and letting $\Omega'\nearrow \Omega$ we get (\ref{claim:minimality}).

So we now prove (\ref{claim:energy_vk_to_v}). As in \cite[Chapter~10]{Bre-Mi}, consider a smooth function $\phi \colon [0,1]\to [0,1]$ such that $\phi(t)=0$ for $t\in [0,1/4]$ and $\phi(t)=1$ for $t\in [1/2,1]$, and write
$$v_k:=\phi(|w_k|)v_k + (1-\phi(|w_k|))v_k.$$

Observe that for every choice of $a\in D_{1/8}$ it holds that $\phi(|w_k|)v_k \to \phi(|v|)v=v$ strongly in $H^\sigma$ on every compact subset of $\Omega$, because the map $z\mapsto \phi(|z|)\pi_a(z)$ is smooth from $D_1$ into itself for every $a\in D_{1/8}$.

Now we claim that we can choose $a_k \in D_{1/8}$ so that $(1-\phi(|w_k|))v_k\to 0$ strongly in $H^\sigma$, at least up to a subsequence. To this end, observe that $1-\phi(|w_k|)=0$ where $|w_k|\geq 1/2$ (hence in particular in $K$) and
$$|\nabla ( (1-\phi(|w_k|))v_k )|\leq C(1+|\nabla \pi_a(w_k)|)|\nabla w_k|\leq \frac{C}{|w_k-a|}|\nabla w_k|.$$

Thus, if we fix any $p\in (2\sigma,2)$ and any $\Omega''\Subset\Omega$, it holds that
\begin{align*}
\int_{D_{1/8}} da \int_{\Omega''} |\nabla ((1-\phi(|w_k|))v_k)|^{2\sigma}\,dx &\leq \int_{\Omega''\cap\{|w_k|<1/2\}}|\nabla w_k|^{2\sigma} \,dx \int_{D_{1/8}}\frac{C}{|w_k-a|^{2\sigma}}\,da\\
&\leq C\int_{\Omega''\cap\{|w_k|<1/2\}} |\nabla w_k|^{2\sigma}\,dx\\
&\leq C |\Omega''\cap\{|w_k|<1/2\}|^{1-\frac{2\sigma}{p}} \biggl(\int_{\Omega''\setminus K} |\nabla w_k|^{p}\,dx\biggr)^{\frac{2\sigma}{p}}.
\end{align*}

Finally, we observe that $|\nabla w_k|\leq |\nabla \rho||u_k-v| + |\nabla v| +|\nabla u_k|=|\nabla \rho||u_k-u| + |\nabla u| +|\nabla u_k|$ in $\Omega''\setminus K$, and hence the last integral is uniformly bounded thanks to \cite[Theorem~1.5]{GlobRegFrac}. It follows that
$$\lim_{k\to \infty}\int_{D_{1/8}} da \int_{\Omega''} |\nabla ((1-\phi(|w_k|))v_k)|^{2\sigma}\,dx=0,$$
so we can find $a_k \in D_{1/8}$ and a subsequence (not relabeled) for which $(1-\phi(|w_k|))v_k\to 0$ in $W^{1,2\sigma}(\Omega'')$, and hence also in $H^\sigma(\Omega'')$. It follows that $v_k \to v$ in $H^\sigma(\Omega'')$ for every $\Omega''\Subset\Omega$. Hence, if $\Omega'\Subset\Omega''\Subset\Omega$, we find that (along this subsequence)
\begin{align*}
\limsup_{k\to \infty} \E_\sigma(v_k,\Omega') &\leq \lim_{k\to \infty} [v_k]_{H^\sigma(\Omega'')}^2 + 2\iint_{\Omega'\times (\R^n\setminus\Omega'')} \frac{|v_k(y)-v_k(x)|^2}{|y-x|^{n+2\sigma}}\\
&= [v]_{H^\sigma(\Omega'')}^2 + 2\iint_{\Omega'\times (\R^n\setminus\Omega'')} \frac{|v(y)-v(x)|^2}{|y-x|^{n+2\sigma}}\\
&\leq \E_\sigma(v,\Omega'').
\end{align*}
This implies (\ref{claim:energy_vk_to_v}) and concludes the proof for minimizing $\sigma$-harmonic maps. 

The null-Jacobian minimizing case is proved in the same way, working with phases instead of the maps themselves, so we do not need to use projections (which might create Jacobian). The phases are well defined since the Jacobians vanish, and are bounded in $W^{1,p}$ for every $p < 2$ by \cite[Theorem~1.5]{GlobRegFrac}. We omit the details.
\end{proof}

\begin{rem}
    A similar property in the case of minimizing intrinsic $\sigma$-harmonic maps has recently been proved in \cite[Theorem 1.6]{IntrinsicFHM}.
\end{rem}

\begin{lemma}\label{lem: dim reduction minimality}
Let $\sigma\in (0,1)$, $n \ge 2$, and let $u \in H^\sigma_{\rm loc}(\R^n;\Sp^1)$ be a $0$-homogeneous $k$-symmetric map. Namely, up to a rotation, there exists a $0$-homogeneous map $v \in H^\sigma_{\rm loc}(\R^{n-k};\Sp^1)$ such that
\[
u(x)=v(x_1,\dots,x_{n-k}).
\]
If $u$ is a minimizing (resp. null-Jacobian minimizing) $\sigma$-harmonic map in $\R^n$, then $v$ is a minimizing (resp. null-Jacobian minimizing) $\sigma$-harmonic map in $\R^{n-k}$. 
\end{lemma}

\begin{proof} The proof for the two cases of free minimizing maps and null-Jacobian minimizing maps is identical; we just prove the statement for minimizing $\sigma$-harmonic maps. 

The proof relies on a dimensional reduction argument, similar to Step~2 in the proof of \cite[Lemma~7.13]{Partialreg}. It is enough to treat the case $k=1$, since the general statement follows by iterating the same argument until no direction of translation invariance is left. We write points of $\R^n$ as $(x',x_n)\in \R^{n-1}\times \R$, so that $u(x',x_n)=v(x')$, and we show that $v$ is minimizing in~$\R^{n-1}$.

Since $v$ is $0$-homogeneous, by scaling it is enough to prove that $v$ is minimizing in $B_1'\subset \R^{n-1}$. Let therefore $w\in H^\sigma_{\mathrm{loc}}(\R^{n-1};\Sp^1)$ be a competitor for $v$ in $B_1'$, namely such that $\operatorname{spt}(w-v)\Subset B_1'$. We have to prove that
\begin{equation}\label{eq: goal minimality v}
\mathcal E_\sigma(v,B_1')\leq \mathcal E_\sigma(w,B_1').
\end{equation}

For every $R>1$, let
\[
Q_R:=B_1'\times (-R,R)
\]
and
\[
C_R^\pm:=\Bigl\{(x',x_n): \pm x_n\in (R,R+1), \ |x'|\leq R+1-|x_n|\Bigr\}.
\]
These are two conical caps of unit slope. We then define
\[
\widetilde w_R(x',x_n):=
\begin{cases}
w(x') & \text{in } Q_R,\\[4pt]
w \bigg(\dfrac{x'}{R+1-|x_n|}\bigg) & \text{in } C_R^+\cup C_R^-,\\[8pt]
v(x') & \text{elsewhere}.
\end{cases}
\]

First, we prove that $\widetilde{w}_R \in H^{\sigma}_{\rm loc}(\R^n;\Sp^1)$. By~\cite[Lemma~15.30]{Bre-Mi} it is enough to show that
    \[
        w_{\pm}:=w\bigg( \frac{x'}{R+1 \mp x_n} \bigg) \in H^{\sigma}(C_R^{\pm};\Sp^1).
    \]
Letting $K:=\{ (x',t) : t \in (0,1), |x'| <t \}$, we have
    \begin{equation} \label{eq:somewhat-long1}
       [w_+]^2_{H^\sigma(C_R^+)}:=\iint_{C_R^+ \times C_R^+} \frac{|w_+(x)-w_+(y)|^2}{|x-y|^{n+2\sigma}} \, dxdy
       =
       \iint_{K \times K} \frac{\big| w\big( \tfrac{x'}{r} \big) - w\big( \tfrac{y'}{t} \big) \big|^2}{ (|x'-y'|^2+|r-t|^2)^{\frac{n+2\sigma}{2}}} dx'dy' dr dt.
    \end{equation}
By substitution $x'=rx''$ and $y'=ty''$, the last integral equals
    \begin{equation} \label{eq:somewhat-long2}
        \iint_{B_{1}' \times B_{1}'} |w(x'')-w(y'')|^2 \Bigg( \int_0^1 \int_0^1  \frac{r^{n-1} t^{n-1}}{(|r x'' - t y''|^2+|r-t|^2)^{\frac{n+2\sigma}{2}}} drdt\Bigg) dx'' dy''.
    \end{equation}
From here, a standard but somewhat tedious estimate on the inner integral gives
\begin{equation} \label{eq:somewhat-long3}
    \int_0^1 \int_0^1  \frac{r^{n-1} t^{n-1}}{(|r x'' - t y''|^2+|r-t|^2)^{\frac{n+2\sigma}{2}}} drdt \le  \frac{C}{|x''-y''|^{n+2\sigma-1}}, 
\end{equation}
for some $C > 0$ whenever $|x''| \le 1$ or $|y''|\le 1$ and $n>2\sigma$. Combining~\eqref{eq:somewhat-long1},~\eqref{eq:somewhat-long2}, and~\eqref{eq:somewhat-long3}
\begin{equation} \label{eq:somewhat-long4}
    [w_+]^2_{H^\sigma(C_R^+)} \le C \iint_{B_{1}' \times B_{1}'} \frac{|w(x'')-w(y'')|^2}{|x''-y''|^{n-1+2\sigma}} dx''dy'' <+\infty, 
\end{equation}
where the last integral is finite since $w \in H^{\sigma}(B_1';\Sp^1)$. The same argument applies to $w_-$.

Moreover, by construction it holds that
\[
\operatorname{spt}(\widetilde w_R-u)\subseteq C_R^- \cup Q_R \cup C_R^+ .
\]
Since $u$ is minimizing in $\R^n$, we have
\begin{equation}\label{eq:minimality-cylinder}
\mathcal E_\sigma(u,C_R^- \cup Q_R \cup C_R^+)\leq \mathcal E_\sigma(\widetilde w_R,C_R^- \cup Q_R \cup C_R^+).
\end{equation}

Now, we claim that, as $R\to \infty$,
\begin{equation}\label{eq:reduction-limits}
\begin{aligned}
\frac{1}{2R}\mathcal E_\sigma(u,C_R^- \cup Q_R \cup C_R^+) &\to c_{n,\sigma} \mathcal E_\sigma(v,B_1'),\\[0.5ex]
\frac{1}{2R}\mathcal E_\sigma(\widetilde w_R,C_R^- \cup Q_R \cup C_R^+) &\to c_{n,\sigma} \mathcal E_\sigma(w,B_1'),
\end{aligned}
\end{equation}
where
\begin{equation*} 
c_{n,\sigma}:=\frac{\alpha_{n-1,\sigma}}{\alpha_{n,\sigma}}>0,
\end{equation*}
with $\alpha_{n,\sigma}$ defined by~\eqref{eq: alpha-def}.

Assuming this for the moment, dividing \eqref{eq:minimality-cylinder} by $2R$ and letting $R\to  \infty$ yields
\[
\mathcal E_\sigma(v,B_1')\leq \mathcal E_\sigma(w,B_1'),
\]
which is exactly \eqref{eq: goal minimality v}. Therefore, it remains to prove \eqref{eq:reduction-limits}.

We only discuss the second convergence, since the first one for $u$ is identical, and in the regions $C_R^\pm$ even simpler. We prove separately that
\begin{equation}\label{eq: conv Q_R}
\frac{1}{2R}\mathcal E_\sigma(\widetilde w_R,Q_R)\to c_{n,\sigma}\,\mathcal E_\sigma(w,B_1')
\end{equation}
and
\begin{equation}\label{eq: conv C_R}
\frac{1}{2R}\mathcal E_\sigma(\widetilde w_R,C_R^\pm)\to 0.
\end{equation}

We start from the cylindrical part in $Q_R$. By definition
\begin{align*}
\mathcal E_\sigma(\widetilde w_R,Q_R)
=
\iint_{Q_R\times Q_R}
\frac{|\widetilde w_R(x)-\widetilde w_R(y)|^2}{|x-y|^{n+2\sigma}}\,dx dy + 2 
\iint_{Q_R\times (\R^n\setminus Q_R)}
\frac{|\widetilde w_R(x)-\widetilde w_R(y)|^2}{|x-y|^{n+2\sigma}}\,dx dy.
\end{align*}
Since \(\widetilde w_R(x',x_n)=w(x')\) on \(Q_R\), the first term can be written as
\begin{align*}
\iint_{Q_R\times Q_R} & 
\frac{|\widetilde w_R(x)-\widetilde w_R(y)|^2}{|x-y|^{n+2\sigma}}\,dx dy
\\ & =
\iint_{B_1'\times B_1'}
|w(x')-w(y')|^2
\Bigg(
\int_{-R}^R \int_{-R}^R 
\frac{dx_n dy_n}
{(|x'-y'|^2+|x_n-y_n|^2)^{\frac{n+2\sigma}{2}}}
\Bigg)
dx' dy'.
\end{align*}
Then, arguing analogously as in Step~2 of the proof of \cite[Lemma~7.13]{Partialreg}, after dividing by \(2R\), applying Fubini's theorem, and performing the change of variable
\[
x_n = y_n + t|x'-y'|,
\]
we obtain, for every $x' \neq y'$, that
\[
\frac{1}{2R}\int_{-R}^R dy_n \int_{-R}^R
\frac{dx_n}{\bigl(|x'-y'|^2+|x_n-y_n|^2\bigr)^{\frac{n+2\sigma}{2}}}
\to
\frac{1}{|x'-y'|^{n-1+2\sigma}} \int_{-\infty}^{+\infty} \frac{dt}{(1+t^2)^{\frac{n+2\sigma}{2}}}.
\]
A direct computation shows that
\[
    \int_{-\infty}^{+\infty} \frac{dt}{(1+t^2)^{\frac{n+2\sigma}{2}}}  = c_{n,\sigma},
\]
where $c_{n,\sigma}$ is the constant introduced above. Consequently,
\begin{equation}\label{eq:inner-limit}
\frac{1}{2R}
\iint_{Q_R\times Q_R}
\frac{|\widetilde w_R(x)-\widetilde w_R(y)|^2}{|x-y|^{n+2\sigma}}\,dx dy
\to
c_{n,\sigma}
\iint_{B_1'\times B_1'}
\frac{|w(x')-w(y')|^2}{|x'-y'|^{n-1+2\sigma}}\,dx' dy'.
\end{equation}

We now consider the mixed term. A similar computation to Step~2 of the proof of \cite[Lemma~7.13]{Partialreg} also shows that 
\[
\frac{1}{2R}
\iint_{Q_R\times (\R^n\setminus Q_R)}
\frac{|\widetilde w_R(x)-\widetilde w_R(y)|^2}{|x-y|^{n+2\sigma}}\,dx dy
\to
c_{n,\sigma}
\iint_{B_1'\times (\R^{n-1}\setminus B_1')}
\frac{|w(x')-w(y')|^2}{|x'-y'|^{n-1+2\sigma}}\,dx' dy'.
\]
Combining this with \eqref{eq:inner-limit}, we deduce \eqref{eq: conv Q_R}.

It remains to prove~\eqref{eq: conv C_R}. We only consider $C_R^+$, since the argument for $C_R^-$ is identical. We write
\begin{equation*}
    \E_\sigma(\widetilde w_R,C_R^+) 
    =
    [w_+]^2_{H^\sigma(C_R^+)}
        + 2 \int_{C_R^+}\int_{\R^n \setminus C_R^+} \frac{|\widetilde w_R(x)-\widetilde w_R(y)|^2}{|x-y|^{n+2\sigma}} \, dx dy.
\end{equation*}
From~\eqref{eq:somewhat-long4} it follows immediately that $R^{-1} [w_+]^2_{H^\sigma(C_R^+)} \to 0$ as $R \to +\infty$.

Setting $T_R(x)=(x',R+1-x_n)$ the second term becomes (up to a factor of $2$)
    \begin{multline*}
         \int_{K}\int_{\R^n \setminus K} \frac{|\widetilde w_R(T_R(x))-\widetilde w_R(T_R(y))|^2}{|x-y|^{n+2\sigma}} \, dx dy
        \\[1ex]
        \le 
        \underbrace{\iint_{(B_2' \times (-2,2))^2} \frac{|\widetilde w_R(T_R(x))-\widetilde w_R(T_R(y))|^2}{|x-y|^{n+2\sigma}} \, dx dy}_{=:\mathrm{I}_R}
        +
        \underbrace{\int_{K}\int_{\R^n \setminus (B_2' \times (-2,2))} \frac{4}{|x-y|^{n+2\sigma}} \, dx dy.}_{=:\mathrm{II}}
    \end{multline*}

Finally, the map $\widetilde{w}_R \circ T_R$ is independent of $R$ on $B_2' \times (-2,2)$. On the other hand, we have already shown that $\widetilde{w}_R \in H^\sigma(\R^n;\Sp^1)$. It follows that $\mathrm{I}_R$ is finite and independent of $R$, and therefore $\mathrm{I}_R/R \to 0$ as $R \to +\infty$. Since $\sigma > 0$, the quantity $\mathrm{II}$ is finite as well, and hence $\mathrm{II}/R \to 0$ as $R \to +\infty$. This shows~\eqref{eq: conv C_R} and concludes the proof. 
\end{proof}

\begin{lemma}\label{lem: A in W^1r} Let $\sigma \in (1/2,1)$, and let $u \in H^\sigma(B_2; \Sp^1)$ be a $\sigma$-harmonic map in $B_2$, namely a solution of $(-\Delta)^\sigma u = A_\sigma(u) u $ in $\mathscr{D}'(B_2)$. Assume that 
\begin{equation*}
    |\nabla u| \in L^{ p }(B_2), \quad \forall p \in [1,3). 
\end{equation*}
Then $ u \in  W^{2\sigma, \tfrac{p}{2\sigma}  }  (B_1) $ and $A_\sigma(u) \in W^{1, \tfrac{p}{1+2\sigma} }   (B_1)$ for every $p \in [1,3)$.
\end{lemma}

\begin{proof}
Since $B_2 \subset \mathbb{R}^n$ is a smooth bounded domain, there exists a single linear extension operator
\[
E \colon \mathscr{D}'(B_2) \to \mathscr{D}'(\R^n) ,\qquad Eu|_{B_2}=u,
\]
which is bounded simultaneously on all the spaces $W^{\sigma,p}$ for $\sigma\ge 0$ and $p\ge 1$. We refer to \cite[Section 4]{Extension-allspaces} for the construction of such an extension operator. 

By the fractional Gagliardo--Nirenberg inequality $ u \in W^{\alpha, p}(B_2)$  for every $\alpha \in (0,1]$ and $p\ge 1$ such that $\alpha p <3$. Fix $\theta \in (\tfrac{2\sigma}{3},1)$. Observe that 
\begin{equation*}
    A_\sigma(u) = \tfrac{1}{2} \Gamma_\sigma(u,u), \quad \mbox{where } \, \Gamma_\sigma(f,g) := (-\Delta)^\sigma(fg)- f (-\Delta)^\sigma g - g(-\Delta)^\sigma f . 
\end{equation*}
Hence, by the Kato-Ponce-Vega (KPV) inequality, for example \cite[Theorem 1.2]{KPV} applied with
\begin{equation*}
    p= \frac{3 \theta }{2\sigma} , \qquad p_1=p_2= \frac{3 \theta}{\sigma}, \qquad s_1=s_2= \sigma  , 
\end{equation*}
we have 
\begin{equation*}
    \| A_\sigma(u)\|_{L^{\frac{3 \theta }{2\sigma}}} \lesssim \|(-\Delta)^{\frac{\sigma}{2}} u \|_{L^{\frac{3 \theta }{\sigma}}} ^2 \simeq \| u\|_{W^{\sigma, \frac{3\theta}{\sigma}}} ^2 < + \infty , 
\end{equation*}
since $\sigma \cdot \frac{3\theta}{\sigma} = 3\theta <3$. Since this holds for every $\theta \in (\tfrac{2\sigma}{3},1)$, this shows that $A_\sigma(u) \in L^{ \tfrac{p}{2\sigma}}  (B_1)$ for every $p \in (2\sigma,3)$. Now, as $u$ solves $(-\Delta)^\sigma u=A_\sigma(u) u$ in $B_2$, fractional Calderón--Zygmund estimates (see e.g. \cite{FracCZ}) yield $ u \in  W^{2\sigma, \tfrac{p}{2\sigma} }  (B_1) $ for every $p<3$. 

Observe that $\nabla A_\sigma(u) = \tfrac{1}{2} \Gamma_\sigma(\nabla u, u)$. Fix $\theta \in (\frac{1+2\sigma}{3}, 1)$. By the KPV inequality (i.e. \cite[Theorem 1.2]{KPV}) with
\begin{equation*}
    p = \frac{3 \theta }{1+2\sigma} , \qquad p_1 = \frac{3 \theta}{2\sigma}, \qquad p_2 = 3\theta , \qquad s_1 =2 \sigma - 1, \qquad s_2 = 1 , 
\end{equation*}
we have 
\begin{align*}
    \| \nabla A_\sigma(u)\|_{L^{\frac{3 \theta }{1+2\sigma}}} & = \| \Gamma_\sigma(\nabla u, u)\|_{L^{\frac{3 \theta }{1+2\sigma}}}  \\ & \lesssim \|(-\Delta)^{\sigma-\frac{1}{2}} (\nabla u) \|_{L^{\frac{3 \theta }{2\sigma}}}  \| (-\Delta)^{\frac{1}{2}} u\|_{L^{3\theta}}      \lesssim \| u \|_{W^{2\sigma, \frac{3 \theta }{2\sigma}}} \|u\|_{W^{1, 3\theta }} < +\infty . 
\end{align*}
Since this holds for every $\theta \in (\frac{1+2\sigma}{3}, 1)$, this concludes the proof.
\end{proof}

\subsection{Some properties of circle-valued $H^s$ maps}

\begin{lemma}\label{lem: convergence modified heat kernel} Let $0< r < \rho< {\rm inj}(M)$, and let $\eta : M \times M \to [0,1]$ be a smooth symmetric function such that $\eta = 1$ when ${\rm dist}(x,y) \le \tfrac{r}{2} $ and $\eta = 0$ when ${\rm dist}(x,y) > r $. For $t>0$ and $f : B_{2\rho}(q) \to \R$ let 
\begin{equation}\label{eq: tilde P def}
     \widetilde P_t f(x) := \frac{\int_M H (x,y,t)\eta(x,y)f(y) \, dV(y)}{\int_M H (x,y,t)\eta(x,y) \, dV(y)} ,  \qquad x\in B_\rho(q). 
\end{equation}
Then, as $t\to 0$, $\widetilde P_t f \to f$ in $W^{1,p}(B_\rho(q))$ for every $p\ge 1$ and $f\in W^{1,p}(B_{2\rho}(q))$, and in
$H^\sigma(B_\rho(q))$ for every $\sigma \in (0,1)$ and
$f\in H^\sigma(B_{2\rho}(q))$. 
\end{lemma}

\begin{proof}
We just sketch the proof. Extend \(f\) to a function $F$ defined on \(M\), continuously in the corresponding Sobolev space. For $x\in B_\rho(q)$, the numerator defining
$\widetilde P_t f(x)$ is unchanged if $f$ is replaced by $F$. Let \(P_t=e^{t\Delta}\) be the heat semigroup and set
\[
m_t(x):=\int_M H(x,y,t)\eta(x,y)\,dV(y).
\]
Since \(1-\eta\) is supported where ${\rm dist}(x,y) \ge \tfrac{r}{2}$, Gaussian estimates for \(H\) and \(\nabla_x H\) give
\begin{equation}\label{eq: m_t decay to 1}
\|1-m_t\|_{C^1(M)}\leq C e^{-c/t} ,
\end{equation}
and 
\begin{equation}\label{eq: (1-eta)F decay}
\left\|
\int_M H (  \cdot ,y,t)(\eta( \cdot,y) -1 )F(y)\,dV(y)
\right\|_{W^{1,p}(B_\rho(q))}
\leq C e^{-c/t}\|F\|_{W^{1,p}(M)}.
\end{equation}

Write $\widetilde P_t f - f$ as 
\begin{align*}
   \left(\frac{1}{m_t}-1 \right)\left( \int_M H ( \cdot ,y,t)\eta(\cdot,y)f(y) dV(y)\right) + \int_M H (\cdot,y,t)(\eta(\cdot,y)-1)F(y)\,dV(y) + P_t F-F . 
\end{align*}
By \eqref{eq: m_t decay to 1} and the boundedness of the heat semigroup in $W^{1,p}$
\begin{equation*}
    \left\| \left(\frac{1}{m_t}-1 \right)\left( \int_M H ( \cdot ,y,t)\eta(\cdot,y)f(y) dV(y)\right)\right\|_{W^{1,p}(B_\rho(q))} \le C\| f \|_{W^{1,p}(B_{2\rho}(q))} \|1-m_t\|_{C^1(M)} . 
\end{equation*}
Moreover, since \(P_t\) is strongly continuous on \(W^{1,p}(M)\), we have $P_t F \to F$ in $W^{1,p}(M)$. Letting $t \to 0$ and using \eqref{eq: m_t decay to 1} and \eqref{eq: (1-eta)F decay} the convergence in $W^{1,p}(B_\rho(q))$ follows.

For $\sigma \in (0,1)$, interpolation of the corresponding $L^2$ and $H^1$ estimates gives exponentially small bounds in $H^\sigma$.
Together with the strong continuity of $P_t$ on $H^\sigma(M)$, the
same decomposition shows the $H^\sigma$ convergence.
\end{proof}

We believe that the following result is known to experts, but it does not appear to be written anywhere in our case of a general closed ambient manifold $M$. It can be found in the literature in \cite[Theorem 3]{Bousquet} for $M=\Sp^n$ and for smooth, connected domains in $\R^n$ in \cite[Theorem 10.4']{Bre-Mi}. For its proof, we essentially follow the strategy in \cite[Theorem 10.4']{Bre-Mi} with some modifications.

\begin{teo}\label{thm: density on M} Let $(M,g)$ be a closed Riemannian manifold and $\sigma\in (1/2,1)$. Then
\begin{equation*}
\overline{C^\infty(M; \Sp^1)}^{H^\sigma(M; \Sp^1)} = \Big\{ u \in H^\sigma(M; \Sp^1) \, : \, \star Ju=0 \Big\} . 
\end{equation*}
\end{teo}

\begin{proof}
The inclusion $\subseteq$ is clear, since $\star Ju=0$ is a purely local property and this result is well-known for bounded domains in $\R^n$. See, for example, \cite[Theorem 10.4]{Bre-Mi}. 

We show the opposite inclusion $\supseteq $. Let $u\in H^\sigma(M;\Sp^1)$ with $\star Ju=0$. Since $M$ is closed, see \cite[Remark 2.10]{CFSfrac}, there exists $r>0$ with the property that $\B_{4r}(p)$ is uniformly flat (see Definition \ref{def: unif flat ball}) for every $p \in M$. By compactness, we can find a finite collection of points $p_1, \dotsc, p_m \in M $ such that the balls $\B_r(p_1), \dotsc, \B_r(p_m)$ still cover $M$. For each $j$, let $F_j$ denote the parametrization provided by Definition \ref{def: unif flat ball} for $\B_{4r}(p_j)$.

Since $\star Ju=0$, for every $j$, by the local lifting theory in Euclidean domains \cite[Theorem 8.8]{Bre-Mi} applied to $u \circ F_j $, we can write
\begin{equation}\label{eq: u local lift}
u = e^{i \varphi_j}\s \mbox{on } \B_{3r}(p_j), \qquad \varphi_j \in \big(H^\sigma + W^{1,2\sigma} \big)(\B_{3r}(p_j)) . 
\end{equation}
Here we are implicitly using Proposition \ref{prop: comparison kernel sharp} to compare (locally) the $H^\sigma$ seminorms on $M$ with the ones on $\R^n$. 

Let $\widetilde P_t $ be the smoothing operator in \eqref{eq: tilde P def}. For $\ep>0$, on $\B_{2r}(p_j)$ define the smooth maps
\begin{equation*}
    \varphi_{j,\ep} :=  \widetilde P_\ep \varphi_j \in C^{\infty}(\B_{2r}(p_j)) , \qquad u_{j, \ep} := e^{i\varphi_{j, \ep}}  \in C^{\infty}(\B_{2r}(p_j); \Sp^1) . 
\end{equation*}
By Lemma \ref{lem: convergence modified heat kernel} we have $\varphi_{j, \ep} \to \varphi_{j}$ in $\big(H^\sigma + W^{1,2\sigma} \big)(\B_{2r}(p_j))$. Moreover, by the continuity of $\varphi \mapsto e^{i \varphi}$ from $\big(H^\sigma + W^{1,2\sigma} \big)(\Omega)$ to $H^\sigma(\Omega; \Sp^1)$---see \cite[Section 15.3]{Bre-Mi}---we have that 
\begin{equation}\label{eq: conv u_i ep}
    u_{j , \ep} \to e^{i\varphi_j} = u \quad \mbox{in } H^\sigma(\B_{2r}(p_j) ; \Sp^1) . 
\end{equation}

Now we check consistency on the overlaps. That is, we claim
\begin{equation*}
    u_{j, \ep} = u_{k, \ep} \quad \mbox{on } \B_{r}(p_j) \cap \B_{r}(p_k), \mbox{ whenever } \B_{r}(p_j) \cap \B_{r}(p_k) \neq \varnothing . 
\end{equation*}
By \eqref{eq: u local lift}, whenever $\B_{r}(p_j) \cap \B_{r}(p_k) \neq \varnothing$ we must have that $\varphi_j-\varphi_k$ has values in $2\pi \Z$ in $\B_{2r}(p_j) \cap \B_{2r}(p_k)$. Moreover, by construction $\B_{2r}(p_j) \cap \B_{2r}(p_k)$ is simply connected hence, by \cite[Corollary 6.2]{Bre-Mi} there exists an integer $n_{jk} \in \Z$ such that 
\begin{equation*}
    \varphi_j  - \varphi_k = 2\pi n_{jk} \s \mbox{on } \B_{2r}(p_j) \cap \B_{2r}(p_k) . 
\end{equation*}
Since $ \widetilde P_\ep$ preserves constants and its kernel has support at distance $\le r$, we deduce that 
\begin{equation*}
 \varphi_{j, \ep} - \varphi_{k, \ep} = 2\pi n_{jk} \s \mbox{on } \B_{r}(p_j) \cap \B_{r}(p_k) .
\end{equation*}
Hence $u_{j, \ep} = u_{k, \ep}$ on $\B_{r}(p_j) \cap \B_{r}(p_k)$ as desired. 

Thus, we may define the single sequence $u_\ep \in C^{\infty}(M; \Sp^1)$ by 
\begin{equation*}
    u_{\ep}(x) := u_{j,\ep}(x) , \s \mbox{if } x\in \B_{r}(p_j) , 
\end{equation*}
and this definition is consistent. By \eqref{eq: conv u_i ep} we also have that $u_\ep \to u$ in $H^\sigma(\B_{r}(p_j); \Sp^1)$ for every $j = 1, \dotsc, m$. To pass from these localized convergences to the global convergence on $M$, we can use a standard localization argument. Let $\{\eta_j\}_{j=1}^m \subset C^\infty(M)$ be a smooth partition of unity subordinate to the covering $\{\B_r(p_j)\}_{j=1}^m$, so that
\[ 0\leq \eta_j \leq 1,\quad \operatorname{supp}\eta_j \subset \B_r(p_j), \quad\text{and}\quad \sum_{j=1}^m \eta_j \equiv 1\ \mbox{ on }M, 
\]
and write
\[
u_\ep-u=\sum_{j=1}^m \eta_j (u_\ep-u).
\]
For each fixed $j$, the function $\eta_j(u_\ep-u)$ is supported in $\B_r(p_j)$. Since the multiplication $u_\ep-u \mapsto \eta_j (u_\ep-u)$ is continuous from $H^\sigma(\B_r(p_j))$ to itself, and since the extension by zero is continuous from $H^\sigma_0(\B_r(p_j)) \to H^\sigma(M)$, the convergence $u_\ep \to u$ in $H^\sigma(\B_r(p_j) ; \Sp^1)$ implies, for every $j$, that
\[
\eta_j (u_\ep - u ) \to 0 \quad\text{in }H^\sigma(M).
\]
Summing over $j$ and using the triangle inequality gives
\[
\|u_\ep-u\|_{H^\sigma(M)}
\leq \sum_{j=1}^m \|\eta_j (u_\ep- u)\|_{H^\sigma(M)} \to 0 , \s \mbox{ as } \ep \to 0. 
\]
Thus $u_\ep \to u$ in $H^\sigma(M; \Sp^1)$ and this concludes the proof.
\end{proof}

\begin{teo}\label{teo:fractional_lifting}
Let $\sigma\in (0,1)$, $\sigma p >1 $, and $u \in W^{\sigma,p}(\Sp^1 ; \Sp^1) $ with $\deg u=0$. Let also $ \varphi $ be a continuous lift of $u$. Then there exists a constant $C=C(p)>0$ independent of $\sigma$ such that 
\begin{equation*}
    [\varphi ]_{W^{\sigma,p}(\Sp^1)}^p \le C  [u ]_{W^{\sigma,p}(\Sp^1)}^p  + \frac{C}{(1-\sigma)^{1-1/(\sigma p)}} [u]_{W^{\sigma,p}(\Sp^1)}^{p/\sigma} . 
\end{equation*}    
\end{teo}

\begin{proof} This result was proved in \cite[Theorem 1.1]{MerletLift} where the author did not track the precise dependence of the constant as $\sigma \to 1$. We sketch the proof from \cite{MerletLift} here, highlighting the two key points where we obtain an improved dependence on $\sigma$. We keep the exact same notation of \cite{MerletLift}. Let $C>0$ be a general constant depending only on $p$ and not on $\sigma$.

Denote by $[x,y]\subset \Sp^1$ the shortest arc connecting the points $x,y \in \Sp^1$. Since we are in one spatial dimension and $\sigma p>1$, by the Sobolev embedding $u$ has a continuous lift $\varphi \in W^{\sigma,p}(\Sp^1)$ and there holds
\begin{equation}\label{eq: Sob emb 1D}
    |u(x)-u(y)| \le C(1-\sigma)^{1/p} |x-y|^{\sigma-1/p} [u]_{W^{\sigma,p}([x,y])} . 
\end{equation}
  For a reference for the dependence on $\sigma$ of the constant on the right-hand side in this inequality, see, for example, the proof of \cite[Lemma 2.8]{codim2-frac}. From here, the proof goes on exactly as the one in \cite{MerletLift}, with $\hat c_0:=C (1-\sigma)^{1/p}$ in place of what is called $c_0$ in \cite{MerletLift}. 
    
    On the ``good" set 
    \begin{equation*}
        E_0 := \{(x,y) \in \Sp^1 \times \Sp^1  : a(x,y)^{1/p} \le 1/{\hat c_0}  \}, \quad \mbox{where } \,\, a(x,y):=[u]_{W^{\sigma,p}([x,y])}^p |x-y|^{\sigma p-1}  , 
    \end{equation*}
    by \eqref{eq: Sob emb 1D} we have that $u([x,y])$ is contained in a half-circle. Thus, on $E_0$ we have the pointwise comparison $|\varphi(x)-\varphi(y)| \le (\pi/2)|u(x)-u(y)| $ which gives
    \begin{equation}\label{eq: estimate good set}
        \iint_{E_0}\frac{|\varphi(x)-\varphi(y)|^p}{|x-y|^{1+\sigma p}} \, dx dy \le C [u]_{W^{\sigma,p}(\Sp^1)}^p . 
    \end{equation}

    On the complement $\Sp^1 \times \Sp^1 \setminus E_0 $ we let $k_{x,y}$ be the smallest dyadic number $2^q$ such that $[x,y]$ can be partitioned into $2^q$ subintervals, each belonging to $E_0$. Clearly $|\varphi(x)-\varphi(y)| \le (\pi/2) k_{x,y}$ on $\Sp^1 \times \Sp^1 \setminus E_0 $. Arguing as in \cite{MerletLift} we get the estimate 
    \begin{equation*}
        k_{x,y} \le C \hat c_0^{1/\sigma} a(x,y)^{1/(\sigma p)} = C (1-\sigma)^{1/(\sigma p)} a(x,y)^{1/(\sigma p)} , 
    \end{equation*}
    and the proof concludes that
    \begin{align*}
          \iint_{\Sp^1 \times \Sp^1 \setminus E_0}\frac{|\varphi(x)-\varphi(y)|^p}{|x-y|^{1+\sigma p}} \, dx dy \le C (1-\sigma)^{1/(\sigma p)} [u]_{W^{\sigma,p}(\Sp^1)}^{p/\sigma} \bigg( \int_{-1/2}^{1/2} \frac{1}{|w|^{m-1}} dw \bigg) , 
    \end{align*}
where $m:=1+1/\sigma-(1-\sigma)p$. Substituting this value of $m$ gives
\begin{equation*}
     \iint_{\Sp^1 \times \Sp^1 \setminus E_0}\frac{|\varphi(x)-\varphi(y)|^p}{|x-y|^{1+\sigma p}} \, dx dy \le C (1-\sigma)^{1/(\sigma p)-1} [u]_{W^{\sigma,p}(\Sp^1)}^{p/\sigma}
\end{equation*}
which, together with \eqref{eq: estimate good set}, gives the desired estimate.
\end{proof}

\newpage

\noindent \textbf{Acknowledgments}: The authors are grateful to Antoine Detaille for his useful correspondence regarding Theorem \ref{thm: density on M}, and to Gianmarco Caldini for detailed clarifications on~\cite[Theorem~1.1]{almgren2024optimal}.

The authors are members of the {\selectlanguage{italian}``Gruppo Nazionale per l'Analisi Matematica, la Probabilità e le loro Applicazioni''} (GNAMPA) of the {\selectlanguage{italian}``Istituto Nazionale di Alta Matematica''} (INdAM).

M.~C. gratefully acknowledges the Erwin Schrödinger International Institute (ESI) in Vienna and Bocconi University for their kind hospitality and financial support during the realization of this work.

M.~F. is funded by the European Union: the European Research Council (ERC), through StG ``MAGNETIC'', project number: 101165368. Views and opinions expressed are however those of the authors only and do not necessarily reflect those of the European Union or the European Research Council. Neither the European Union nor the granting authority can be held responsible for them.

N.~P. and M.~F. acknowledge the INdAM-GNAMPA Project {\selectlanguage{italian} \em Gamma-convergenza di funzionali geometrici non-locali}, CUP \#E5324001950001\#.

N.~P. acknowledges the MIUR Excellence Department Project awarded to the Department of Mathematics, University of Pisa, CUP I57G22000700001.

\vspace{10pt}\noindent
\textbf{Data availability}: This paper does not use any data set.

\vspace{10pt}\noindent
\textbf{Conflict of interest}: The authors declare that there is no conflict of interest.

\begingroup

            \bibliographystyle{abbrv}
         
\bibliography{references}

\endgroup

\end{document}